\documentclass[10pt,leqno]{article}
\usepackage{amsfonts}
\usepackage{amsmath, amsthm, amsfonts, amssymb,mathrsfs,color}
\usepackage{indentfirst}

\usepackage{CJK}
\usepackage{tensor}
\usepackage{hyperref}
\hypersetup{
colorlinks,
citecolor=blue,
filecolor=black,
linkcolor=blue,
urlcolor=blue
}

\usepackage{url}

\usepackage{soul}
\usepackage{enumitem}
\usepackage{a4wide}
\usepackage{geometry}
\usepackage[normalem]{ulem}
\usepackage[numbers,sort&compress]{natbib}

\newcommand\R{\mathbb R}
\newcommand\T{\mathbb T}
\newcommand\N{\mathbb N}
\newcommand\E{\mathbb E}
\newcommand\U{\mathbb U}
\newcommand\M{\mathbb M}
\newcommand\Z{\mathbb Z}
\newcommand\Q{\mathcal Q}

\newcommand\A{\mathcal A}

\newcommand\p{\mathbb P}
\newcommand\W{\mathcal W}
\newcommand\F{\mathcal F}
\newcommand\LL{\mathcal L}
\newcommand{\D}{\mathcal D}

\theoremstyle{plain}
\numberwithin{equation}{section}
\newtheorem{Theorem}{Theorem}[section]
\newtheorem{Proposition}{Proposition}[section]
\newtheorem{Lemma}{Lemma}[section]
\newtheorem{Definition}{Definition}[section]

\newtheorem{Assumption}{Assumption}

\newenvironment{ManualHypo}[1]{%
\renewcommand\theAssumption{#1}%
\Assumption
}{\endAssumption}

\newtheorem{MResult}{Main Result}

\theoremstyle{definition}

\newtheorem{Remark}{Remark}[section]

\def\p{\mathbb P}
\def\P{\mathbb P}
\def\nn{\nabla}
\def\pp{\partial}
\def\B{\mathbf B}
\def\E{\mathbb E}
\def\BB{\mathcal B}
\def\W{\mathcal W}
\def\S{\mathcal S}
\def\SS{\mathbf S}
\def\Wlip{W^{1,\infty}}
\def\hh{\widehat}
\def\ol{\overline}
\def\I{{\bf I}}
\def\A{\mathcal A}
\def\D{\mathcal D}

\def\EE{\mathcal E} 
\def\G{\mathcal G}
\def\J{\mathcal J}
\def\K{\mathcal K}
\def\Q{\mathcal Q}
\def\Y{\mathcal Y}
\def\OP{{\rm OP}}
\def\I{{\bf I}}
\def\TT{\mathcal T}
\def\T{\mathbb T}
\def\R{\mathbb R} 
\def\N{\mathbb N} 
\def\Z{\mathbb Z} 
\def\kk{\kappa}
\def\d{{\,\rm{d}}} 
\def\LL{\mathcal L}
\def\si{\sigma}
\def\<{\langle} 
\def\>{\rangle}

\def\DD{\Delta}
\def\H{\mathbb H}
\def\U{\mathbb U}
\def\M{\mathbb M}
\def\F{\mathcal F}
\def\tt{\widetilde}
\def\be{\beta}

\title{{\bf A General Framework for Solving Singular SPDEs with  Applications to Fluid Models Driven by  Pseudo-differential Noise}\footnote{Feng-Yu Wang is supported 
by   the National Key R\&D Program of China (No. 2022YFA1006000, 2020YFA0712900), NNSFC (No.  11921001) and by the Deutsche Forschungsgemeinschaft (DFG, German Research Foundation) -- SFB 1283/2 2021 -- 317210226.} }
\author{{ \textbf{Hao Tang}$^{2)}$ and \textbf{Feng-Yu Wang}$^{1)}$ }\\
\footnotesize{\it $^{1)}$Center for Applied Mathematics, Tianjin University, Tianjin 300072, China}\\
\footnotesize{\it $^{2)}$ Department of Mathematics,
University of Oslo, P.O. Box 1053, Blindern, N-0316 Oslo, Norway}\\
\footnotesize{ haot@math.uio.no; wangfy@tju.edu.cn}}

\begin{document}

\begin{CJK}{UTF8}{gbsn}

\maketitle

\begin{abstract} 
In this paper we focus on nonlinear SPDEs with  singularities included in both drift and noise coefficients, for which the Gelfand-triple  argument  developed for (local) monotone SPDEs turns out to be invalid.   We propose a general framework of \emph{proper regularization} to solve such singular SPDEs.  As applications, the (local and global) existence is presented for a broad class of fluid models driven by pseudo-differential noise of arbitrary order, which include the stochastic magnetohydrodynamics (hence Navier-Stokes/Euler) equations, stochastic Camassa-Holm type equations, stochastic aggregation-diffusion equation and stochastic surface quasi-geostrophic equation. Thus, some recent results derived in the literature are considerably extended in a unified way.  
\end{abstract} 


\tableofcontents 

\section{Introduction}

The study of stochastic partial differential equations (SPDEs), in particular for singular nonlinear models arising from fluid mechanics, is very rich and active. There exist a huge number of references in the literature, see, for instance, \cite{Mikulevicius-Rozovskii-2004-SIAM,Flandoli-Luo-2021-PTRF,
Li-Liu-Tang-2021-SPA,Rohde-Tang-2021-JDDE,Alonso-Bethencourt-2020-JLNS,Crisan-Flandoli-Holm-2018-JNS} and monographs \cite{Flandoli-2008-SPDE-book, Flandoli-2011-book,Kuksin-Shirikyan-2012-book,Breit-Feireisl-Hofmanova-2018-Book}. 
In this paper, we intend to build up a general framework for nonlinear SPDEs with singularity in both drift and noise, such that a large class of models  can be solved in a unified way. 

\subsection{Singular evolution equation in Hilbert space}

We will study a stochastic system on a separable Hilbert space $\H$, with coefficients taking values in a larger Hilbert space $\M$ such that
\begin{equation*}
\H\hookrightarrow \M, 
\end{equation*}
that is,  $\H$ is densely embedded into $\M$ with  $\|\cdot\|_\M\lesssim \|\cdot\|_\H$. Here  and in the sequel,  for two nonnegative variables $A$ and $B$, $A\lesssim B$ means that there exists a constant $c>0$ such that $A\le c\,B$. 
The level of singularity  is described by how larger  $\M$ is than $\H$.  

The system will be  driven by  the cylindrical Brownian motion  $\W$ on  another separable Hilbert space $\U$:  
$$\W(t):=\sum_{k\ge 1} W_k(t) e_k,\ \ t\ge0, $$  
where   $\{e_k\}_{k\ge 1}$ is a   complete orthonormal basis of $\U$, and 
$\{W_k\}_{k\ge 1}$ is a sequence of independent 1-D Brownian motions on a right-continuous complete filtration probability space $(\Omega,\{\F_t\}_{t\ge 0}, \P).$  Let   $\LL_2(\U; \H)$ be  the space of Hilbert-Schmidt operators from $\U$ to $\H$.

We now consider the following stochastic equation  for unknown process $X=(X(t))_{t\ge 0}$ on  $\H$: 
\begin{equation}
\d X(t)=\big\{b(t,X(t))+g(t,X(t))\big\}\d t+h(t,X(t))\d\W(t), \ \ t\ge 0, \label{E1}
\end{equation}
where  
\begin{equation}\label{b g h}
b:[0,\infty)\times \H\rightarrow\H,\ \ \ g:[0,\infty)\times \H\rightarrow\M,\ \ \ h:[0,\infty)\times \H\rightarrow \LL_2(\U;\M)
\end{equation} 
are   measurable. In applications,   $b$ refers to the regular part of drift, while the drift term $g $ and the noise term $h(t,X(t))\d \W(t)$ are singular in the sense that they   take values in $\M$, which is larger than the state space $\H$.

\begin{Definition}\label{solution definition} Let $\tau$ be a stopping time satisfying $\P(\tau>0)=1,$ and let $(X,\tau):=(X_t)_{t\in [0,\tau)}$ be a progressively measurable process on $\H$. 
\begin{enumerate}[label={{\rm (\arabic*)}}]
\item $($Local\ solution$)$. We call $(X,\tau)$ a local solution to \eqref{E1}, if $\p$-{\rm a.s.} we have $t\mapsto X(t)$ is continuous in $\M$ $($hence weakly continuous in $\H$$)$,
$$\sup_{s\in[0,t]}\|X(s)\|_{\H}<\infty,\ t\in[0,\tau),$$ and the following equation holds on the space $\M$:
\begin{align}
&X(t)-X(0) \notag\\
=\,&\int_0^{t} \big\{ b(s, X(s))+ g(s,X(s)) \big\} \d s+\int_0^{t}h(s,X(s) ) \d \W(s),\ \ t\in[0,\tau).\label{define solution} \end{align} 
\item $($Maximal\ solution$)$. 
A local solution $(X,\tau)$ is called maximal,
if 
\begin{equation*}
\limsup_{t\uparrow\tau}\|X(t)\|_{\H}=\infty\ \ {\rm a.s.\ on}\ \ \{\tau<\infty\}.
\end{equation*}
Particularly, if $\P(\tau=\infty)=1$, the solution is called global or non-explosive. 
\end{enumerate} 

\end{Definition}

The first main result  in this paper  presents the following properties for solutions to \eqref{E1} with  coefficients given in \eqref{b g h}. 

\begin{MResult}[see Theorem \ref{T1} below] The existence, uniqueness,    continuity in $\H$, blow-up  and non-explosion criterion. \end{MResult}

\subsection{SPDE with pseudo-differential noise}
To see that  \eqref{E1} includes a large class of singular SPDEs as special situations, we introduce below a general type SPDE on $\mathbb K^d$ for some $d\in \N$, where 
$\mathbb K=\R$ or $\mathbb K= \T:= \R/2\pi\mathbb Z$.  
We refer to Section \ref{Section:Notations} for the precise definitions of related notions. 

Let  $\{W_k,\tt W_k\}_{k\ge 1}$ be  independent 1-D Brownian motions,  and  let $\circ \d W_k(t)$ be the Stratonovich stochastic differential. 
Let $\mathscr{F}$ be the Fourier transform on $\mathbb K$ and ${\rm i}=\sqrt{-1}$ be the imaginary unit.
For $s>0$ and $m\in\mathbb N$,  we denote by $H^s(\mathbb K^d;\R^m)$ the standard Sobolev spaces of order $s$ for $\R^m$-valued functions on $\mathbb K^d$.
Let  $\OP\SS^s$ be the class of pseudo-differential operators with symbols in $\SS^s$, and let $\OP\SS_0^s$ be subset of $\OP\SS^s$ with symbols independent of $x$ (see \eqref{Ss Real} and \eqref{OPS Real}).   

Let $\tt {\mathbb K}:= \R$ if $\mathbb K=\R$, and  $\tt {\mathbb K}:= \Z$ for $\mathbb K=\T$.  
Assume that $\tt\Pi: L^2(\mathbb K^d; \R^m)\to L^2(\mathbb K^d; \R^m)$ is a projection operator satisfying 
\begin{equation}\label{tt Pi}
[\mathscr{F}(\tt\Pi f)](\xi)=\tt\pi(\xi)(\mathscr{F}f)(\xi),\ \ \big\<\tt\Pi f,g\big\>_{L^2}=\<f,\tt\Pi g\>_{L^2},\ \ 
\big\|\tt\Pi f\big\|_{H^s}\leq \|f\|_{H^s}
\end{equation} for all $ s\ge0,\ \xi\in \tt{\mathbb K}^d, f,g\in L^2,$ and 
some measurable   $\tt\pi: \tt{\mathbb K}^d\to \mathbb C^{m\times m}$ such that 
\begin{equation}\label{tt Pi pi0}
\tt \pi(-\xi)=\overline{\tt\pi(\xi)}:= {\rm Re} [\tt\pi(\xi)]- {\rm Im} [\tt\pi(\xi)]\,{\rm i},\ \  \ \xi\in\tt{\mathbb K}^d.
\end{equation}
Typical examples of $\tt\Pi$ in the theory of PDEs modeling fluid dynamics include
\begin{equation*}
\tt\Pi=
\begin{cases}
\I:\ \text{identity mapping},\\
\Pi_d:\ \text{Leray projection on}\ \mathbb K^d,\  \text{see}\ \eqref{Pi-d define},\\
\Pi_0:\ \text{zero-average projection on}\ \T^d,\  \text{see}\ \eqref{Pi-0 define}.
\end{cases}
\end{equation*}
Particularly, if $\tt \Pi=\Pi_d$, we assume $m=d$ and we will consider the equations in
$H^s_{{\rm div}}(\mathbb K^d;\R^d)$ (see \eqref{H-div}).

Consider the following nonlinear PDE in $\H:=\tt\Pi H^s(\mathbb K^d;\R^m)$:
\begin{equation}\label{PDE}
\frac{\d}{\d t} X(t)=\tt\Pi \EE X(t)+\tt b(X(t))+\tt g(X(t)), \ \ t\ge0,
\end{equation}
where  for some constants $p_0,q_0>0$, $$\OP\SS_0^{2p_0}\ni \EE:H^s(\mathbb K^d;\R^m)\to H^{s-2p_0}(\mathbb K^d;\R^m)$$ is a  negative semi-definite operator  (see Section \ref{Section:Notations} and \ref{Assum-E}  below for more details),  $$\tt b:\tt\Pi H^s(\mathbb K^d;\R^m)\to\tt\Pi H^s(\mathbb K^d;\R^m)$$ is the regular part, and 
$$\tt g:\tt\Pi H^s(\mathbb K^d;\R^m)\to\tt\Pi H^{s-q_0}(\mathbb K^d;\R^m)$$ is the singular part losing regularities of order $q_0$ (see \ref{Assum-hbg} for the statement).
The parameters $s,q_0,p_0$ are to be determined in different examples (see Sections \ref{Section:Results on moels} and \ref{Section:Applications} for concrete $(\EE,\tt b,\tt g)$ and $(s,q_0,p_0)$).

In the existing literature, SPDEs with transport noise has been intensively investigated in recent years, where the transport noise
is given by
\begin{equation}\label{known transport noise}
\Big(c_kX+\Big(\sum_{i=1}^d d_{k,i} \pp_i\Big)X \Big) \circ \d W_k(t),\ \  \pp_i:=\frac{\pp}{\pp x_i},
\end{equation}
where $c_k,\ d_{k,i}$ are nice $\R^{m\times m}$-valued and $\R$-valued functions, respectively, 
see for examples
\cite{Flandoli-Gubinelli-Priola-2010-Invention,Flandoli-Galeati-Luo-2021-CPDE,Crisan-Holm-2018-PhyD,Crisan-Flandoli-Holm-2018-JNS,Alonso-etal-2019-NODEA,Holden-Karlsen-Pang-2021-JDE,Flandoli-Luo-2021-PTRF,Albeverio-etal-2021-JDE,Alonso-Bethencourt-2020-JLNS,Alonso-Rohde-Tang-2021-JLNS,Goodair-Crisan-Lang-2022-arXiv}.  However, as far as we know, there is no result for the case that $\pp_i$ in \eqref{known transport noise} is replaced by more general differential operators, for which   the system is allowed to have non-local  higher order singular noise coefficients.

Therefore, the second main result in this paper is  to study \eqref{PDE} with pseudo-differential noise.  We will consider the following SPDE for unknown process $X(t)$ on $\H=\tt\Pi H^s(\mathbb K^d;\R^m)$:
\begin{align}
\d X(t)=\,& \Big\{(\tt\Pi\EE X)(t)+\tt b(X(t)) +\tt g(X(t)) \Big\}\d t\nonumber\\
&+\sum_{k=1}^\infty \Big\{ (\tt\Pi \A_kX)(t) \circ \d W_k(t)
+\tt\Pi \tt h_k (t,X(t))\d \tt W_k(t)\Big\},\ t\ge0,\label{EN}
\end{align} 
where $\{\A_k\}_{k\ge1}\subset \OP\SS^{r_0}$ is a sequence of pseudo-differential operators not far away from
anti-symmetric (see \ref{Ak-assum} for the precise assumptions) and $\tt h_k(t,\cdot): H^s(\mathbb K^d;\R^m)\to H^s(\mathbb K^d;\R^m)$ ($k\ge 1,\ t\ge 0$) are locally Lipschitz continuous (we refer to \ref{Assum-h} for the details).

Then the second main result in this paper focuses on

\begin{MResult}[see Theorem \ref{T3.1} below] Well-posedness of \eqref{EN}, including existence, uniqueness, time-continuity, blow-up criterion and global existence.
\end{MResult}

\subsection{Comparison, motivation and  remarks}

 To begin with, we give some comments on  Definition \ref{solution definition}.  Due to the singularities of $g$ and $h$, the corresponding integrals in \eqref{define solution} are only defined in the larger space $\M$, but their sum has to take value in the state space $\H$ since $X(t)$ is a process on $\H$. In applications, the singular coefficients $g$ and $h$ may take values in $\M_1$ and $\LL_2(\U;\M_2)$ for some different Hilbert spaces $\M_1$ and $\M_2$.  In this case, we take a larger Hilbert space $\M$ with $\M_i \hookrightarrow\M$ $(i=1,2).$


\subsubsection{Comparing Main Result \textbf{(I)} with existing results}

\begin{enumerate}[label={ $\bf (\arabic*)$}] \setlength\itemsep{-0.1em}
\item  For the present model with singular noise coefficient, the Gelfand-triple  argument   developed for (local) monotone SPDEs turns out to be invalid.  To see this, we consider the triplet of embedded Hilbert Spaces  $\mathbb{V}\hookrightarrow\H\hookrightarrow\M$ such that $\mathbb{V}\hookrightarrow\H$ is dense and $\langle X,Y\rangle_{\M\times\mathbb{V}}=\<X,Y\>_{\H}$. Then we consider the  following  SDE: $$\Psi(t)=\Psi(0)+\int_{0}^{t}A \d t'+\int_{0}^{t}B \d \W(t'),\ \ A\in \M,\ \ B\in\LL_2(U;\H),\ \ \Psi(0)\in \H.$$
In order to apply It\^o's formula to the above SDE, the condition that $\Psi(t)\in \mathbb{V}$ (see \cite{Prevot-Rockner-2007-book,Wang-2013-Book,Goodair-Crisan-Lang-2022-arXiv}) is  \textit{necessary} because the dual product $\langle A,\Psi\rangle_{\M\times\mathbb{V}}$ needs to be well-defined. But this means that $\Psi(t)\in \mathbb{V} \hookrightarrow \H\ni \Psi(0)$, that is, $\Psi(t)$ is more regular than its initial data $\Psi(0),$ which is \textit{inconsistent} for ideal fluid motion without viscosity. For example, for the following inviscid Burgers' equation with the choice $\H:=H^s$, $\mathbb{V}:=H^{s+1}$, $\M:=H^{s-1}$ and $s>3/2$ (see Sections \ref{Section:Results on moels} and \ref{Section:Notations} for more notations)
$$\frac{\d}{\d t}X+X\partial X=0,\ \ X(0)\in \H,$$ one can \textit{only} know that $X\in \H$ and hence $X\partial X\in \M$. Then, neither the inner product $\<X\partial X,X\>_{\H}$ nor the dual product $\<X\partial X,X\>_{\M\times \mathbb{V}}$ make no sense.  Alternatively to the Galerkin approximation used in the monotone situation, we will propose  the \textit{proper regularization}   (see Definition \ref{Define-regularization})  to solve  \eqref{E1}.

\item  A blow-up criterion of solutions  in $\H$ is described by the $\M$-norm, see \eqref{Blow-up criterion} below. Since $\|\cdot\|_\M\lesssim \|\cdot\|_\H$, the blow up in $\|\cdot\|_\M$   is stronger than that  in  $\|\cdot\|_\H$, but they are  indeed equivalent  under the present framework. Due to the singularities, in general, a solution $X(t)$ to \eqref{E1} may be only continuous in $\M$ rather than $\H$, so a criterion on the continuity in $\H$  is provided   in Theorem \ref{T1} \ref{T1-conti}. 
The non-explosion is included in  Theorem \ref{T1} \ref{T1-global} for strong enough  noise  in the sense of condition \ref{C}, see also \cite{Brzezniak-etal-2005-PTRF,Hasminskii-1969-Book,Ren-Tang-Wang-2020-Arxiv} for  non-explosion results in similar spirits.  

\item In Theorem \ref{T1}, $X(0)$ is only assumed to be an $\F_0$-measurable $\H$-valued variable without any moment condition. In this case the
conditional expectation $\E[\cdot|\F_0]$ will be used to replace the expectation $\E$ in the construction of solutions,   see Lemma \ref{Lemma:Convergence of X-n}.
It seems that conditional expectation has been rarely used in the literature of SPDEs. Besides, it is worthwhile mentioning that our framework does \textit{not} require any compactness of the embedding $\H\hookrightarrow\M$, so that the main result  applies not only to SPDEs on compact spaces, where the compactness is needed to apply Prokhorov's Theorem and Skorokhod's Theorem, but also to SPDEs on unbounded domain like $\R^d$, as shown by Theorems \ref{T3.1} and examples in Section \ref{Section:Applications}. 

\end{enumerate}

\subsubsection{Motivation and remarks on Main Result \textbf{(II)}}

\begin{description}

\item[Motivation.] Pseudo-differential operators offer a non-local extension to classical differential operators. Exploring pseudo-differential noise can provide a  versatile framework to model intricate random phenomena involving non-local random interactions. This can be particularly useful in turbulence models, where the behavior of fluid at one point is influenced by the behavior of fluid at distant points. To gain more insight into the non-locality arising from pseudo-differential noise, which classical transport noise cannot capture, we examine the stochastic Burgers' equation as a simple yet intriguing example. Let $W(t)$ be a standard 1-D Brownian motion. We first consider the case of  classical derivative, i.e.,
\begin{equation} \label{Toy model}
\d X+ X\partial X\d t=\sqrt{2\mu}\partial X\circ \d W(t),\ \ \mu>0.
\end{equation}
By utilizing the following relation for a semi-martingale $\Theta(t)$:
\begin{equation}
\Theta(t)\circ \d W(t)= \Theta(t)  \d W(t)+ \frac{1}{2}  \left\langle \Theta, W\right\rangle (t)\ \text{with}\  \left\langle\cdot,\cdot\right\rangle\ \text{being the quadratic variation}, \label{transform into Ito}
\end{equation}
we can reformulate \eqref{Toy model} as:
 \begin{equation*}
 \d X+ X\partial X  \d t=\sqrt{2\mu}\partial X \d W(t)+\mu \partial^2 X \d t.
 \end{equation*}
 Let $\Xi(t):=\exp\big\{-\sqrt{2\mu}W(t)\partial\big \}$ be an operator-valued process. Then, in the sense of Fourier multiplier, we have the following SDE:
 $$\d \Xi (t)=-\sqrt{2\mu}\partial \Xi (t) \d W(t)+\mu \partial^2 \Xi (t) \d t.$$
 Therefore, for $Y(t):=[\Xi X](t)$,  we can derive
 \begin{align}
 \d  Y=\,&  [{\rm d}\Xi ] (X) + \Xi  ({\rm d}X) -
 2\mu \partial^2  \Xi X {\rm d} t
 =\,
 -\Xi \big(\Xi^{-1}Y \cdot
 \partial \Xi ^{-1}Y\big) \d t.\label{S-Burgers}
 \end{align}
Next, we consider the case of pseudo-differential noise, i.e., 
\begin{equation*}  
\d X+ X\partial X\d t=\sqrt{2\mu}(-\partial^2)^\alpha X\circ \d W(t),\ \ \mu>0,\ \ \alpha\in(0,1/2].
\end{equation*}
A similar argument yields
\begin{align*}
\frac{\d}{\d t}Y +\Xi_{\alpha} \big(\Xi_{\alpha} ^{-1}Y  \cdot \partial \Xi_{\alpha} ^{-1}Y\big)=0,\ \ 
Y(t):=[\Xi_{\alpha} X](t),\ \  
\Xi_{\alpha} (t):=\exp\big\{-\sqrt{2\mu}W(t)(-\partial^2)^{\alpha}\big\}.
\end{align*}
Comparing the two cases above, we observe that the kernels of $\Xi$ and $\Xi^{-1}$ are delta functions, which indicates their local-in-$x$ nature. In the case of pseudo-differential noise, the non-locality arising from $\sqrt{2\mu}(-\partial^2)^{\alpha}X\circ \d W(t)$ is characterized by the term $\Xi_{\alpha}\big(\Xi_{\alpha}^{-1}Y \cdot \partial \Xi_{\alpha}^{-1}Y\big)$ at the level of $Y$. However, it seems that the kernels of $\Xi_{\alpha}$ and $\Xi_{\alpha}^{-1}$ cannot be explicitly determined.

\item[Projection $\tt\Pi$.]  For simplicity, in \eqref{EN} we assume that $\tt b$ and $\tt g$ already take values in projected space, as they come from deterministic PDEs. However, we keep $\tt\Pi$ in  $\tt \Pi \EE, \tt \Pi \A_k$ and $\tt\Pi \tt h_k$. For example, if $\tt\Pi=\Pi_d$ and $\EE=\DD$, $\Pi_d \EE$ is known as Stokes operator. As for $\A_k$ and $\tt h_k$, the projection is also necessary to make solutions stay in the projected space.  However,  in calculations it is non-trivial to deal with the case $\tt\Pi\neq\I$ (see \ref{Ak-assum-Pi} and Lemma \ref{LAK}) since  $\tt\Pi$ may \textit{not} be a pseudo-differential operator. For instance, the   Fourier multiplier of  $\tt\Pi=\Pi_d$ is singular at $0$ (see \eqref{Pi-d define} below).

\item[Order of pesudo-differential operators.] We mainly consider the case that the operator $\A_k$ in \eqref{EN} contains two parts:  $x$-dependent part with  order $r_1\in[0,1]$ and $x$-independent part with order $r_2\ge r_1$ (see  \ref{Ak-assum}). Each of them 
extends the classical transport type noise structure \eqref{known transport noise}. In particular, $r_2$ can be arbitrary large.  Accordingly,  Theorem \ref{T3.1} enlarges many results derived in the literature by allowing highly non-local and singular    noise (see examples in Section \ref{Section:Results on moels}). We refer to \cite{Crisan-Flandoli-Holm-2018-JNS,Alonso-Bethencourt-2020-JLNS,Alonso-etal-2019-NODEA,Alonso-Rohde-Tang-2021-JLNS,Flandoli-Galeati-Luo-2021-CPDE,Goodair-Crisan-Lang-2022-arXiv,Holden-Karlsen-Pang-2022-DCDS} and the references therein for the known results with transport type noise given by \eqref{known transport noise}.

\end{description}

%
%
%

\subsection{Results on specific models}\label{Section:Results on moels}

As \eqref{EN}  covers a broad class of SPDEs with different choices of $(\mathcal E, b,g)$,  we apply   Theorem \ref{T3.1} to 
some important  models from fluid mechanics.   Since there is  an enormous literature on these equations,   we do not try to provide a complete account but  only mention a few results. 

Recall that $\mathbb K=\R$ or $\T$ and $\pp_i$ stands for the $i$-th partial derivative on $\mathbb K^d$. For a function $f=(f_i)_{1\le i\le d}:\mathbb K^d\to\R^d$, we let 
$$(f\cdot\nn):=\sum_{i=1}^d f_i\pp_{i},\ \ \ \ \nn\cdot f:= \sum_{i=1}^d \pp_i f_i.$$
In particular, when $d=1$, we simply denote  by $\pp=\pp_1$  the   derivative  in  $\mathbb K$ and let (see \eqref{Ds Lambda s define} for precise definition)
$$\Lambda:=\left(-\DD\right)^{\frac{1}{2}},\ \ \ \ \D:=\left(\I-\DD\right)^{\frac{1}{2}}.$$ Let ${\rm diag}(\cdot,\cdots,\cdot)$  be  the diagonal operator. For more notations, we refer to Section \ref{Section:Notations}.

\textbf{(1)} 
\textit{Magnetohydrodynamics} ({\bf MHD}) \textit{equation}. Let $d\ge 2$. Consider the following incompressible generalized magnetohydrodynamics equation with fractional kinematic dissipation and magnetic diffusion on $\mathbb K^d$, cf. \cite{Wu-2003-JDE},
\begin{equation}\label{MHD}
\left\{\begin{aligned}
&\frac{\d}{\d t} V(t) + \mu_1 \Lambda^{2\alpha_1} V(t) + \Pi_d(V (t)\cdot \nabla ) V(t) -\Pi_d(M(t) \cdot\nn) M(t) =0,\\
&\frac{\d}{\d t} M(t) +\mu_2 \Lambda^{2\alpha_2} M(t) + (V(t) \cdot \nabla) M(t)- (M(t)\cdot\nn) V(t) =0,\\
&\nabla\cdot V(t) = \nabla\cdot M (t)=0, 
\end{aligned}\right.
\end{equation} 
where $V: [0,\infty)\times \mathbb K^d\to \R^d$ is the velocity field, $M: [0,\infty)\times \mathbb K^d\to \R^d$ is magnetic field, $\Pi_d$ is the Leray projection, $\alpha_1,\alpha_2\in [0,1]$ are the fractional powers and $\mu_1,\mu_2\ge 0$ stand for the kinematic viscosity and magnetic diffusivity constants, respectively. 
Letting $X=(V,M)^T$, \eqref{MHD} can be reformulated as
$$\frac{\d}{\d t} X(t) =\EE^{{\rm mhd}}X(t) +g^{{\rm mhd}}(X(t))$$ with
\begin{equation} \label{MHD-E-g} 
\begin{cases}
\EE^{{\rm mhd}}:=\,-{\rm diag} \big(\mu_1 \Lambda^{2\alpha_1},\ \mu_2\Lambda^{2\alpha_2}\big),\\
g^{{\rm mhd}}(X):=\,\Big( \Pi (M\cdot\nn ) M-\Pi (V\cdot\nn )V,\ (M\cdot\nn)V-(V\cdot\nn )M\Big)^T.
\end{cases} 
\end{equation}
Obviously, when $M\equiv 0$ and $\alpha_1=1$, the equation for $V(t)$  covers  the incompressible Navier-Stokes equation ($\mu_1>0$) and the Euler equation $(\mu_1=0)$.

\textbf{(2)}
\textit{Camassa-Holm} ({\bf CH}) \textit{type equation}. The following equation  for $X: [0,\infty)\times\mathbb K\to\R$ 
\begin{equation}
\frac{\d}{\d t} (X -\pp^2X )(t)+3 X(t)  \pp X (t) = 2\pp X(t) \pp^2 X(t) +X(t) \pp^3 X(t)\label{CH}
\end{equation}
was first introduced by Fokas \& Fuchssteiner \cite{Fuchssteiner-Fokas-1981-PhyD} to study completely integrable generalizations of KdV equation with bi-Hamiltonian structure. In \cite{Camassa-Holm-1993-PRL}, Camassa \& Holm proved that \eqref{CH} can be connected to the unidirectional propagation of shallow water waves over a flat bottom ($X$ represents the free surface of water), and now \eqref{CH} is usually called the Camassa-Holm equation. In order to include some closely related equations (such as the Degasperis-Procesi equation \cite{Degasperis-Procesi-1999-chapter}, b-family equations \cite{Holm-Staley-2004-PhyA}, recently derived rotation Camassa-Holm equation \cite{Gui-Liu-Sun-2019-JMFM}, and even different drift terms in stochastic cases \cite{Albeverio-etal-2021-JDE,Crisan-Holm-2018-PhyD}), we consider the following \textbf{CH} type equations:
\begin{equation}\label{CH-b-g}
\frac{\d}{\d t} X(t) =b^{{\rm ch}}(X(t))+g^{{\rm ch}}(X(t))
\end{equation}
by taking
\begin{equation*} 
b^{{\rm ch}}(X):= -\pp\D^{-2}\Big(\sum_{i=1}^4 a_i X^i+a |\pp X|^2\Big),\ \ g^{{\rm ch}}(X):= -X\pp X,
\end{equation*} 
where $\{a,a_i\}_{1\le i\le 4}$ are constants. We refer to \cite{Tang-Yang-2022-AIHP} for the link of \eqref{CH-b-g} to the above mentioned Camassa-Holm-family equations.

\textbf{(3)} 
\textit{Aggregation-diffusion} ({\bf AD}) \textit{equation}.
Let $d\ge 2$. Consider the following aggregation-diffusion type models on $\mathbb K^d$:
\begin{align}\label{AD}
\frac{\d}{\d t} X(t) +\nu\Lambda^{2\be}X (t)+ \gamma\nn\cdot \big\{X(t) \nabla [\Phi\star X(t)]\big\}=0,\ \ t\ge 0,
\end{align}
where $X: [0,\infty)\times \mathbb K^d\to \R$ represents the density of species (cells), $\nu\ge0$ and $\gamma\in\R\setminus\{0\}$ are the diffusion and aggregation coefficients respectively, $\be\in[0,1]$ is the diffusion order, $\Phi$ is an interaction kernel and $\star$ stands for the convolution.
The equation \eqref{AD} has a range of applications arising in physics and biology with specific choices of $\Phi$, see \cite{Perthame-2006-book,Keller-Segel-1970-JTB} for self-organization of chemotactic movement, see \cite{Burger-etal-2014-SIAM} for biological swarm, and see the survey \cite{Carrillo-Craig-Yao-2018-chapter} for some other choices. In this paper
we assume that $\Phi\in H^\infty(\mathbb K^d,\R):=\cap_{s\ge0} H^s(\mathbb K^d,\R)$ such that $(\mathscr F\Phi)(\xi)\in \S_0^{-2}$, where $\S_0^{s}$ on $\mathbb K$ with $s\in\R$ is defined in Section \ref{Section:Notations} below.

A typical example of $\Phi$ is the Bessel kernel for which $\{\Phi\star\}=\D^{-2}\in\OP\SS_0^{-2}$ (see \eqref{Ds Lambda s define} and  \eqref{OPS Real}), and in this case \eqref{AD} reduces to the well-known parabolic-elliptic Keller-Segel system (see for example \cite{Corrias-Perthame-Zaag-2004-MJM}):
$$ \frac{\d}{\d t} X(t) +\nu\Lambda^{2\be}X (t)
+ \gamma\nn\cdot (X (t)\nabla Y(t))=0,\ \ Y(t)-\Delta Y(t)=X(t).$$
Then we rewrite \eqref{AD} as
\begin{equation*}
\frac{\d}{\d t}X(t) =\EE^{{\rm ad}}X (t)+g^{{\rm ad}}(X(t)),
\end{equation*}
where
\begin{equation}\label{AD-E-g} 
\ \EE^{{\rm ad}}:=-\nu(-\DD)^{\beta},\ \ g^{{\rm ad}}(X):= - \gamma\nn\cdot\left(X\BB X\right),\ \ \BB:= \nn[\Phi\star]. 
\end{equation}

\textbf{(4)}  
\textit{Surface quasi-geostrophic} ({\bf SQG}) \textit{equation}.  We consider the following equation on $\mathbb K^2$:
\begin{equation}\label{SQG-g}
\frac{\d}{\d t}X (t) =g^{{\rm sqg}}(X(t) ),\ \ \ g^{{\rm sqg}}(X): = -(\mathcal R^\perp X)\cdot \nn X, 
\end{equation}
where $X: [0,\infty)\times\mathbb K^2\to\R$, $\mathcal{R}=(\mathcal{R}_j)_{j=1,2}=(\pp_j\Lambda^{-1})_{j=1,2}=(\pp_j(-\DD)^{-\frac12})_{j=1,2}$ is the Riesz transform on $\mathbb K^2$, and
\begin{equation*} 
\mathcal{R}^\perp
:=\big(-\mathcal{R}_2,\mathcal{R}_1\big).
\end{equation*}
Equation \eqref{SQG-g} is an important model in geophysical fluid dynamics. Actually,
it is a special case of the general quasi-geostrophic approximations for atmospheric
and oceanic fluid flow with small Rossby and Ekman numbers. Besides, \eqref{SQG-g} is also an important example of a 2-D active scalar with
a specific structure most closely related to the incompressible Euler equations. 
We refer to the recent monograph \cite{Castro-etal-2020-book} for more details.

Considering stochastic variants of the above PDEs with  both $\mathbb K=\R$ and $\mathbb K=\T$,  the final main result of this paper is to establish
\begin{MResult} 
Novel results on  the well-posedness of   the SPDE  
\begin{align*}
\d X(t)=\,& \Big\{(\tt\Pi\EE X)(t)+\tt b(X(t)) +\tt g(X(t)) \Big\}\d t\\ &+\sum_{k=1}^\infty \Big\{ (\tt\Pi \A_k)X(t) \circ \d W_k(t)+ \tt\Pi \tt h_k (t,X(t))\d \tt W_k(t)\Big\},\ \ t\ge0,
\end{align*} 
on  $\H:=\tt\Pi H^s (\mathbb K^d;\R^m)$   with the following choices:

\begin{itemize}
\item Stochastic {\bf MHD} equations $($cf. Theorem $\ref{SMHD-T}):$ $d\ge2$, $m=2d$, $\mathbb K=\R$ or $\mathbb T$,  $\tt b\equiv0$, 
$(\EE,\tt g)=(\EE^{{\rm mhd}}, g^{{\rm mhd}})$ is given in \eqref{MHD-E-g}, and $\tt\Pi={\rm diag}(\Pi_d,\Pi_d)$ if $\mathbb K=\T$  
while $\tt\Pi={\rm diag}(\Pi_d\Pi_0,\Pi_d\Pi_0)$ if $\mathbb K=\T$, where $\Pi_d$ and $\Pi_0$ are defined in \eqref{Pi-d define} and \eqref{Pi-0 define}, respectively;

\item Stochastic {\bf CH} equation $($see Theorem $\ref{SCH-T}):$ $d=m=1$, $\mathbb K=\R$ or $\mathbb T$, $\tt \Pi=\I$, $\EE\equiv0$, $(\tt b,\tt g)=(b^{{\rm ch}},g^{{\rm ch}})$ is given in \eqref{CH-b-g};

\item Stochastic {\bf AD} equation $($see Theorem $\ref{SAD-T}):$ $d\ge2$, $m=1$, $\mathbb K=\R$ or $\mathbb T$, $\tt \Pi=\I$, $\tt b\equiv0$, $(\EE,\tt g)=(\EE^{{\rm ad}}, g^{{\rm ad}})$ is given in \eqref{AD-E-g};

\item Stochastic {\bf SQG} equation $($cf. Theorem $\ref{SSQG-T}):$ $d=2$, $m=1$, $\mathbb K=\R$ or $\T$,  $\tt b\equiv0$, $\EE\equiv0$ and $\tt g=g^{{\rm sqg}}$ are given in \eqref{SQG-g}, $\tt\Pi=\I$ if $\mathbb K=\R$ and $\tt\Pi=\Pi_0$ if $\mathbb K=\T$, where $\Pi_0$ is given in \eqref{Pi-d define}.
\end{itemize}

\end{MResult}

To conclude this section,  we briefly recall some  references  on nonlinear SPDEs. The stochastic \textbf{MHD} equation has been studied in
\cite{Barbu-Prato-2007-AMP,Hong-Li-Liu-2021-FMC,Li-Liu-Tang-2021-SPA} and the references therein,   see   \cite{Chen-Duan-Gao-2021-PhyD,Chen-Duan-Gao-2021-CMS,Chen-Miao-Shi-2023-JMP,Ren-Tang-Wang-2020-Arxiv,Holden-Karlsen-Pang-2021-JDE,Rohde-Tang-2021-NoDEA,Rohde-Tang-2021-JDDE,Holden-Karlsen-Pang-2022-DCDS,Tang-Yang-2022-AIHP} for the study of  stochastic \textbf{CH} type equations,  \cite{Hausenblas-etal-2022-JDE,Misiats-etal-2022-NoDEA,Rosenzweig-Staffilani-2023-PTRF,Flandoli-Galeati-Luo-2021-CPDE,Tang-Wang-2023-CCM} for the stochastic \textbf{AD} equation,  and  \cite{Rockner-Zhu-Zhu-2015-AoP,Brzezniak-Motyl-2019-SIAM,Zhu-Brzezniak-Liu-2019-SIAM,Alonso-Miao-Tang-2022-JDE} for \textbf{SQG} type equations. 
Results presented in these references are now extended to the case with highly non-local and singular  noise.

The remainder of the paper is organized as follows. 
In Section \ref{Section:General}, we present a general framework (see Theorem \ref{T1}) on the existence, uniqueness, blow-up and non-explosion criteria for \eqref{E1}. 
In Section \ref{Section:SPDEs}, we apply our general framework to derive a well-posedness result (see Theorem \ref{T3.1}) on SPDEs with pseudo-differential noises in Sobolev spaces. This result will be used in Section \ref{Section:Applications} to obtain novel results (see Theorems \ref{SMHD-T}-\ref{SSQG-T}) on above mentioned stochastic models (\textbf{MHD}, \textbf{CH}, \textbf{AD} and \textbf{SQG}), see also  Section \ref{Section:2 example}  for  two more examples of SPDEs.

\section{A general framework} \label{Section:General}

In this section, we establish a general framework to solve \eqref{E1}.
Section \ref{Section:Assum-resutls} includes assumptions on the coefficients $(b,g,h)$ and the main result Theorem \ref{T1} on the local existence, uniqueness, blow-up criterion and global existence. Complete proof of Theorem \ref{T1} is addressed in Section \ref{Section:T1 proof}, 
and a  further improved blow-up criterion is given in Section \ref{Subsection-T-remark}. 

\subsection{Assumptions and main results}\label{Section:Assum-resutls}\ 
We first introduce the notion of \textit{proper regularization} for the singular coefficients $g$ and $h$, which covers concrete regularizations 
used in \cite{Majda-Bertozzi-2002-book,Taylor-1991-book} for the deterministic case, and in
\cite{Li-Liu-Tang-2021-SPA,Ren-Tang-Wang-2020-Arxiv,Alonso-Rohde-Tang-2021-JLNS,Tang-Yang-2022-AIHP} for the stochastic case. 
For two topological spaces $E_1$ and $E_2$, let $\mathscr{B}(E_1;E_2)$ be the class of all measurable maps from $E_1$ to $E_2$, while $C(E_1;E_2)$ 
consists of all continuous maps from $E_1$ to $E_2$. When $E_2$ is a metric space, let $\mathscr{B}_b(E_1;E_2)$ be the set of bounded elements in 
$\mathscr{B}(E_1;E_2)$. Let $\mathscr{K}\subset \mathscr{B}\big( [0,\infty)\times [0,\infty); (0,\infty)\big)$ such that
\begin{equation}\label{SCRK} \mathscr K:=\Big\{K(x,y)\ \text{is\ increasing\ in\ } y \text{\ and\ locally\ integrable\ in\ }x\Big\}.\end{equation} 

\begin{Definition}[Proper Regularization]\label{Define-regularization} $\{(g_n, h_n)\}_{n\ge 1}$ is called a proper regularization of $(g,h)$, if 
$$g_n: [0,\infty)\times\H \rightarrow\H,\ \ \ h_n: [0,\infty)\times\H \rightarrow\LL_{2}(\U;\H), \ \ n\ge 1$$ 
are measurable   such that the following conditions hold for some    $K\in \mathscr K$ 
and a dense subset $\M_0\subset \M$: 
\begin{enumerate}[label={ $\bf (R_\arabic*)$}]

\item\label{R1} For any $t\ge0$ and $X\in\H$, 
\begin{align*} 
\sup_{n\ge1}
\Big\{ \|g_n(t,X)\|_{\M}+ \|h_n(t,X)\|_{\LL_2(\U;\M)} 
\Big\}\leq K(t,\|X\|_{\H}),
\end{align*} 
$$\lim_{n\to\infty}\big\{\|g_n(t,X)- g(t,X)\|_\M+ \|h_n(t,X)- h(t,X)\|_{\mathcal L_2(\mathbb U;\M)}\big\}=0.$$

\item\label{R2} For any $n,N\ge 1$, 
\begin{align*} 
&\sup_{t\lor\|X\|_\H\le N}
\bigg\{\|g_n(t,X)\|_\H +\|h_n(t,X)\|_{\LL_2(\U;\H)}\bigg\}<\infty,\\
\sup_{t\lor\|X\|_\H\lor\|Y\|_\H\le N}&
{\bf 1}_{\{X\ne Y\}} \frac{\|g_n(t,X)-g_n(t,Y)\|_\H+\|h_n(t,X)-h_n(t,Y)\|_{\LL_2(\U;\H)}}{\|X-Y\|_{\H}} <\infty.
\end{align*}

\item\label{R3} For any $Y\in\M_0,\ T>0$ and $\{X_n, X\}_{n\ge 1} \subset   \mathscr{B}_b([0,T];\H)\cap C([0,T];\M)$ with $X_n\to X$ in $C([0,T];\M)$ as $n\to\infty$, 
\begin{align*}
\lim_{n\to\infty}\int_{0}^{T}\Big\{ 
&\Big|\big\<g_n(t,X_n(t))-g(t,X(t)),Y\big\>_{\M}\Big|\notag\\
&+\sum_{k\ge1}
\big\<\big\{h_n (t,X_n(t)) -h(t,X(t))\big\}e_k,Y\big\>_{\M}^2\Big\} \d t=0.
\end{align*}

\item\label{R4} $($Cancellation of singularities$)$ For any $t \ge 0$ and $X\in\H$,
\begin{equation*}
\sup_{n\ge 1 } \sum_{k=1}^\infty \<h_n(t,X)e_k, X\>_\H^2\le K(t,\|X\|_\M)(1+\|X\|_\H^4),
\end{equation*}
\begin{equation*}
\sup_{n\ge1 } \left\{ 2\left\<g_n(t,X),X\right\>_{\H}+\|h_n(t,X)\|^2_{\LL_2(\U;\H)}\right\}\leq K(t,\|X\|_\M)(1+\|X\|_\H^2).
\end{equation*}
\end{enumerate} 
\end{Definition} 

With a proper regularization of $(g,h)$, we will prove   the existence and uniqueness of maximal solution to \eqref{E1} under the following assumption. 

\begin{Assumption}\label{A}
There exists a proper regularization $\{(g_n, h_n)\}_{n\ge 1}$ of $(g,h)$ such that the following conditions hold for some $K\in \mathscr K$ in \eqref{SCRK} and some function $\lambda: \mathbb N\times\mathbb N \rightarrow[0,\infty)$ with $\displaystyle \lim_{n,l\to\infty} 
\lambda_{n,l}=0:$
\begin{enumerate}[label={ $\bf (A_\arabic*)$}]
\item\label{A1} $($Regular drift$)$ For any $X\in \H$ and $t,N\ge 0$, 
\begin{equation*} 
\|b(t,X)\|_{\H}\leq K\big(t,\|X\|_{\M}\big)\|X\|_{\H},
\end{equation*}
\begin{equation*} 
\sup_{\|X\|_\H,\|Y\|_\H\le N}{\bf 1}_{\{X\ne Y\}} \left\{ \frac{\|b(t,X)-b(t,Y)\|_{\H}}{\|X-Y\|_{\H}}
+\frac{\|b(t,X)-b(t,Y)\|_{\M}}{\|X-Y\|_{\M}} \right\} \le K(t,N).
\end{equation*}

\item\label{A2} $($Asymptotic quasi monotonicity$)$ 
For any $n,l\ge1$, $t\ge 0$ and $X,\, Y\in \H$,
\begin{align*} 
\sum_{k\ge 1} \big\<\{h_n(t,X)&-h_l(t,Y)\}e_k, X-Y\big\>_\M^2\\
&\le\,  K(t,\|X\|_\H+ \|Y\|_\H)\|X-Y\|^2_{\M}\left(\lambda _{n,l}+\|X-Y\|^2_{\M}\right),
\end{align*} 
\begin{align*} 
2\big\<g_n(t,X)-& g_l(t,Y), X-Y \big\>_{\M}+ 
\| h_n(t,X)-h_l(t,Y)\|^2_{\LL_2(\U;\M)} \\
\le\,& K(t,\|X\|_\H+ \|Y\|_\H)\left(\lambda_{n,l}+\|X-Y\|^2_{\M}\right).
\end{align*} 
\end{enumerate}
\end{Assumption}
We call \ref{A2} an asymptotic quasi monotonicity condition in $\M$, since  it becomes a monotonicity condition with  coefficient $K(t,\|X\|_\H+ \|Y\|_\H)$  depending on the larger $\H$-norm when $n,l\to\infty$.

Due to the singularity of $(g,h)$, It\^{o}'s formula does not directly apply to $\|X(t)\|^2_{\H}$ for a solution to \eqref{E1}, which is crucial to deduce the continuity in $\H$ from the weak continuity imposed by definition. 
To this end, we make the following assumption, where the regularization operator $T_n$ makes It\^{o}'s formula available for $\|T_n X(t)\|^2_{\H}$ instead of $\|X(t)\|^2_{\H}$. 

\begin{Assumption}\label{B}
There exist $K\in\mathscr K$ in \eqref{SCRK}  and $\{T_n \}_{n\ge1}\subset \LL(\M;\H)$ such that 
$$
\ \lim_{n\rightarrow\infty} \|T_n X-X\|_\H=0, \ \ X\in \H,
$$
and for all $ t\geq0$ and $ N\ge1$, 
\begin{align*}
\sup_{n\ge1,\|X\|_\H\le N} \bigg\{&\sum_{i=1}^{\infty}\big\<T_n h(t,X)e_i,T_n X\big\>_{\H}^2,\\
&\Big|2\left\<T_n g(t,X), T_n X\right\>_{\H}
+\|T_n h(t,X)\|^2_{\LL_2(\U;\H)}\Big|
\bigg\} \leq\, K(t,N).
\end{align*}
\end{Assumption}
Finally, to prove the non-explosion, we assume the following Lyapunov type condition. 
When $V''<0$, a fast enough growth of the noise coefficient will kill the growth of other terms, such that the non-explosion is ensured.

\begin{Assumption}\label{C}
There exists a function $1\le V\in C^2([0,\infty))$ 
satisfying 
$$ V'(r)>0, \ V''(r)\le 0 \ \text{and} \ \lim_{r\to\infty} V(r)=\infty,$$
and a positive function $F(\cdot) \in L_{loc}^1([0,\infty))$
such that for all $(t,X)\in [0,\infty)\times \H$,
\begin{align*} 
V'(\|X\|^2_{\M})&\Big\{
2\big\< b(t,X)+g(t,X),X\big\>_\M
+\|h(t,X)\|^2_{\LL_{2}(\U;\M)}\Big\}\\
&+2V''(\|X\|^2_{\M})\sum_{k=1}^\infty \<h(t,X)e_k, X\>_\M^2 
\leq F(t)V(\|X\|^2_{\M}).
\end{align*}
\end{Assumption}

The main result is stated as follows. 

\begin{Theorem}\label{T1} Assume {\rm \ref{A}} and let $X(0)$ be an $\F_0$-measurable $\H$-valued random variable. Then the following assertions hold. 
\begin{enumerate}[label={\bf (\roman*)}]
\item\label{T1-existence} \eqref{E1} has a unique maximal solution $(X,\tau^*)$, and the solution satisfies
$\p$-{\rm a.s.} 
\begin{equation}\label{Blow-up criterion}
\limsup_{t\rightarrow \tau^*}\|X(t)\|_{\M}=\infty \ \text{on}\ \{\tau^*<\infty\}.
\end{equation}

\item\label{T1-conti} Under  {\rm \ref{B} }  the solution is continuous in $\H$, i.e., 
$\P\big(X\in C([0,\tau^*);\H)\big)=1.$
\item\label{T1-global} Under  {\rm \ref{C}}  the maximal solution $(X,\tau^*)$ is non-explosive, i.e., $\P(\tau^*=\infty)=1$.
\end{enumerate}

\end{Theorem} 

\subsection{Proof of Theorem \ref{T1}}\label{Section:T1 proof}

To construct the solution, we first localize the coefficients and then make regularity approximations. 
For any $R>1$, we take a cut-off function $\chi_R\in C^{\infty}([0,\infty);[0,1])$ such that 
$\chi_R(r)=1\ \text{for}\ |r|\le R,\ \text{and}\ \chi_R(r)=0\ \text{for}\ r>2R.$
Consider the following localization of \eqref{E1}: 
\begin{equation}\label{cut-off problem}
\left\{
\begin{aligned}
{\rm d}X^{(R)}(t) 
=\,&\chi^2_R\big(\|X^{(R)}(t)-X(0)\|_{\M}\big)\left[b(t,X^{(R)}(t))+g(t,X^{(R)}(t))\right]\d t\\
&\qquad \qquad \qquad+\chi_R\big(\|X^{(R)}(t)-X(0)\|_{\M} \big)h(t,X^{(R)}(t))\,{\rm d}\W(t),\\ X^{(R)}(0)=\,&X(0).
\end{aligned} 
\right.
\end{equation} 
To solve this problem, we consider its regularization equation for every $n\ge 1$: 
\begin{equation} \label{approximation cut problem}
\left\{
\begin{aligned}
{\rm d}X_n^{(R)}(t)=\,&\chi^2_R\big(\|X_n^{(R)}(t)-X(0)\|_{\M} \big)\left[b(t,X_n^{(R)}(t))+g_n(t,X_n^{(R)}(t))\right]\d t\\
&\qquad \qquad \qquad+\chi_R\big(\|X_n^{(R)}(t)-X(0)\|_{\M} \big)h_n(t,X_n^{(R)}(t))\,{\rm d}\W(t),\\ \ 
X_n^{(R)}(0)=\, &X(0).
\end{aligned} 
\right.
\end{equation}

\begin{Lemma}\label{Xn T estimates} Assume {\rm \ref{A}} and let $X(0)$ be an $\F_0$-measurable $\H$-valued random variable.
\begin{enumerate}[label={\bf (\arabic*)}]
\item \label{Global cut-off estimate} For any $R,n\ge1$, \eqref{approximation cut problem} has a unique global solution, which is continuous process on $\H$. Moreover, there exists a function $F: [0,\infty)\times [0,\infty)\to (0,\infty)$ increasing in both variables such that  for any $R>1$ and $T>0$,
\begin{align}
&\sup_{n\ge1}\E\Big[\sup_{t\in[0,T]}\|X_n^{(R)}(t)\|^2_{\H}\Big|\F_0\Big] 
\leq\, F(T,2R+\|X(0)\|_\M)(1+\|X(0)\|_\H^2).\label{Xn uniform bound} 
\end{align} 
\item \label{Global cut-off convergence} For any $n,l\ge 1$ and $N>0$, let
$$\tau_N^{n,l}(R):= N\land \inf\big\{t\ge 0: \|X_n^{(R)}(t)\|_\H\lor\|X_l^{(R)}(t)\|_\H\ge N\big\}.$$
Then $\p$-{\rm a.s.} 
\begin{equation} \label{Cauchy in E-M}
\lim_{n\rightarrow\infty}\sup_{l\ge n}
\E\bigg[\sup_{t\in[0,\tau_N^{n,l}(R)]} \|X_n^{(R)}(t)-X_l^{(R)}(t)\|^2_{\M}\bigg|\F_0\bigg]=0,\ \ N>0.
\end{equation}\end{enumerate} 
\end{Lemma}

\begin{proof}
\ref{Global cut-off estimate} By \ref{A1} and \ref{R2}, the drift and noise coefficients are locally Lipschitz continuous in $X\in\H$ locally uniform in $t$. 
Hence for any deterministic initial data, \eqref{approximation cut problem} has a unique solution, which is continuous in $\H$. See for instance \cite{Prevot-Rockner-2007-book,Wang-2013-Book}.  Since
$\F_0$ is independent of the equation, 
\eqref{approximation cut problem} also admits a unique solution $X_n^{(R)}(t)$ for any $\F_0$-measurable $\H$-valued random variable $X(0)$,  and $X_n^{(R)}(t)$ is continuous in $\H$ up to its lifespan $\tau_n(R)$ defined as
$$\tau_n(R):=\lim_{N\to\infty}\tau_{n,N}(R),\ \ 
\tau_{n,N}(R):=\inf\big\{t\ge 0: \|X_n^{(R)}(t)\|_{\H}\ge N\big\},\ \ N\ge 1.$$
It remains to prove $\tau_n(R)=\infty$ and the estimate \eqref{Xn uniform bound}. 
Let $$\tt K_{t,X(0)}:= K(t, \|X(0)\|_\M+2R),\ \ \ t\ge 0.$$
By \ref{R4} and It\^o's formula, we obtain 
\begin{align*} 
&{\rm d}\|X_n^{(R)}(t)\|^2_{\H}-\d M_n(t) \\
=\, &\chi^2_R \big(\|X_n^{(R)}(t)-X(0)\|_{\M} \big)\bigg\{ \left\|h(t,X_n^{(R)}(t))\right\|_{\LL_2(\U;\H)}^2\\
&\qquad\qquad\qquad\qquad\qquad\quad\ \   + 2\Big\<b(t,X_n^{(R)}(t))+ g_n(t,X_n^{(R)}(t)),\ X_n^{(R)}(t)\Big\>_{\H} \bigg\} \d t\\
\le\,& 2 \tt K_{t,X(0)} (1+\|X_n^{(R)}(t)\|_\H^2)\d t,\end{align*}
where 
$$ \d M_n(t):=\, 2\chi_R \big(\|X_n^{(R)}(t) - X(0)\|_\M\big) \left\<X_n^{(R)}(t), \ h_n(t,X_n^{(R)}(t))\d \W(t)\right\>_{\H}$$ 
satisfies 
$$\d\<M_n\>(t)\le 4 \tt K_{t,X(0)} (1+\|X_n^{(R)}(t)\|_\H^4\big)\d t.$$

So, by BDG's inequality, for any $T>0$, we find constants $c_1,c_2>0$  such that for any $ s\in [0,T]$ and $N\ge 1$, 
\begin{align*}
&\E\bigg[\sup_{t\in[0,s\land\tau_{n,N}(R)]}\|X_n^{(R)}(t)\|^2_{\H}\bigg|\F_0\bigg]- \|X(0)\|^2_{\H}\\
\le \,&
c_1 \E\bigg[\bigg(\int_0^{s\land\tau_{n,N}(R)}\tt K_{t,X(0)}
\Big(1+\|X_n^{(R)}(t)\|_\H^4\Big)\d t\bigg)^{\frac 1 2}\bigg|\F_0\bigg]\\  
&+ c_1 \E\bigg[\int_0^{s\land\tau_{n,N}(R)}\tt K_{t,X(0)} (1+\|X_n^{(R)}(t)\|_\H^2)\d t\bigg|\F_0\bigg]\\
\le\, & \frac 1 2 \E\Big[\sup_{t\in [0,s\land\tau_{n,N}(R)]} \|X_n^{(R)}(t)\|_\H^2\Big|\F_0\Big]\\
&+ c_2  + c_2 \int_0^s\tt K_{t,X(0)} \E\Big[\sup_{r\in [0,t\land \tau_{n,N}(R)]} \|X_n^{(R)}(r)\|_\H^2\Big|\F_0\Big]\d t.
\end{align*}
By Gr\"{o}nwall's inequality, there exists a function $F: [0,\infty)\times [0,\infty)\to (0,\infty)$ increasing in both variables such that for all $n,N\ge 1$,
\begin{align} 
\E\Big[
\sup_{t\in[0,T\land\tau_{n,N}(R)]}\|X_n^{(R)}(t)\|^2_{\H}\Big|\F_0\Big] 
\le\, F(T,2R+\|X(0)\|_\M)(1+\|X(0)\|_\H^2).\label{EXN}
\end{align} 
This implies that for all $n,N\ge1$,
\begin{align*}
\P(\tau_{n,N}(R)<T|\F_0)\le\,& \P\bigg(\sup_{t\in[0,T\land\tau_{n,N}(R)]}\|X_n^{(R)}(t)\|_{\H} 
\ge N\bigg|\F_0\bigg) \\
\le\, &\frac{F(T,2R+\|X(0)\|_\M)(1+\|X(0)\|_\H^2)}{N^2},
\end{align*} 
so that $\tau_n(R):=\lim_{N\to\infty}\tau_{n,N}(R)$ satisfies
$$ \P(\tau_n(R)<T|\F_0)\le \lim_{N\to\infty} \P(\tau_{n,N}(R)<T|\F_0) =0.$$
Hence, $\P(\tau_n(R)\ge T)= \E[\P(\tau_n(R)\ge T|\F_0)]=1$ for all $T>0$, which means  $\P(\tau_n(R)=\infty)=1$. By letting $N\to\infty$ in \eqref{EXN}, we prove 
\eqref{Xn uniform bound}.

\ref{Global cut-off convergence} Let $Z_{n,l}^{(R)}=X_n^{(R)}-X_l^{(R)}$ for $n,l\ge 1$. We have that
\begin{equation}\label{Z-n-m equation}
{\rm d}Z_{n,l}^{(R)}(t)=\sum_{i=1}^4 A_i^{n,l}(t)\d t+\sum_{i=1}^2 B_i^{n,l}(t) {\rm d}\W(t),\ \ Z_{n,l}^{(R)}(0)=0, 
\end{equation}
where
\begin{align*}
&A_1^{n,l}(t) := \left[\chi^2_R\left(\|X_n^{(R)}(t)-X(0)\|_{\M} \right)
-\chi^2_R(\|X_l^{(R)}(t)-X(0)\|_{\M})\right]b(t,X_n^{(R)}(t)),\\
&A_2^{n,l}(t):=\chi^2_R(\|X_l^{(R)}(t)-X(0)\|_{\M})\left[b(t,X_n^{(R)}(t))- b(t,X_l^{(R)}(t))\right],\\
&A_3^{n,l}(t):=\left[\chi^2_R\left(\|X_n^{(R)}(t)-X(0)\|_{\M}\right)
-\chi^2_R(\|X_l^{(R)}(t)-X(0)\|_{\M})\right]g_n(t,X_n^{(R)}(t)),\\
&A_4^{n,l}(t):= \chi^2_R(\|X_l^{(R)}(t)-X(0)\|_{\M}) \left[g_n(t,X_n^{(R)}(t))- g_l(t,X_l^{(R)}(t))\right],\end{align*} and 
\begin{align*}
&B_{1}^{n,l}(t):= \left[\chi_R(\|X_n^{(R)}(t)-X(0)\|_{\M})-\chi_R(\|X_l^{(R)}(t)-X(0)\|_{\M})\right]h_n(t,X_n^{(R)}(t)),\\
&B_2^{n,l}(t):= \chi_R(\|X_l^{(R)}(t)-X(0)\|_{\M})
[h_n(t,X_n^{(R)}(t))-h_l(t,X_l^{(R)}(t))]. 
\end{align*}
By the It\^{o} formula, we obtain 
\begin{align*}
\d \left\|Z_{n,l}^{(R)}(t)\right\|^2_{\M} = \, & 2\sum_{i=1}^2\left\<Z_{n,l}^{(R)}(t),\ B_i^{n,l}(t) \d\W(t)\right\>_{\M}\\
&+ \bigg\{\sum_{i=1}^2\left\|B_i^{n,l}(t)\right\|_{\mathcal L_2(\U;\M)}^2 + 2 \sum_{i=1}^4 \left\<A_i^{n,l}(t), Z_{n,l}^{(R)}(t)\right\>_\M\bigg\}\d t. 
\end{align*}
By the Lipschitz continuity of $ \chi_R$, \ref{R3}, \ref{A1} and \ref{A2}, we find a constant $C_N>0$ such that  for all $n,l\ge 1$ and $t\in [0,\tau_N^{n,l}]$,
\begin{align*}
\sum_{i=1}^2\sum_{k=1}^\infty \left\<Z_{n,l}^{(R)}(t), B_i^{n,l}(t) e_k\right\>_\M^2  
\le\,& C_NK(t,2N) \left\|Z_{n,l}^{(R)}(t)\right\|_\M^2
\left\{\lambda_{n,l}+\left\|Z_{n,l}^{(R)}(t)\right\|_\M^2\right\},
\end{align*}
\begin{align*}
\sum_{i=1}^2\left\|B_i^{n,l}(t)\right\|_{\mathcal L_2(\U;\M)}^2 + 2 \sum_{i=1}^4 \left\<A_i^{n,l}(t), Z_{n,l}^{(R)}(t)\right\>_\M 
\le\,& C_NK(t,2N)\left\{\lambda_{n,l}+\left\|Z_{n,l}^{(R)}(t)\right\|_\M^2\right\}.
\end{align*}
Therefore, by BDG's inequality to \eqref{Z-n-m equation}, we find constants $a_1,a_2>0$ depending on $N$ such that for all $n,l\ge 1$ and $s\in [0,N]$, 
\begin{align*} &\E\bigg[\sup_{t\in [0,s\land \tau_N^{n,l}(R)]} \left\|Z_{n,l}^{(R)}(t)\right\|_\M^2\bigg|\F_0\bigg]\\
\le\, & a_1 \E \bigg[\int_0^{s\land \tau_N^{n,l}(R)}K(t,2N)\Big\{\lambda_{n,l}+\left\|Z_{n,l}^{(R)}(t)\right\|_\M^2\Big\}\d t\bigg|\F_0\bigg]\\
&+ a_1 \E\bigg[\bigg(\int_0^{s\land \tau_N^{n,l}(R)}K(t,2N)\left\|Z_{n,l}^{(R)}(t)\right\|_\M^2\Big\{\lambda_{n,l}+\left\|Z_{n,l}^{(R)}(t)\right\|_\M^2\Big\}\d t\bigg)^{\frac 1 2}\bigg|\F_0\bigg]\\
\le\, & \frac 1 2 \E\bigg[\sup_{t\in [0,s\land \tau_N^{n,l}(R)]} \left\|Z_{n,l}^{(R)}(t)\right\|_\M^2\bigg|\F_0\bigg]+ a_2 \lambda_{n,l} \\
&+ a_2  \int_0^{s } K(t,2N)\E\bigg[\sup_{r\in [0,t\land\tau_N^{n,l}(R)]} \|Z_{n,l}^{(R)}(r)\|_\M^2\bigg|\F_0\bigg]\d t.\end{align*}
By Gr\"onwall's inequality and noticing $\lambda_{n,l}\to 0$ as $n,l\to\infty$, we prove \eqref{Cauchy in E-M}.
\end{proof} 

Next, we prove that up to a subsequence, $X_n^{(R)}$ converges to a process on $\H$.

\begin{Lemma}\label{Lemma:Convergence of X-n} There exists an $\mathcal{F}_t$-progressive measurable $\H$-valued process $X^{(R)}=(X^{(R)}(t))_{t\ge 0}$ such that 
\begin{equation}\label{X L2 bound}
\E\Big[\sup_{t\in [0,T]} \|X^{(R)}(t)\|_\H^2\Big|\F_0\Big]\le F(T,2R+\|X(0)\|_\M)(1+\|X(0)\|^2_{\H}),\ \ T>0.
\end{equation} 
Moreover, $\{X_n^{(R)}\}$ has a subsequence $($still labeled as $\{X_n^{(R)}\}$ for simplicity$)$ such that $\p$-{\rm a.s.},
\begin{equation}\label{Xn to X}
X_n^{(R)}\xrightarrow[]{n\to \infty}X^{(R)} \ {\rm in}\ C([0,\infty);\M).
\end{equation}
\end{Lemma}

\begin{proof} We first fix a time $T>0$. 
For any $N\ge T$ and $\epsilon>0$, by using \ref{Global cut-off estimate} in Lemma \ref{Xn T estimates} and Chebyshev's inequality, 
we obtain
\begin{align*} 
&\P(\tau_N^{n,l}(R)<T|\F_0)\\
\le\, & \P\bigg(\sup_{t\in [0,T]} \|X_n^{(R)}(t)\|_\H\ge N\bigg|\F_0\bigg)+ \P\bigg(\sup_{t\in [0,T]} \|X_l^{(R)}(t)\|_\H\ge N\bigg|\F_0\bigg)\\
\le\, & \frac {2 F(T,2R+\|X(0)\|_\M)(1+\|X(0)\|^2_{\H})}{N^2}.
\end{align*} 
Then for any $ N>T, n,l\ge 1,$ 
\begin{align*}
&\p\bigg(\sup_{t\in[0,T]}\|X_n^{(R)}(t)-X_l^{(R)}(t)\|_{\M}>\epsilon\bigg|\F_0\bigg)\\
\le\, & \P(\tau_N^{n,l}(R)<T|\F_0) + \p\bigg(\sup_{t\in[0,\tau_N^{n,l}(R)]}\|X_n^{(R)}(t)-X_l^{(R)}(t)\|_{\M}>\epsilon\bigg|\F_0\bigg)\\
\leq\, &\frac {2 F(T,2R+\|X(0)\|_\M)(1+\|X(0)\|^2_{\H})}{N^2}\\ &+\p\bigg(\sup_{t\in[0,\tau_N^{n,l}(R)]}\|X_n^{(R)}(t)-X_l^{(R)}(t)\|_{\M}>\epsilon\bigg|\F_0\bigg).
\end{align*} 
According to \ref{Global cut-off convergence} in Lemma \ref{Xn T estimates}, by first letting $n,l\to\infty$ and then $N\to\infty$, we obtain 
$$ \lim_{n,l\rightarrow\infty} \p\bigg(\sup_{t\in[0,T]}\|X_n^{(R)}(t)-X_l^{(R)}(t)\|_{\M}>\epsilon\bigg|\F_0\bigg)=0,\ \ \epsilon,\,T>0.$$ 
By Fatou's lemma, this implies 
\begin{align*} &\limsup_{n,l\rightarrow\infty} \p\bigg(\sup_{t\in[0,T]}\|X_n^{(R)}(t)-X_l^{(R)}(t)\|_{\M}>\epsilon\bigg)\\
= \, &\limsup_{n,l\rightarrow\infty} \E\bigg[\p\bigg(\sup_{t\in[0,T]}\|X_n^{(R)}(t)-X_l^{(R)}(t)\|_{\M}>\epsilon\bigg|\F_0\bigg)\bigg]\\
\le \, & \E\bigg[\limsup_{n,l\rightarrow\infty} \P\bigg(\sup_{t\in[0,T]}\|X_n^{(R)}(t)-X_l^{(R)}(t)\|_{\M}>\epsilon\bigg|\F_0\bigg)\bigg]=0,\ \ \epsilon,\, T>0.
\end{align*} 
Therefore, up to a subsequence, \eqref{Xn to X} holds for some progressively measurable process $X^{(R)}$ on $\H$, which together with 
\eqref{Xn uniform bound} implies \eqref{X L2 bound}. \end{proof} 

\begin{Lemma}\label{cut-off global solution} Assume {\rm \ref{A}} and let 
$X(0)$ be an $\H$-valued $\mathcal{F}_0$-measurable random variable. Then for any $R\ge 1$, $X^{(R)}$ given by Lemma \ref{Lemma:Convergence of X-n} is the unique global solution to \eqref{cut-off problem} and 
\begin{equation}\label{NI} \E\bigg[\sup_{t\in [0,T]} \|X^{(R)}(t)\|_\H^2\bigg|\F_0\bigg]<\infty,\ \  T>0.
\end{equation} 
\end{Lemma}

\begin{proof}
Obviously, \eqref{NI} follows from \eqref{X L2 bound}. 

To prove that $X^{(R)}$ solves \eqref{cut-off problem} in the sense of Definition \ref{solution definition},
by \ref{R3} we only need to show that $\p$-{\rm a.s.},
\begin{align}
\<X^{(R)}(t) -X(0),Y\>_\M 
=\,& \int_0^t \big\<A(s), Y\>_\M\d s+ \int_0^t \big\<B(s)\d\W(s), Y\big\>_\M,\ \ Y\in\M_0,\ \  t>0,\label{MN4} \end{align} 
where
\begin{align*} &A(s):= \chi^2_R\big(\|X^{(R)}(s)-X(0)\|_{\M}\big)\left\{b(s, X^{(R)}(s))+ g(s,X^{(R)}(s)) \right\}, \\
&B(s):= \chi_R\big(\|X^{(R)}(s)-X(0)\|_{\M}\big) h(s,X^{(R)}(s) ). \end{align*} 
Note that this formula holds for $(X_n^{(R)}, g_n, h_n)$ replacing $(X^{(R)}, g,h)$, i.e., 
for any $n\ge 1$,
\begin{equation}\label{MN5} 
\left\{\begin{aligned}
&\big\<X_n^{(R)}(t) - X(0), Y\big\>_\M =\, \int_0^t \big\<A_n(s), Y\big\>_\M\d s+\int_0^t \big\<B_n(s)\d \W(s), Y\big\>_\M,\\
&A_n(s):=\, \chi^2_R\big(\|X_n^{(R)}(s)-X(0)\|_{\M}\big)
\left\{b(s, X_n^{(R)}(s))+ g_n(s,X_n^{(R)}(s)) \right\},\\
&B_n(s):= \,\chi_R\big(\|X_n^{(R)}(s)-X(0)\|_{\M}\big) h(s,X_n^{(R)}(s) ).
\end{aligned}\right.
\end{equation} 
By \ref{R3}, \ref{A1} and \eqref{Xn to X}, we have 
$$\lim_{n\to\infty} \bigg\{\Big|\big\<A_n(t)- A(t), Y\big\>_\M\Big|+ \sum_{k\ge 1} \big\<\{B_n(t)-B(t)\}e_k, Y\big\>_\M^2\bigg\} =0,\ \ t\ge 0\ \ \p\text{-a.s.}$$
By \eqref{Xn uniform bound} and the fact that $\chi_R(r)=0$ for $r\ge 2R$,  we find random variables $C_1=C_1(t,2R+\|X_0\|_\M)>1$ and $C_2=C_2(t,2R+\|X_0\|_\M)>1$ such that  
\begin{align}
&\E\bigg[\sup_{ s\in [0,t]} \Big\{\big\<A_n(s),Y\>_\M^2 + \sum_{k\ge 1} \big\<B_n(s)e_k, Y\big\>_\M^2\Big\} \bigg|\F_0\bigg]\notag\\
\le\, &C_1+C_1 \E\bigg[\sup_{s\in [0,t]} \|X_n^{(R)}(s)\|_\H^2\bigg|\F_0\bigg]\le C_2,\ \ n\ge 1.\label{MN7}
\end{align}
So, by the dominated convergence theorem we derive 
\begin{align}
\lim_{n\to\infty}\E \bigg\{
&\int_0^t\Big[ \big|\<A_n(s)-A(s),Y\>_\M\big|\notag\\
&+ \Big(\sum_{k\ge 1} \big\<\{B_n(s)-B(s)\}e_k, Y\big\>_\M^2\Big)^{\frac 1 2} \Big]\d s\bigg|\F_0\bigg\}=0,\label{MN8} 
\end{align}
and in the conditional probability $\P(\cdot|\F_0)$ (equivalently, in $\P$), 
\begin{align*} 
&\lim_{n\to\infty} \int_0^t \sum_{k\ge 1} \Big\<\{B_n(s)-B(s)\}e_k, Y\Big\>_\M^2 \d s \\
\le\,& \lim_{n\to\infty}\sqrt{C+C \sup_{s\in [0,t]} \|X_n^{(R)}(s)\|_\H^2} \int_0^t \Big(\sum_{k\ge 1} \Big\<\{B_n(s)-B(s)\}e_k, Y\Big\>_\M^2\Big)^{\frac 1 2} \d s=0.\end{align*} 
From this, \eqref{MN7}, BDG's inequality and the dominated convergence theorem, we find a constant $c>1$ such that 
\begin{align*}&\lim_{n\to\infty}\E\bigg[\sup_{s\in [0,t]} \bigg|\int_0^s \big\<\{B_n(r)-B(r)\}\d\W(r), Y\big\>_\M\bigg| \bigg|\F_0\bigg]\\
\le\,& c \lim_{n\to\infty} \E\bigg[\bigg(\int_0^t\sum_{k\ge 1} \big\<\{B_n(s)-B(s)\}e_k, Y\big\>_\M^2\d s\bigg)^{\frac 1 2}
\bigg|\F_0\bigg] =0.\end{align*}
Combining this, \eqref{Xn to X} and \eqref{MN8}, we prove \eqref{MN4} by letting $n\to\infty$ in \eqref{MN5}. 

To prove the uniqueness, by \ref{R1} and \ref{A2}, we may take $n=m\to\infty$ to derive that
\begin{align} 
2\left\<g(t,X)-g(t,Y), X-Y \right\>_{\M}+&
\| h(t,X)-h(t,Y)\|^2_{\LL_2(\U;\M)} \notag\\
\le\, & K(t,\|X\|_\H+ \|Y\|_\H)\|X-Y\|^2_{\M},\ \ t\ge 0,\ \ X,\, Y\in \H.\label{LN1} 
\end{align}
If $\tt X$ is another solution to \eqref{cut-off problem} with $\tt X(0)=X(0)$ such that \eqref{NI} holds for $\tt X$ replacing $X$, by \eqref{LN1} and It\^o's formula for $\|X^{(R)}(t)- \tt X(t)\|_\M^2$, we prove $\tt X(t)= X^{(R)}(t)$ for all $t\ge 0$. 
\end{proof}
We are now ready to prove the main result in this section.
\begin{proof}[Proof of Theorem \ref{T1}] \ref{T1-existence} By Lemma \ref{cut-off global solution}, for any $R\ge 1$, \eqref{cut-off problem} with initial value $X(0)$ has a unique global solution $X^{(R)}$ satisfying \eqref{NI}. Let 
\begin{equation}
\tau(R):= \inf\big\{t\ge 0: \|X^{(R)}(t)-X(0)\|_\M\ge R\big\}.\label{tau-R-M}
\end{equation}
By the continuity of $X^{(R)}(t)$ in $\M$, we have $\P(\tau(R)>0)=1$ for any $R>0$. Since $\chi_R(\|X^{(R)}(t)-X(0)\|_\M)=1$ for $t\le \tau(R)$, \eqref{E1} coincides with 
\eqref{cut-off problem} up to time $\tau(R)$. This together with \eqref{NI} implies that 
$(X^{(R)},\tau(R))$ is a local solution to \eqref{E1}. By the uniqueness of \eqref{cut-off problem}, we see that $\tau(R)$ is increasing in $R$, and 
$$X^{(R)}(t)= X^{(R+1)}(t),\ \ t\le \tau(R), \ R\ge 1\ \ \p\text{-a.s.}$$
Let $\tau^* :=\lim_{R\to\infty} \tau(R)$, $\tau(0):=0$ and we define
\begin{align} 
X(t):=\sum_{R=1}^\infty \textbf{1}_{[\tau(R-1), \tau(R))}(t) X^{(R)}(t),\ \ t\in [0,\tau^*).\label{tau-R-M-X} 
\end{align}
Then one can conclude that $(X,\tau^*)$ is a local solution to \eqref{cut-off problem}. Moreover, according to the definitions of $\tau^*$ and $\tau(R)$, we have $\p$-{\rm a.s.}, $$\limsup_{t\to\tau^*}\|X(t)\|_\M=\infty\ \text{on}\ \{\tau^*<\infty\},$$ so that it is actually a maximal solution. 
Uniqueness follows from \eqref{LN1} and It\^o's formula for $\|X(t)-\tt X(t)\|_\M^2$ if $(\tt X, \tt\tau^*)$ is another maximal solution with initial value $X(0)$. 

\ref{T1-conti} 
Let $(X,\tau^*)$ be the unique maximal solution as above. 
Since $X\in C([0,\tau^*);\M)$ and $\H\hookrightarrow\M$ is dense, $X$ is weakly continuous in $\H$ (cf. \cite[page 263, Lemma 1.4]{Temam-1977-book}), so that 
$\|X(\cdot)\|_\H$ is lower semi-continuous. Thus, 
\begin{equation}\label{TNN} \tau_N:=N\land \inf\big\{t\ge 0: \|X(t)\|_\H\ge N\big\},\ \ N\ge 1\end{equation} are stopping times and 
$\tau^*=\lim_{N\to\infty} \tau_N.$ 
By the weak continuity of $X$ in $\H$, it suffices to prove 
\begin{equation}\label{CTT} 
\|X(\cdot)\|_\H\in C([0,\tau_N]),\ \ \ N\ge 1.
\end{equation} 
Let $K_N=K(N,N)$ for $K$ in \ref{B}. By  It\^{o}'s formula, for any $n\ge 1$ we find a martingale $M_t^{(n)}$ such that  for all $n\ge1$,
\begin{equation}\label{C0T}
\left\{\begin{aligned}
&\d \<M^{(n)}\>(t)\le K_N \d t,\ \ t\in [0,\tau_N],\\
-K_N \d t\le & \d \|T_n X(t)\|_\H^2 + \d M^{(n)}(t) \le K_N \d t,\ \ t\in [0,\tau_N].
\end{aligned}\right.
\end{equation}
Then there exists a constant $C_N>0$ such that 
$$\E\Big[\big|\|T_n X(t\land\tau_N)\|_\H^2-\|T_n X(s\land\tau_N)\|_\H^2\big|^4\Big]\le C_N|t-s|^2,\ \ \ t,s\ge 0,\, n\ge 1.$$
By \ref{B} and Fatou's lemma with $n\to\infty$, we derive 
$$\E\Big[\big|\|X(t\land\tau_N)\|_\H^2-\|X(s\land \tau_N)\|_\H^2\big|^4\Big]\le C_N|t-s|^2,\ \ t,s\ge 0.$$
This together with Kolmogorov's continuity theorem proves \eqref{CTT}.

\ref{T1-global} 
By It\^o's formula for $\|X(t)\|_\M^2$, we have 
\begin{align*} \d \|X(t)\|_\M^2=\,& 2\left\<h(t,X(t))\d \W(t), X(t)\right\>_\M \\
&+ 2\left\<g(t,X(t))+b(t,X(t)), X(t)\right\>_\M\d t\\
&+ \|h(t,X(t))\|_{\mathcal L_2(\U;\M)}^2 \d t,\ \ t\in [0,\tau^*). \end{align*} By \ref{C} and It\^o's formula, this implies 
$$\d V(\|X(t)\|_\M^2)\le F(t) V(\|X(t)\|_\M^2)\d t+\d M_t,\ \ t\in [0,\tau^*),$$
where $M_t$ is a martingale up to $\tt\tau_{N}$ (for any $N\ge1$) defined as
\begin{equation*}
\tt\tau_N:=N\land \inf\big\{t\ge 0: \|u(t)\|_{\M}\ge N\big\},\ \ N\ge 1.
\end{equation*}
Thus, 
$$ \E \big[V(\|X(t\land\tt\tau_{N})\|_{\M}^{2})\big|\F_0\big]\leq V(\|X(0)\|_{\M}^{2}){\rm e}^{\int_0^tF(s)\d s},\ \ t\ge 0.$$
Accordingly, by the continuity in $\M$ of $X(t)$ (see \eqref{Xn to X}), we derive that $\tt\tau^*:=\lim_{N\to\infty}\tt\tau_{N}$ satisfies
\begin{align*} 
\p\big(\tt\tau^*<t\big|\F_0\big)  \le\, & \p\big(\tt\tau_{N}<t\big|\F_0\big)\\ 
\le \,&\frac{ \E\big[V(\|X(t\land \tt\tau_{N})\|_\M^2)\big|\F_0\big]}{V(N^2)} \le \frac{ V(\|X(0)\|_{\M}^{2}){\rm e}^{\int_0^tF(s)\d s}}{V(N^2)},\ \ N\ge 1,\ t>0.
\end{align*} 
Letting $N\uparrow\infty$ then $t\uparrow\infty$ yields
$\p(\tt\tau^*<\infty|\F_0)=0$. However, we infer from \eqref{Blow-up criterion} that $\tau^*=\tt\tau^*$ $\p$-a.s., hence we obtain $\p(\tau^*<\infty)=0.$ 
\end{proof} 
%

\subsection{Improved blow-up criterion}\label{Subsection-T-remark} 

According to Definition \ref{solution definition}, for the maximal solution $(X,\tau^*)$ we have $\limsup_{t\uparrow\tau} \|X(t)\|_{\H}=\infty$ on $\{\tau^*<\infty\}$.
Since $\|\cdot\|_\M\lesssim \|\cdot\|_\H$, 
\eqref{Blow-up criterion} is a criterion of blow-up, i.e., $X$ blows up in $\H$ if and only if it blows up in $\M$, which has been used in the proof for \ref{T1-global}.
The following result further improves this criterion. Particularly, if in the following $\B(\cdot)=\|\cdot\|_{\mathbb B}$ with $\M\hookrightarrow\mathbb B$ for some Banach space $\mathbb B$, then $X$ blows up in $\H$ if and only if it blows up in $\mathbb B$.

\begin{Proposition}\label{P1}
Assume {\rm \ref{A}} and let $\B(\cdot)\in C(\M;[0,\infty))$ be a subadditive function such that $$\B(X)\lesssim \|X\|_{\M}.$$
If the growth factor $K(t,\|X\|_\M)$ in \ref{R4} and \ref{A1} is replaced by $K\left(t,\B(X)\right)$,
then $\p$-{\rm a.s.} we have
\begin{equation}\label{Blow-up criterion V}
\limsup_{t\rightarrow \tau^*}\B(X)=\infty \ \text{on}\ \{\tau^*<\infty\}.
\end{equation}
\end{Proposition}

\begin{proof} For any $R\ge 1$, let 
$$\tau_{\B}(R):=\inf\big\{t\ge 0: \B(X)\ge R\big\}.$$ 
Now we consider the cut-off problem \eqref {cut-off problem} with $\B\big(X^R(t)-X^R(0)\big)$ replacing $\|X^{(R)}(t)-X(0)\|_{\M}$. 
Notice that the
subadditivity of $\B(\cdot)$ implies 
$\B(Y)\leq \B(Y-Z)+\B(Z).$
Using
this, \ref{R4} and \ref{A1} with $K\left(t,\B(X)\right)$ being replaced by $K(t,\|X\|_\M)$ in the proof of \eqref{Xn uniform bound}, we find
$$\sup_{n\ge 1} \E\bigg[\sup_{t\in [0,T\land\tau_{\B}(R)] } \|X_n^{(R)}(t)\|_\H^2\bigg|\F_0\bigg]\le F\Big(T, 2R+\B(X(0))\Big)(1+\|X(0)\|_\H^2),\ \ T, R\ge 1.$$
By  the
subadditivity of $\B(\cdot)$ again, $\B(X)\lesssim \|X\|_{\M}$ and the construction of $X(t)$ (cf. Lemma \ref{Xn T estimates} \ref{Global cut-off convergence},  \eqref{Xn to X} and \eqref{NI}), we derive 
$$\E\bigg[\sup_{t\in [0,T\land\tau_{\B}(R)] } \|X(t)\|_\H^2\bigg|\F_0\bigg]\le F\Big(T, 2R+\B(X(0))\Big)(1+\|X(0)\|_\H^2),\ \ T,R\ge 1.$$
Since $\limsup_{t\uparrow \tau^*} \|X(t)\|_\H=\infty$ on $\{\tau^*<\infty\}$, this implies $\tau_\B(R)\le \tau^*$ for any $R\ge 1$, so that by
the definition of $\tau_\B(R)$, we obtain \eqref{Blow-up criterion V}.
\end{proof}

\section{SPDE with pseudo-differential noise}\label{Section:SPDEs} 

In this part, we apply Theorem \ref{T1} to nonlinear SPDE \eqref{EN} with pseudo-differential noise on $\tt\Pi H^s(\mathbb K^d;\R^m)$ for some $s>\frac{d}{2}+1$. To this end,
we first recall some notions and preliminary results on pseudo-differential operators, then state the main results and finally present a proof.

\subsection{Notations and preliminary results}\label{Section:Notations}

\subsubsection{Pseudo-differential operators} We mainly focus on two cases:  $\mathbb K=\R$ and $\mathbb K=\T$.

{\bf Case of $\mathbb K=\R$.}  Let $d,m\in \mathbb N$.
Recall the Fourier and inverse Fourier transforms: 
$$(\mathscr Ff)(\xi):=\int_{\mathbb R^d}f(x){\rm e}^{-{\rm i}(x\cdot \xi)}{\rm d}x,\ \ 
(\mathscr F^{-1}f)(x)= \frac{1}{(2\pi)^d}\int_{\mathbb R^d}f(\xi){\rm e}^{{\rm i} (x\cdot \xi)}{\rm d}\xi,\ \ x,\,\xi\in\mathbb R^d,$$ 
where ${\rm i}=\sqrt{-1}$ as before, and $(x\cdot\xi):=\sum_{i=1}^m x_i\xi_i$. 
Let $\alpha=(\alpha_1,\cdots,\alpha_d)\in \N_0^d:=(\N\cup\{0\})^d$ be a multi-index and recall that $\pp_k$ is the $k$-th partial derivate in $\mathbb R^d$. 
We define
$$|\alpha|_1:=\sum_{k=1}^d \alpha_k\ \ \partial^\alpha:= \prod_{k=1}^d \partial_{k}^{\alpha_k}.$$ 
When both space variable $x$ and frequency variable $\xi$ appear, we use 
$\pp_{x}^{\alpha}$ and $\pp_{\xi}^{\alpha}$ to denote $\pp^{\alpha}$ in $x$ and $\xi$, respectively. 
Next, we will introduce the class of nice functions called symbols, from which we can define pseudo-differential operators. 
For any $s\in \R$ and $d,m\ge1$, we define the class of symbols as
\begin{align}
\SS^s\big(\R^d\times \R^d;\mathbb C^{m\times m}\big)
:=\bigg\{&\wp\in C^\infty\big(\R^d\times \R^d;\mathbb C^{m\times m}\big)\, : \,\notag\\ &\sup_{(x,\xi)\in\R^d\times \R^d}\frac{|\pp_{x}^{\beta} \pp_{\xi}^{\alpha} {\wp }(x, \xi)|_{m\times m}}{(1+|\xi|)^{ s-|\alpha|_1}}<\infty,\  \beta,\, \alpha\in \N_0^d\bigg\},\label{Symbol}
\end{align}   
where, and in the sequel, $|\cdot|_{m\times m}$ and $|\cdot|$ are usual norms in $\mathbb C^{m\times m}$ and $\R^d$, respectively. 
For any $\wp\in \SS^s\big(\R^d\times \R^d;\mathbb C^{m\times m}\big)$, the pseudo-differential operator $\OP(\wp)$ with symbol $\wp$ is defined as
$$[\OP(\wp)f](x):=\frac{1}{(2\pi)^d} \int_{\mathbb R^d} \wp(x,\xi) (\mathscr Ff)(\xi) {\rm e}^{ {\rm i} (x \cdot \xi)} \d \xi,\ \ \ x\in\mathbb R^d.$$ 
Throughout this paper, we focus on real-valued operators, i.e., $[\OP({\wp})f]$ is real for real function $f$. Equivalently, as in \eqref{tt Pi pi0}, we require
\begin{equation}\label{RE} 
\wp(x,-\xi)=\overline{ \wp(x,\xi)},\ \ (x,\xi)\in\R^d\times \R^d.
\end{equation}
To highlight the difference between symbols depending on $(x,\xi)$ and symbols depending only on $\xi$, we use $\R^d_{x}$ and $\R^d_{\xi}$ to distinguish the space $\R^d$ for $x$ and $\xi$, respectively, and then $\SS^s\big(\R^d\times \R^d;\mathbb C^{m\times m}\big)$ defined in \eqref{Symbol} is relabeled as
$\SS^s\big(\R_x^d\times \R_{\xi}^d;\mathbb C^{m\times m}\big)$. 
Then we define 
$$\SS\big(\R_{\xi}^d;\mathbb C^{m\times m}\big):=\Big\{\wp\in \SS^s\big(\R_x^d\times \R_{\xi}^d;\mathbb C^{m\times m}\big)\, : \, \wp(x,\xi)=\wp(\xi)\Big\}.$$
To simplify notations, when $d,m$ are clear in the context and no confusion can arise, we will write
\begin{equation}\label{Ss Real}
\left\{\begin{aligned}
\SS^s:=\,&\Big\{\wp\in \SS^s\big(\R_x^d\times \R_{\xi}^d;\mathbb C^{m\times m}\big)\ : \ \eqref{RE} \ \text{holds}\Big\},\\ 
\SS_0^s:=\,& \Big\{\wp\in \SS^s\big(\R_{\xi}^d;\mathbb C^{m\times m}\big)\ : \ \eqref{RE} \ \text{holds}\Big\}.
\end{aligned}\right.
\end{equation}
Then, for any $s\in\R$, we define
\begin{align}
\OP\SS^s:=\Big\{\OP({\wp })\, : \, {\wp } \in\SS^s\Big\},\ \ \OP\SS_0^s:= \Big\{\OP({\wp })\, : \, {\wp } \in \SS_0^s\Big\}.\label{OPS Real}
\end{align}
It is well-known that $\SS^s$ is a Fr\'{e}chet space equipped with the topology generated by seminorms 
$\big\{|\cdot|^{\R^d\times\R^d}_{\beta,\alpha;s}\big\}_{\beta,\alpha\in \N_0^d}$, where
$$|\wp|^{\R^d\times\R^d}_{\beta,\alpha;s}:=\sup_{(x,\xi)\in\R^d_{x}\times \R^d_{\xi}}\frac{|\pp_{x}^{\beta} \pp_{\xi}^{\alpha} {\wp }(x, \xi)|_{m\times m}}{(1+|\xi|)^{ s-|\alpha|_1}}.$$
A set $S\subset \SS^{s}$ is called bounded, if 
$\sup_{\wp \in S} |\wp|^{\R^d\times\R^d}_{\beta,\alpha;s} <\infty, \ \ \beta,\alpha\in \N_0^d.$ 

Let $\LL(E_1;E_2)$ be the space of bounded linear operators between two 
Banach spaces $E_1$ and $E_2$.  Although $\OP\SS^s$ can be measured by operator norm $\LL(H^q;H^{q-s})$ (see Lemma \ref{LOP2}), it is   convenient to consider the boundedness for  subsets of
$\OP\SS^s$ in terms of symbols, see for instance Lemmas \ref{LOP2}, \ref{LOP4} and \ref{LOP5} below. Indeed, since $\wp\mapsto\OP(\wp)$ is one-to-one,  $\OP\SS^s$ is also a
Fr\'{e}chet space with seminorms $\{|{\rm OP}({\wp })|^{\R^d\times\R^d}_{\beta,\alpha;s}:=|{\wp }|^{\R^d\times\R^d}_{\beta,\alpha;s}\}_{\beta,\alpha\in \mathbb N_0^d}$, 
and similarly one can define the boundedness in $\OP\SS^s$. 

To emphasize the scalar symbols (i.e., $m=1$ in \eqref{Symbol}), as in \eqref{Ss Real}, we simply write
\begin{align*}
\S^s=\Big\{\wp\in \SS^s\big(\R_x^d\times \R_{\xi}^d;\mathbb C\big)\, : \, \eqref{RE} \ \text{holds}\Big\},\ \ \S_0^s:= \Big\{\wp\in \SS^s\big(\R_{\xi}^d;\mathbb C\big)\, : \, \eqref{RE} \ \text{holds}\Big\}. 
\end{align*}
Then, 
as in \eqref{OPS Real}, we denote by $\OP\S^s$ and $\OP\S_0^s$ the pseudo-differential operators with scalar symbols in $\S^s$ and $\S_0^s$, respectively.
\medskip

{\bf Case of $\mathbb K=\T$.} 
On torus, the Fourier and inverse Fourier transforms are defined as
$$(\mathscr Ff) (k):=\int_{\mathbb T^d}f(x){\rm e}^{-{\rm i}(x\cdot k)}{\rm d}x,\ \ 
(\mathscr F^{-1}f)(x)=  \frac{1}{(2\pi)^d}\sum_{k\in\Z^d}f(k){\rm e}^{{\rm i} (x\cdot k)},\ \ x\in\mathbb T^d,\ k\in\Z^d.$$

Now we recall some facts about (toroidal) symbol class and (toroidal) pseudo-differential operators. 
For two multi-indexes $\alpha, \beta$ with $\beta\le\alpha$,   
let 
$$\left(\begin{matrix}\alpha \\ \beta\end{matrix}\right):=\prod_{i=1}^d\left(\begin{matrix}\alpha_{i} \\ \beta_{i}\end{matrix}\right)
=  \prod_{i=1}^d\frac{\alpha_i!}{\{\beta_i! \}\cdot\{(\alpha_i-\beta_i)!\}} .$$
Then for any $\alpha\in \mathbb N_0^d$, the   partial difference operator $\triangle^\alpha$ for   a function  $g:\mathbb{Z}^{d}\to \mathbb{C}$ is given by 
$$   (\triangle^\alpha_{k}g)(k):=\displaystyle\sum_{\gamma\in\N_0^{d}, \gamma\leq \alpha} (-1)^{|\alpha-\gamma|_1}\binom{\alpha}{\gamma}g(k+\gamma),\ \ k\in \mathbb Z^d.$$
Then   the (toroidal) symbol class of order $s$ for $s\in\R$ is  defined as (cf. \cite{Ruzhansky-Turunen-2010-JFAA}):
\begin{align*}
\SS^{s}(\T^d\times\mathbb Z^d;\mathbb{C}^{m\times m}):=\bigg\{&\wp(\cdot,k)\in  C^\infty (\T^d;\mathbb{C}^{m\times m}),\, k\in\Z^d\, : \,\notag\\
&\sup_{(x,k)\in \T^d\times \Z^d}  \frac{|\triangle_{k}^{\alpha} \partial_{x}^{\beta} \wp(x, k)|_{m\times m}}{(1+|k|)^{ s-|\alpha|_1}}<\infty,\, \beta,\, \alpha\in \mathbb N_0^d\bigg\}.
\end{align*}
Similar to the  case $\mathbb K=\R$, $\SS^{s}(\T^d\times\mathbb Z^d;\mathbb{C}^{m\times m})$ is also a Fr\'{e}chet space equipped with the topology generated by seminorms $\big\{|\cdot|^{\T^d\times\Z^d}_{\beta,\alpha;s}\big\}_{\beta,\alpha\in \N_0^d}$ with 
$$|\wp|^{\T^d\times\Z^d}_{\alpha,\beta;s}:=\sup_{(x,k)\in \T^d\times \Z^d}  \frac{|\triangle_{k}^{\alpha} \partial_{x}^{\beta} \wp(x, k)|_{m\times m}}{(1+|k|)^{ s-|\alpha|_1}}.$$ 
The (toroidal) pseudo-differential operator   with symbol $\wp$ is defined by
\begin{equation*}
[\OP(\wp)f](x):= \displaystyle\frac{1}{(2\pi)^d}\sum_{k\in \Z^d} \wp(x,k) 
(\mathscr Ff)(k){\rm e}^{ {\rm i}(x\cdot k)} ,\ \ x\in \T^d.
\end{equation*}
Similarly, we consider real-valued operators such that \eqref{RE} holds true with $k\in\Z^d$ replacing $\xi\in\R^d$. By convenient abuse of notation, 
$\SS^s$, $\SS_0^s$ also denote symbols on $\T^d\times \Z^d$ or $\Z^d$ as in \eqref{Ss Real}, and $\OP\SS^s$ and $\OP\SS_0^s$ stand for the corresponding real pseudo-differential operators and \eqref{OPS Real} as in \eqref{Ss Real}. Similarly, when $m=1$ we denote these classes as  $\S^s$, $\S_0^s$, $\OP\S^s$ and $\OP\S_0^s$ respectively.

\begin{Remark} \label{OPS R T}
We write $\wp\in \S^s\big(\T^d\times \R^d;\mathbb C\big)$ if $\wp\in \S^s\big(\R_x^d\times \R_{\xi}^d;\mathbb C\big)$ and $\wp(\cdot,\xi)$ is periodic with period $2\pi$ for all $\xi$. 
Moreover, according to \cite[Theorem 5.2]{Ruzhansky-Turunen-2010-JFAA} (see also \cite[Corollary 2.11]{Delgado-2013-Chapter}), we see that if,
\begin{equation}
|\wp|^{\R^d\times\R^d}_{\beta,\alpha;s} \leq c_{\alpha \beta s}\ \ \text{for\ some}\ c_{\alpha \beta s}>0,\ \ |\alpha|<N_1,\ |\beta|< N_2,\ \ N_1,\, N_2\ge1,\label{wp R bound}
\end{equation}
then $\tt\wp=\wp\big|_{\mathbb{T}^n\times \mathbb{Z}^n}$ satisfies 
\begin{equation}
|\tt\wp|^{\T^d\times\Z^d}_{\beta,\alpha;s} \leq C_{\alpha \beta s}\ \ \text{for\ some}\ C_{\alpha \beta s}>0,  \ \ |\alpha|<N_1,\ |\beta|<N_2,\ \ N_1,\, N_2\ge1.\label{wp T bound}
\end{equation}
Conversely, every symbol $\tt\wp \in \S^s\big(\T^d\times \Z^d;\mathbb C\big)$ satisfies \eqref{wp T bound} is a restriction $\tt\wp=\wp\big|_{\mathbb{T}^n \times \mathbb{Z}^n}$ of a symbol $\wp \in \S^s\big(\T^d\times \R^d;\mathbb C\big)$, where $\wp$ satisfies \eqref{wp R bound}.
Therefore we see that
$\OP\S^s\big(\T^d\times \Z^d;\mathbb C\big)=\OP\S^s\big(\T^d\times \R^d;\mathbb C\big)$
and any bounded  set in 
$\OP\S^s\big(\T^d\times \Z^d;\mathbb C\big)$  coincides with the restriction of a bounded  set in  $\OP\S^s\big(\T^d\times \R^d;\mathbb C\big)$ (see also \cite[Theorem 2.10 and Corollary 2.11]{Delgado-2013-Chapter}). By considering each element in a matrix, this also holds true for matrix-valued symbol, i.e., $\OP\SS^s$. 
Moreover, the following introduced  Lemmas \ref{LOP2}-\ref{LOP5}, which are known  for pseudo-differential operators on $\R^d$, also hold for those on $\T^d.$ 
\end{Remark}

%
%
%
%
%
%

\subsubsection{Function spaces and related estimates}

Remember that $\mathbb K=\R$ or $\T$. Let $k\ge1$. The space $W^{k,\infty}(\mathbb K^d;\R^m)$ is the set of weakly differential functions $f:\mathbb K^d\to\R^m$ such that 
$$\|f\|_{W^{k,\infty}}:= \sum_{j=1}^m\sum_{0\le |\alpha|_1\le k}\|\pp^{\alpha} f_j\|_{L^\infty} <\infty.$$
Here, and in the following, the weak derivatives $\pp^{\alpha} f$ is understood in the sense of distribution:
$$\left\<\pp^{\alpha} f,g\right\>_{L^2}=(-1)^{|\alpha|_1}\left\< f,\pp^{\alpha} g\right\>_{L^2},\ \ g\in C_0^\infty(\mathbb R^d)\ \text{if}\ \mathbb K=\R\ \text{and}\ g\in C^\infty(\mathbb T^d)\ \text{if}\ \mathbb K=\T.$$
As before we denote by $\I$ the identity mapping. Then 
for any $s\in\R$, we define
\begin{equation}\label{Ds Lambda s define}
\D^s=\left(\I-\Delta\right)^{\frac s2}:=\mathscr F^{-1}\left((1+|\cdot|^2)^{\frac s 2}\I\right)\mathscr F\ \text{and}\ 
\Lambda^{s}=\left(-\Delta\right)^{\frac s2}:=\mathscr F^{-1}\left(|\cdot|^s\I\right)\mathscr F,
\end{equation} 
respectively, where the function $|\cdot|$ is defined on $\R^d$
if $\mathbb K=\R$, and on $ \Z^d$ if $\mathbb K=\T$.
These two operators  are self-adjoint  and pseudo-differential operators if $s\ge0$. For $s\ge0$, $d,m\ge1$, the Sobolev space $H^s(\mathbb R^d;\R^m)$ and $H^s(\mathbb T^d;\R^m)$ are defined as the completion of $C_0^\infty(\mathbb R^d;\R^m)$ and $C^\infty(\mathbb T^d;\R^m)$, respectively, under the inner product
$$\<X, Y\>_{H^s} := \<\D^s X, \D^s Y\>_{L^2}= \sum_{j=1}^m \int_{\mathbb K^d} \big[ (\D^s X_j)(\D^s Y_j) \big](x)\d x.$$
For $s\ge 0$ and $d\ge2$, we denote by $H^s_{{\rm div}}(\mathbb K^d;\R^d)$ a subset of $H^s(\mathbb K^d;\R^d)$ such that
\begin{equation}\label{H-div}
\left\{ \begin{aligned}
H^s_{{\rm div}}(\mathbb R^d;\R^d):=\,&\Big\{f\ : \ \nn\cdot f=0\Big\},\\ 
H^s_{{\rm div}}(\mathbb T^d;\R^d):=\,&
\bigg\{f\ : \ \nn\cdot f=0,\  \int_{\T^d} f(x)\d x=0\bigg\},
\end{aligned}\right.
\end{equation}
where $\nn\cdot f$ is also understood in the sense of distribution as mentioned above.
Notice that $H^s_{{\rm div}}(\mathbb K^d;\R^d)$ is a closed subspace of $H^s(\mathbb K^d;\R^d)$.  The Leray projection $\Pi_d$ on $\mathbb K^d$ is defined by
matrix-valued multiplier $\varPi$:
\begin{equation}\label{Pi-d define}
\Pi_d:=\mathscr{F}^{-1}\varPi \mathscr{F}, \ \  \varPi :=(\varPi_{i,j} )_{1\le i,j\le d},
\ \ \varPi_{i,j}(\xi):=\delta_{i,j}-\frac{\xi_i \xi_j}{|\xi|^2},\ \ 1\le i,j \le d,
\end{equation} 
where $ \xi\in\R^d$
if $\mathbb K=\R$,  $\xi \in\Z^d$ if $\mathbb K=\T$, and
$\delta_{i,j}$ is the Kronecker delta.

When $\mathbb K=\R$,  we have 
\begin{equation}\label{Pi-d bound}
\Pi_d\in \LL\big(H^s(\mathbb R^d;\R^d);H^s_{{\rm div}}(\mathbb R^d;\R^d)\big).
\end{equation}

When $\mathbb K=\T$, the Laplacian has a spectral gap.  To make it invertible  we restrict to the zero-average sub-space $H_0^s(\T^d;\R^d)$ of $H^s(\T^d;\R^d)$,
where in general,  
\begin{equation}\label{Pi-0 Hs}H_0^s(\T^d;\R^m):=\bigg\{f\in H^s(\T^d;\R^m):\ \int_{\T^d} f(x)\d x=0\bigg\},\ \ d, m\in\N.\end{equation} 
For any $m\ge 1$, the natural zero-average projection on $L^2(\T^d;\R^m)$ is 
\begin{equation}\label{Pi-0 define} 
\Pi_0 f(x):= 
f(x)- \frac{1}{(2\pi)^d}\int_{\T^d} f(y)\d y,\ \ x\in \T^d,\ f\in L^2(\T^d;\R^m).
\end{equation}
It is easy to see that  $\Pi_d$ defined in \eqref{Pi-d define} leaves  $H_0^s(\T^d;\R^d)$ invariant, so that  
\begin{equation}\label{Pi-0 bound}
\Pi_d\Pi_0\in \LL\big(H^s(\mathbb T^d;\R^d); H^s_{{\rm div}}(\mathbb T^d;\R^d)\big),  \end{equation} where 
$  H^s_{{\rm div}}(\mathbb T^d;\R^d)$ is in \eqref{H-div}.

When $d,m\in\mathbb N$ are fixed in the context, for $s\in\R$, $p\in [1,\infty]$, we will simply write
$$H^s=H^s(\mathbb K^d;\R^m),\ \ H^s_{{\rm div}}=H^s_{{\rm div}}(\mathbb K^d;\R^d),\ \ W^{1,\infty}=W^{1,\infty}(\mathbb K^d;\R^m), \ \ L^p=L^p(\mathbb K^d;\R^m).$$ 
Recall that for two nonnegative variables $A$ and $B$, $A\lesssim B$ means that, for some constant $c>0$, $A\le c B$ holds.
We have the following estimates on product of functions in $H^s$.

\begin{Lemma}[\cite{Bahouri-Chemin-Danchin-2011-book}]
\label{LOPN} 
For any $s\ge 0$, 
\begin{align*} 
\|fg\|_{H^s} \lesssim \|f\|_{L^\infty}\|g\|_{H^s}+\|g\|_{L^\infty}\|f\|_{H^s},\ \ f,g\in H^s\cap W^{1\infty}.\end{align*}
\end{Lemma}

For an operator $\A$, $\A^*$ stands for the adjoint operator of $\A$ in $L^2$.
The following two lemmas with single $\Q$ and a pair $(\Q_1,\Q_2)$ on $\R^d$ are well known in the literature, and they can be easily extended to bounded sets in $\OP\SS^r$, and to the case on $\T^d$ (see Remark \ref{OPS R T}).

\begin{Lemma}[see \cite{Abels-2012-Book,Taylor-1991-book,Alinhac-Gerard-2007-Book} for $\OP\S^r$ and \cite{Benzoni-Gavage-Serre-2007-Book} for $\OP\SS^r$] \label{LOP2} Let $r\in\R$ and $\Q\in \OP\SS^{r}$. Then $\Q^*\in \OP\SS^{r}$ and $\Q\in \LL(H^s;H^{s-r})$ for any $s\in\R$. Furthermore, 
for any bounded set $\mathscr O\subset \OP\SS^r$,
$$\sup_{\Q\in\mathscr O} \|\Q\|_{\LL(H^s;H^{s-r})}<\infty,\ \ \ s\in\R.$$

\end{Lemma}

\begin{Lemma}[see  \cite{Abels-2012-Book,Taylor-1991-book,Alinhac-Gerard-2007-Book} for $\OP\S^r$ and \cite{Benzoni-Gavage-Serre-2007-Book,Taylor-1974-note} for $\OP\SS^r$]\label{LOP3} Let $r_i\in\R$ $(i=1,2)$, and let $\mathscr O_i\subset \OP\SS^{r_i} $ be bounded. Then 
$$\big\{\Q_1\Q_2: \Q_i\in \mathscr O_i, i=1,2\big\}\subset \OP\SS^{r_1+r_2}$$ is bounded as well. If $\mathscr O_1$ and $\mathscr O_2$ have commuting matrices, then
$$
\big\{ [\Q_1,\Q_2]:\Q_i\in \mathscr O_i, i=1,2\big\}\subset \OP\SS^{r_1+r_2-1}
$$
is also bounded.
\end{Lemma}

For two operators $\A$ and $\BB$, $[\A,\BB]:=\A\BB-\BB\A$. The original statements of the following two lemmas are for a single operator $\Q\in \OP\S^s$ and for functions   on  $\R^d$, where 
the constant $C$ depends on seminorms of $\Q$. 
Thus, it can be shown that these estimates  hold 
uniformly for $\Q$ in a bounded set of $\OP\S^s.$ Again, by Remark \ref{OPS R T}, they also hold true for functions  on $ \T^d$.

\begin{Lemma} [Proposition 3.6.A in \cite{Taylor-1991-book}]\label{LOP4} Let $ s>0$ and $\mathscr O\subset \OP\S^{s}$ be bounded. Then for any $\sigma\ge0$, $g\in H^{s+\sigma}\cap W^{1,\infty},\ u\in H^{s-1+\sigma}\cap L^{\infty}$,
\begin{equation*}
\sup_{\Q\in \mathscr O}\left\|\left[\Q,g\I\right] u\right\|_{H^{\sigma}} \lesssim \|g\|_{W^{1,\infty}}\|u\|_{H^{s-1+\sigma}}+ \|g\|_{H^{s+\sigma}}\|u\|_{L^{\infty}}.
\end{equation*}
\end{Lemma}

\begin{Lemma}[Proposition 4.2 in \cite{Taylor-2003-PAMS}]\label{LOP5} 
Let $s\ge 0$ and $\mathscr O\subset \OP\S^{s}$ be bounded. Then for any $\si>1+\frac d 2$ and $q\in [0, \si-s]$, 
$$
\sup_{\Q\in \mathscr O} \left\|\left[\Q, g\I\right] u\right\|_{H^{q}} \lesssim \|g\|_{H^{\si}}\|u\|_{H^{q+s-1}},\ \ g\in H^{s},\ u\in H^{q+s-1}.$$
\end{Lemma}


We also recall the Friedrichs mollifier $J_n$ on $\mathbb K^d$. Let $\phi\in\mathscr{S}(\R^d;\R)$ (the Schwarz space of rapidly decreasing $C^\infty$ functions on $\R^d$) satisfy $0\leq\phi(y)\leq1$ for all $y\in \R^d$ and $\phi(y)=1$ for any $|y|\leq1$,  and let 
$$\phi_n(\xi):=\phi_n(\xi/n),\ \ n\ge 1$$  for $\xi\in\R^d$ when $\mathbb K= \R$, and $\xi\in \Z^d$ when $\mathbb K= \T$. 
Then the Friedrichs mollifier $\{J_n\}_{n\ge 1}$  is defined as 
\begin{align}
J_n=\OP\big(\phi_n),\ \ \ \ n\ge1.
\label{JN} 
\end{align}
Of course, one may choose different $J_n$ in different situations, but this is a simple and uniform choice for SPDEs considered in the paper.

Finally, to conclude this part, in the following lemma we summarize some frequently used properties regarding the above operators (cf. \cite{Li-Liu-Tang-2021-SPA,Tang-Yang-2022-AIHP,Taylor-1981-Book}):

\begin{Lemma}\label{LJN} 
Let $\tt\Pi$, $(\D^s,\Lambda^{s})$ and $J_n$ be given in \eqref{tt Pi}, \eqref{Ds Lambda s define} and \eqref{JN}, respectively. Then
for any $1\le j\le d$,  the following properties hold. 
\begin{enumerate}  
\item[$(1)$] For all $\si\ge0$ and $f\in H^\si$, $\displaystyle\sup_{n\ge 1} \|J_n\|_{\mathcal L(L^\infty;L^\infty)}<\infty$, $\|J_n\|_{\mathcal L(H^\si;H^\si)}=1$ $(n\ge1)$ and 
$\displaystyle\lim_{n\to\infty}  \|J_n f-f\|_{H^\si}=0$.  For any  $s_1,s_2\ge 0$, 
$$\|J_n\|_{\LL(H^{s_1};H^{s_2})} \lesssim n^{(s_2-s_1)^+},\  \|J_l-J_n\|_{\LL(H^{s_1};H^{s_2})}\lesssim (l\land n)^{-(s_1-s_2)^+}, \ \  l,n\ge 1.$$ 
\item[$(2)$] For all $s,\si\in\R$ and $n\ge1$, $\tt\Pi$, $\mathcal D^s$, $\Lambda^\si$ and $J_n$ are self-adjoint in $L^2$  and  
$$
[T, \tt T]=[T, \pp_{j}]=0,\ \ T,\, \tt T\in\{\tt\Pi, \D^s, \Lambda^\si, J_n\}_{s,\si\in\R;n\ge1},\ \ 1\le j\le d.$$
\item[$(3)$] For all $g\in W^{1,\infty}$ and $ f\in L^2$,
$$\sup_{n\ge 1} \big\|[J_n, g\I]\nn f\big\|_{L^2}\lesssim \|g\|_{W^{1,\infty}}\|f\|_{L^2},\ \ g\in W^{1,\infty}.$$  \end{enumerate} 

\end{Lemma} 

\subsection{Assumptions and main results}

Recall that $\mathcal A^*$ 
is the $L^2$-adjoint operator of a linear operator $\mathcal A$. For $\mathbb K=\R$ or $\T$ and $d,m\ge1$,  we  simplify notation by letting
\begin{equation}\label{tt Hs}
\tt H^s:=\tt\Pi H^s,\ \ H^s=H^s(\mathbb K^d;\R^m),\ \ s\ge0.
\end{equation}
Because $\tt\Pi$ is a projection, for $X,\, Y\in \tt H^s$ with $s\ge 0$, we have
\begin{equation}\label{tt X Hs}
\<X, Y\>_{\tt H^s}:= \<\tt\Pi X,\tt\Pi Y\>_{H^s}=\<X, Y\>_{H^s},\ \ 
\|X\|^2_{\tt H^s}:=\<X,X\>_{\tt H^s}=\|X\|^2_{H^s}.
\end{equation}

We first introduce the assumption on $\{\mathcal  A_k\}_{k\ge 1} $ included in the noise coefficients of \eqref{EN}. As extension of skew-adjoint operators $\{\pp_i\}_{1\le i\le d}$   in the transport noise \eqref{known transport noise},  
each $\A_{k}$ is not far away from
anti-symmetric, and contains a higher order operator with symbol independent of $x$ and a lower order operator with symbol depending on $x$.

\begin{Assumption}\label{Ak-assum} Let    $\tt\Pi$ be a projection operator in $L^2$  given in \eqref{tt Pi}
and \eqref{tt Pi pi0}. Let  $a_k, q_k\in \R$, $r_1\in [0,1]$,  $r_2\ge r_1$. For $k\ge 1$, we let 
\begin{equation}\label{Ak define}
\A_k=a_k \J_k+q_k \K_k,\ \ \J_k:={\rm diag} (\J_{k,1},\cdots,\J_{k,m}),\ \ \K_k:={\rm diag} (\K_{k,1},\cdots,\K_{k,m}), 
\end{equation}
where $\{\J_{k,i}\}_{k\ge1}\in\OP\S_0^{r_2}$ and $\{\K_{k,i}\}_{k\ge1}\in\OP\S^{r_1}$ $(1\le i\le m)$ are bounded, respectively.
We assume that the following conditions hold:
\begin{enumerate}[label={ $\bf (D_\arabic*)$}]
\item\label{Ak-assum-akbk} $\{a_{k}\}_{k\ge1}, \{q_k\}_{k\ge1}\in l^2$ and $a_kq_k=0$ for $k\ge 1$.
\item\label{Ak-assum-JK}  There exist operators  $\TT_k:={\rm diag} (\TT_{k,1},\cdots,\TT_{k,m})$ and $\Q_k:={\rm diag} (\Q_{k,1},\cdots,\Q_{k,m})$ such that $\{\mathcal{T}_{k,i}\}_{k\ge1}\subset \OP\S_0^{0}$ and $\{\Q_{k,i}\}_{k\ge1}\subset \OP\S^{0}$ $(1\le i\le m)$ are bounded, respectively, 
and 
\begin{equation*}
\J_{k}^*=\TT_{k}-\J_{k},\ \  \K_{k}^*=\Q_{k}-\K_{k}, 
\ \ 
1\le i\le m,\  k\ge1.
\end{equation*} 
\item \label{Ak-assum-Pi} Either $\tt\Pi={\rm diag}\{\tt\Pi_1,\cdots,\tt\Pi_m\}$ with $\tt\Pi_i\in \OP\S_0^0$ $(1\le i\le m)$, or 
\begin{equation}\label{tt Pi U}
\tt\Pi\mathcal U\big|_{\tt H^s} =\mathcal U\big|_{\tt H^s} \ \text{with}\   s\ge r_0:=\max\Big\{
r_2{\bf 1}_{\{\|a_k\|_{l^2}>0\}},\ r_1{\bf 1}_{\{\|q_k\|_{l^2}>0\}}\Big\} \end{equation}  
holds for any $\mathcal{U}\in\left\{\J_k,\, \TT_k,\, \K_k,\, \Q_k:\ k\ge 1\right\}.$ 
\end{enumerate}
\end{Assumption}

Then we assume that the operator $-\EE$ is a positive semi-definite operator satisfying the following assumption:
\begin{Assumption}\label{Assum-E}
Assume that there exists
$\G_j\in \OP\S_0^{p_j}$ with $p_j\ge0$ and $1\le j\le m$  such that
for $p_0:= \max\big\{p_1,\cdots,p_m \big\}$,
\begin{equation*}
\EE\in\OP\SS_0^{2p_0},\ \ 
-\left\<\EE X,X\right\>_{L^2}= \|\G X\|^2_{L^2},\ \ X\in H^{2p_0},\ \ \G={\rm diag}(\G_1,\cdots,\G_m).
\end{equation*}
\end{Assumption}

Usually, if  $\G_j\neq0$ with $1\le j\le m$, then $-\EE$ is positive definite operator. In this work, we allow $\G_{j}=0$ for some $j$ to cover degenerate cases. Hence the case of degenerated diffusion in fluid models can be covered. For instance, degenerated diffusion appears in boundary layer equations.

Finally, 
we introduce the following assumptions on the regular noise coefficients $\tt h_k: [0,\infty)\times H^s\to H^s$ and the drifts $(\tt b,\tt g):\tt H^s\to \tt  H^s$ given by nonlinear PDEs, where $\tt H^s$ is given in \eqref{tt Hs}. 

\begin{Assumption}\label{Assum-hbg}
Assume that the following conditions hold for some 
\begin{equation}\label{theta 0}
\theta_1,\theta_2\ge0,\ q_0>0,\ s_0>\theta_0+q_0 \ \text{with}\ \theta_0:=\max\{\theta_1,\theta_2\},
\end{equation}
some function $K\in\mathscr K$ in \eqref{SCRK}  and an increasing function 
$\tt K:[0,\infty)\rightarrow (0,\infty).$
\begin{enumerate} [label={ $({\bf F}_h)$}]
\item\label{Assum-h}  
Let  $\si\in\{\theta_0,s_0\}$. 
$\tt h_k:[0,\infty)\times H^{s_0}\to H^{s_0}$  satisfies that for all $t\ge0$ and $X,\, Y\in  H^{s_0}$,
\begin{align*}
\sum_{k\ge1}\|\tt h_k(t,X)\|^2_{H^{s_0}}
\leq\, K&(t,\|X\|_{H^{\theta_1}})(1+\|X\|^2_{H^{s_0}}),\\
\sum_{k=1}^\infty \| \tt h_k(t,X)- \tt h_k(t, Y)\|_{H^{\si}}^2 
\le\, &K(t, \|X\|_{H^{s_0}}+\|Y\|_{ H^{s_0}})\|X-Y\|^2_{H^{\si}}.
\end{align*} 
\end{enumerate} 
\begin{enumerate}[label={ $({\bf F}_b)$}]
\item\label{Assum-b} Let  $\si\in\{\theta_0,s_0\}$.     $\tt b:\tt H^{s_0}\to\tt H^{s_0}$  satisfies   that for all $X,\, Y\in \tt H^{s_0}$ and $N>0$,
\begin{equation*}
\|\tt b(X)\|_{H^{s_0}}\leq\, \tt K\big(\|{X}\|_{H^{\theta_2}}\big)\|X\|_{H^{s_0}},
\end{equation*}
\begin{equation*}
\sup_{\|X\|_{H^{s_0}},\|Y\|_{H^{s_0}}\le N}{\bf 1}_{\{X\ne Y\}}  
\frac{\|\tt b(X)-\tt b(Y)\|_{H^{\si}}}{\|X-Y\|_{H^{\si}}}\leq \tt K(N).
\end{equation*}
\end{enumerate}

\begin{enumerate}[label={ $({\bf F}_g)$}]
\item\label{Assum-g}  $\tt g:\tt H^{s_0}\to\tt H^{s_0-q_0}$ satisfies   for all $X,\, Y\in \tt H^{s_0}$ and $N>0$ that 
\begin{equation*}
\|\tt g(X)\|_{H^{s_0-q_0}}\le \tt K(\|X\|_{H^{s_0}}),
\end{equation*}
\begin{equation*}
\sup_{\|X\|_{H^{s_0}},\|Y\|_{H^{s_0}}\le N}{\bf 1}_{\{X\ne Y\}} 
\frac{\|J_n\tt g( J_n X)-J_n\tt g( J_n Y)\|_{H^{s_0}}}{\|X-Y\|_{H^{s_0}}} 
<\infty,
\end{equation*}
\begin{equation*}
\sup_{n\ge 1} \Big\{\big|\<\tt g(J_nX), J_n X\>_{H^{s_0}}\big|
+\big|\<J_n\tt g(X), J_n X\>_{H^{s_0}}\big|\Big\}\le \tt K(\|X\|_{H^{\theta_2}})\|X\|_{H^{s_0}}^2.
\end{equation*}
Moreover, 
there exists   a function $\lambda:\mathbb N\times\mathbb N\to (0,\infty)$ with $\lambda_{n,l}\to 0$ as $n,l\to\infty$ such that for all $X,\, Y\in\tt H^{s_0}$,
\begin{align*}
\sup_{n\ge 1} &\big\<J_n\tt g(J_nX)-J_l\tt g(J_lY),X-Y\big\>_{H^{\theta_0}} \\
&\le\,\tt K(\|X\|_{H^{s_0}}+\|Y\|_{H^{s_0}}) \big(\lambda_{n,l}+\|X-Y\|_{H^{\theta_0}}^2),\ \
n,l\ge 1.
\end{align*}
\end{enumerate}
\end{Assumption}

Recall $\tt H^s$ in \eqref{tt Hs}, $r_0$  in \eqref{tt Pi U}, $p_0$ in \ref{Assum-E} and $\theta_0$  in \ref{Assum-hbg}. 
Now we are in the position to state our main result for \eqref{EN}  as

\begin{Theorem}\label{T3.1} 
Assume {\rm \ref{Ak-assum}} and {\rm \ref{Assum-E}}. Let {\rm \ref{Assum-hbg}} hold with $s_0>\theta_0+\max\{q_0,2r_0,2p_0\}$.  Let $X(0)$ be an $\F_0$-measurable $\tt H^{s_0}$-valued random variable.

\begin{enumerate}[label={$\bf (\alph*)$}]
\item \label{T3.1-a} \eqref{EN} has a unique maximal solution $(X,\tau^*)$ in the sense of Definition \ref{solution definition} for the choice $\H:=\tt H^{s_0}$ and $\M:=\tt H^{\theta_0}$. Besides,
$(X,\tau^*)$ satisfies  $\P\big(X\in C([0,\tau^*);\tt H^{s_0})\big)=1$ and
$\limsup_{t\rightarrow \tau^*}\|X\|_{H^{\theta_0}}=\infty$ on $\{\tau^*<\infty\}$.

\item \label{T3.1-b} 
Let $\B(\cdot)\in C(\tt H^{\theta_0};[0,\infty))$ be a subadditive function satisfying $\B(X)\lesssim \|X\|_{H^{\theta_0}}$ for $X\in\tt H^{\theta_0}$.  If in {\rm \ref{Assum-hbg}}, $K(t,\|X\|_{H^{\theta_1}})$ and $\tt K(\|X\|_{H^{\theta_2}})$ are replaced by $K(t,\B(X))$ and $\tt K(\B(X))$, respectively,  then  
\begin{equation*}
\limsup_{t\rightarrow \tau^*}\B(X)=\infty \ \text{on}\ \{\tau^*<\infty\}.
\end{equation*}

\item \label{T3.1-c} $\P(\tau^*=\infty)=1$ provided that 
\begin{equation}\label{NE SPDE}
\limsup_{\|X\|_{H^{\theta_0}} \to \infty}
\frac{\Psi(T,X,\theta_0)}{\tt K(\|X\|_{H^{\theta_2}})\|X\|_{H^{\theta_0}}^2} 
<-1,\ \ T\in(0,\infty),
\end{equation}
where we define
\begin{equation}\label{NE h-k}
\Psi(T,X,\kk):=\sup_{t\in [0,T]}\sum_{k=1}^\infty \left(\left\|\tt \Pi\tt h_k(t,X)\right\|_{H^{\kk}}^2 
- \frac{2\big\<\tt \Pi\tt h_k(t,X),X\big\>_{H^{\kk}}^2}{{\rm e}+ \|X\|_{H^{\kk}}^2}\right),\ \ \kk>0.
\end{equation}

\end{enumerate}

\end{Theorem} 

\begin{Remark}\label{T3.1-Remark} 
Now we give some technical remarks on \ref{Assum-hbg} and Theorem \ref{T3.1}.

\textbf{(1)}  We assume the growth factor for noise coefficients $h_k$ in \ref{Assum-h} involves $\|X\|_{\theta_1}$. 
Since $(\tt b,\tt g)$ comes from nonlinear PDEs, we assume that the growth factor in \ref{Assum-b} and \ref{Assum-g} involves $\|X\|_{\theta_2}$.  The parameters   $(q_0 ,\theta_2)$ will be determined by $(\tt b,\tt g)$ in  specific models. 
In many cases we require $H^{\theta_2}\hookrightarrow W^{k,\infty}$, for which we may take 
$\theta_2\in(\frac{d}{2}+k,\infty)$ by the Sobolev embedding theorem. 

\textbf{(2)} Assume {\rm \ref{Ak-assum}} and {\rm \ref{Assum-E}}. Theorem \ref{T3.1} means that if {\rm \ref{Assum-hbg}} holds for certain $\theta_1,\theta_2\ge0$ and some $s_0>\theta_0+\max\{q_0,2r_0,2p_0\}$, then  $X(0)\in\tt H^{s_0}$ implies that \eqref{EN} has a unique solution  $X\in C([0,\tau^*);\tt H^{s_0})$. Here the   parameter $s_0$ has freedom in the interval $(\theta_0+\max\{q_0,2r_0,2p_0\},\infty)$, so that the result implies stronger continuity of the solution when $s_0$ is replaced by larger parameters.    For instance, if $X(0)\in \tt H^\infty:=\cap_{k\ge 1} \tt H^k$ and  \ref{Assum-hbg} holds for any  $s_0\in (\theta_0+\max\{q_0,2r_0,2p_0\},\infty),$ then Theorem \ref{T3.1} implies that the unique  solution  is in $C([0,\tau^*);\tt H^{\infty})$, where the lifespan is   determined by $\|X\|_{\M}$ (cf. \eqref{tau-R-M} and \eqref{tau-R-M-X}), and  for any
$s> \theta_0$,   $\limsup_{t\to \tau^*} \|X(t)\|_{H^{\theta_0}}=\infty$ is equivalent to $\limsup_{t\to \tau^*} \|X(t)\|_{H^{s}}=\infty$.  See Sections \ref{Section:Applications} for more examples.
\end{Remark}

\subsection{Proof of Theorem \ref{T3.1}  } 
We 
denote 
\begin{equation}\label{tt E A J K}
\tt{ \mathcal{U}}:=\tt \Pi \mathcal{U}\ \text{for}\ \mathcal{U}\in\{\EE,\, \A_k,\, \J_k,\,  \K_k\},\ \ k\ge1.
\end{equation}
Since $a_kq_k=0$ in \ref{Ak-assum-akbk}, for $\A_k$ in \eqref{Ak define}, we may write 
\begin{equation*}
\{ \tt\A_k X(t)\}\circ \d W_k(t)=\{a_k\tt\J_k X(t)\}\circ \d\ol W_k(t)+\{q_k\tt\K_k X(t)\}\circ \d\hh W_k(t),\ \ k\ge1
\end{equation*} for independent $1$-D Brownian motions $\{\ol W_k(t), \hh W_k(t)\}$, which are also independent of $\{\tt W_k(t)\}$.
Thus, by  \eqref{transform into Ito},
we
rewrite \eqref{EN} as
\begin{align} \label{EN-Ito}
\d X(t)
=\,&\bigg\{(\tt{\EE} X)(t)+\tt b(X(t))+\tt g(X(t)) 
+ \frac 1 2 \sum_{k=1}^\infty \left[\left(a_k\tt\J_k\right)^2 X+\left(q_k\tt\K_k\right)^2 X\right](t)\bigg\} \d t\nonumber\\
&+ \sum_{k=1}^\infty \bigg\{\left(a_k\tt\J_k X\right)(t) \d \ol W_k(t)
+\left(q_k\tt\K_k X\right)(t) \d\hh W_k(t)\bigg\}\nonumber\\
&+\sum_{k=1}^\infty \tt h_k(t,X(t))\d \tt W_k(t),\ \ t\ge 0.
\end{align} 

Before we prove Theorem \ref{T3.1}, we state the following cancellation properties  regarding on $\A_k$, which is required to verify  \ref{R4}.

\begin{Lemma}\label{LAK} Assume {\rm \ref{Ak-assum}}. For any $\si\ge0$, there exists a constant $C>0$ such that $\Y_k\in\big\{a_k \tt\J_k,q_k \tt\K_k\big\}$ satisfies 
\begin{equation}\label{LO1-Ak} 
\sup_{n\ge 1} \sum_{k=1}^\infty \left\<J_n \Y_k X, J_n X\right\>_{H^\si}^2 +\sup_{n\ge 1} \sum_{k=1}^\infty \left\<J_n \Y_kJ_nX, X\right\>_{H^\si}^2 \le C\|X\|_{H^{\si}}^4,\ \ X\in \tt H^{\si+r_0},
\end{equation}
\begin{equation}\label{LO2-Ak} 
\sup_{n\ge 1} \sum_{k=1}^\infty\bigg|\left\<J_n \Y_k^2 X, J_nX\right\>_{H^\si} + \left\| J_n \Y_k X\right\|_{H^\si}^2\bigg| 
\le C\|X\|_{H^\si}^2,\ \ X\in \tt H^{\si+2r_0},
\end{equation}
\begin{equation}\label{LO3-Ak} 
\sup_{n\ge 1} \sum_{k=1}^\infty \bigg|\left\<J^3_n \Y_k^2 J_n X, X\right\>_{H^\si} + \left\| J_n \Y_k J_n X\right\|_{H^\si}^2\bigg| 
\le C\|X\|_{H^\si}^2,\ \ X\in \tt H^{\si+2r_0}.
\end{equation}
\end{Lemma} 
\begin{proof}
%


\textbf{Verify \eqref{LO1-Ak} for $\Y_k=q_k\tt\K_k$.} Remember that in this case $r_0=r_1\in[0,1]$ (cf. \eqref{tt Pi U}). Since $J_n$ and $D^{\si}$ have scalar symbols and $\K_k^*=\Q_k-\K_k$, we have 
for $\{{\mathcal P}_n:= \D^\si J_n\}\subset \OP\S_0^{\si}$,
\begin{align*} \left\<J_n\tt\K_k X, J_n X\right\>_{H^\si} 
=\,&\left\<{\mathcal P}_n\K_k X, {\mathcal P}_n X\right\>_{L^2}\\
=\,& \left\<[{\mathcal P}_n,\K_k]X, {\mathcal P}_n X\right\>_{L^2}
+ \left\<{\mathcal P}_nX, \Q_k {\mathcal P}_nX\right\>_{L^2}
-\left \<{\mathcal P}_n X, \K_k {\mathcal P}_n X\right\>_{L^2} \\
=\,&2 \left\<[{\mathcal P}_n,\K_k]X, {\mathcal P}_nX\right\>_{L^2}
+\left \<{\mathcal P}_n X, \Q_k {\mathcal P}_n X\right\>_{L^2}
- \left\<{\mathcal P}_n X, {\mathcal P}_n \K_k X\right\>_{L^2},\end{align*}
which means
\begin{equation} \left\<J_n\tt\K_kX, J_nX\right\>_{H^\si} 
=\left\<[{\mathcal P}_n,\K_k]X, {\mathcal P}_nX\right\>_{L^2}
+ \frac 1 2
\left \<{\mathcal P}_n X, \mathcal Q_k {\mathcal P}_nX\right\>_{L^2}.\label{LO1-Ak-K}
\end{equation} 
It follows from \ref{Ak-assum}, the boundedness of $\{{\mathcal P}_n\}_{n\ge 1}\subset \OP\S_0^\si$, $\{\K_k\}_{k\ge1}\subset \OP\SS^{r_1}$ and $\{\Q_k\}_{k\ge1}\subset \OP\SS^{0}$,  Lemmas \ref{LOP2} and \ref{LOP3} that 
\begin{equation*}
\sup_{n\ge 1} \sum_{k=1}^\infty \left\<J_n (q_k \tt \K_k) X, J_n X\right\>_{H^\si}^2 \le C\sum_{k=1}^\infty q_k^2\|X\|_{H^{\si}}^4\le C \|X\|_{H^{\si}}^4,\ \ X\in \tt H^{\si+r_0}.
\end{equation*}
Similarly, by Lemma \ref{LJN}, repeating the above process with ${\mathcal P}:=\D^\si$ leads to
\begin{equation*}
\sup_{n\ge 1} \sum_{k=1}^\infty \left\<J_n (q_k \tt\K_k)J_nX, X\right\>_{H^\si}^2 \le C\|X\|_{H^{\si}}^4,\ \ X\in \tt H^{\si+r_0}.
\end{equation*}
Combining the above two estimates gives rise to \eqref{LO1-Ak} for the case $\mathcal Y_k=q_k\tt\K_k$.

\textbf{Verify \eqref{LO1-Ak} for $\Y_k=a_k\tt\J_k$.} Similar to \eqref{LO1-Ak-K}, in this case we have $[\J_k,{\mathcal P}_n]=0$ and hence we have
\begin{equation*} \left\<J_n\tt\J_kX, J_nX\right\>_{H^\si} 
=\frac 1 2
\left \<{\mathcal P}_n X, \TT_k {\mathcal P}_nX\right\>_{L^2}. 
\end{equation*} 
By the same reason leading to \eqref{LO1-Ak} for the case $\A_k=q_k\tt\K_k$, we have
\begin{equation*}                       
\sup_{n\ge 1} \sum_{k=1}^\infty \left\<J_n a_k \tt \J_k X, J_n X\right\>_{H^\si}^2 \le C\sum_{k=1}^\infty a_k^2\|X\|_{H^{\si}}^4\le C \|X\|_{H^{\si}}^4,\ \ X\in \tt H^{\si+r_0}.
\end{equation*}
Similarly, one can obtain the desired bound for the other term.
Hence \eqref{LO1-Ak} also holds true for the case $\mathcal Y_k=a_k\J_k$.

\textbf{Verify \eqref{LO2-Ak} for $\Y_k=q_k\tt\K_k$.} 
Because $\tt\Pi$ is self-adjoint, we have 
\begin{align*}
(\tt\K_k)^*=\K_k^*\tt\Pi=\Q_k\tt\Pi-\K_k\tt\Pi.
\end{align*}
We will frequently use the above facts as well as the following properties without further notice:
$$\tt\K_k=\tt\Pi\K_k=\tt\Pi\tt\K_k,\ [\mathcal P_n,\tt\Pi]=0,\ \tt\Pi=\tt\Pi^*, \ \Q_k=\Q_k^*\ \text{and}\ \mathcal P_n X=\tt\Pi\mathcal P_n X\ \text{for}\ X\in \tt H^{\si+2r_0}.$$
For $k,n\in\N$, we let
\begin{align*}           
\mathcal{P}_n:=\D^{\si}J_n,\ \ 
\tt{\mathcal{R}}_{1,k}:= [\tt\K_k,\tt\Pi\Q_k],\ \
\tt{\mathcal{R}}_{2,k,n}:=[\mathcal{P}_n,\tt \K_k],\ \ \tt{\mathcal{R}}_{3,k,n}:=[\tt\Pi\mathcal{R}_{2,k,n},\tt \K_k],
\end{align*}
We observe that for all $X\in\tt H^{\si+2r_0}$,
\begin{align*}
&\left\<J_n\tt\K_k^2 X, J_nX\right\>_{H^\si}\nonumber\\
=\,&\left\<\mathcal{P}_n \tt\K_k X, (\tt\K_{k} )^*\mathcal{P}_n X \right\>_{L^{2}}
+\left\<\tt{\mathcal{R}}_{2,k,n} \tt\K_k X, \mathcal{P}_n X \right\>_{L^{2}}\nonumber\\
=\,&-\left\<\mathcal{P}_n \tt \K_k X, \tt\K_k \mathcal{P}_n X \right\>_{L^{2}}+ \left\<\mathcal{P}_n \tt\K_k X, \Q_k \mathcal{P}_n X \right\>_{L^{2}}+ \left\<\tt{\mathcal{R}}_{2,k,n}\tt \K_k X, \mathcal{P}_n X \right\>_{L^{2}}\nonumber\\
=\,&- \left\<\mathcal{P}_n \tt \K_k X, \mathcal{P}_n \tt\K_k X \right\>_{L^{2}}
+ \left\<\mathcal{P}_n \tt\K_k X, \tt{\mathcal{R}}_{2,k,n} X \right\>_{L^{2}}\nonumber\\
& + \left\<\mathcal{P}_n \tt\K_k X, \Q_k \mathcal{P}_n X \right\>_{L^{2}}+ \left\<\tt{\mathcal{R}}_{2,k,n} \tt\K_k X, \mathcal{P}_n X \right\>_{L^{2}},
\end{align*}
which means
\begin{align*}
&\left\<J_n\tt\K_k^2 X, J_nX\right\>_{H^\si} + \| J_n \tt\K_k X\|_{H^\si}^2 \\
=\,&
\left\<\mathcal{P}_n \tt\K_k X, \tt{\mathcal{R}}_{2,k,n} X \right\>_{L^{2}}
+ \left\<\mathcal{P}_n \tt\K_k X, \Q_k \mathcal{P}_n X \right\>_{L^{2}}
+ \left\<\tt{\mathcal{R}}_{2,k,n} \tt\K_k X, \mathcal{P}_n X \right\>_{L^{2}}.
\end{align*}
Then, 
we observe that
\begin{equation*}\begin{split}
& \left\<J_n\tt\K_k^2 X, J_nX\right\>_{H^\si} + \| J_n \tt\K_k X\|_{H^\si}^2  \\ 
=\,& \left\< \tt\K_k \mathcal{P}_n X, \tt{\mathcal{R}}_{2,k,n} X \right\>_{L^{2}}
+ \left\<\tt{\mathcal{R}}_{2,k,n} X, \tt{\mathcal{R}}_{2,k,n} X \right\>_{L^{2}}\\
&+ \left\<\tt{\mathcal{R}}_{2,k,n} \tt\K_k X, \mathcal{P}_n X \right\>_{L^{2}}+ \left\<\mathcal{P}_n \tt\K_k X, \Q_k\mathcal{P}_n X \right\>_{L^{2}} \\ 
=\,& - \left\<\mathcal{P}_n X, \tt\K_k \tt\Pi\tt{\mathcal{R}}_{2,k,n} X \right\>_{L^{2}}+ \left\<\mathcal{P}_n X, \Q_k\tt\Pi \tt{\mathcal{R}}_{2,k,n} X \right\>_{L^{2}} \\
&+ \left\< \tt{\mathcal{R}}_{2,k,n} X, \tt{\mathcal{R}}_{2,k,n} X \right\>_{L^{2}}+ \left\<\tt\Pi\tt{\mathcal{R}}_{2,k,n} \tt \K_k X, \mathcal{P}_n X \right\>_{L^{2}} + \left\<\mathcal{P}_n\tt\K_k X, \Q_k\mathcal{P}_n X \right\>_{L^{2}} \\ 
=\,& \left\<\tt{\mathcal{R}}_{3,k,n} X, \mathcal{P}_n X \right\>_{L^{2}}+ \left\<\mathcal{P}_n X, \Q_k \tt\Pi\tt{\mathcal{R}}_{2,k,n} X \right\>_{L^{2}}\\
&+ \left\< \tt{\mathcal{R}}_{2,k,n} X, \tt{\mathcal{R}}_{2,k,n} X \right\>_{L^{2}} + \left\<\mathcal{P}_n\tt\K_k X, \tt\Pi\Q_k\mathcal{P}_n X \right\>_{L^{2}}.
\end{split}\end{equation*} 
Hence $$- \left\<\mathcal{P}_n X, \tt\K_k \tt\Pi\tt{\mathcal{R}}_{2,k,n} X \right\>_{L^{2}}+ \left\<\tt\Pi\tt{\mathcal{R}}_{2,k,n} \tt \K_k X, \mathcal{P}_n X \right\>_{L^{2}}
= \left\<\tt{\mathcal{R}}_{3,k,n} X, \mathcal{P}_n X \right\>_{L^{2}}.$$ 
Once again, to deal with $\mathcal{P}_n\tt\K_k$, we have
\begin{align*}
& \left\<\mathcal{P}_n\tt\K_k X, \tt\Pi\Q_k\mathcal{P}_n X \right\>_{L^{2}}\\
=\,& - \left\<\mathcal{P}_n X, \tt\Pi\Q_k\tt \K_k \mathcal{P}_n X \right\>_{L^{2}}
- \left\<\mathcal{P}_n X, \tt{\mathcal{R}}_{1,k} \mathcal{P}_n X \right\>_{L^{2}}\\
&+ \left\<\mathcal{P}_n X, \Q_k\tt\Pi\Q_k \mathcal{P}_n X \right\>_{L^{2}}
+ \left\<\tt{\mathcal{R}}_{2,k,n} X, \tt\Pi\Q_k \mathcal{P}_n X \right\>_{L^{2}}\\
=\,& - \left\<\tt\Pi\Q_k\mathcal{P}_n X, \tt\K_k \mathcal{P}_n X \right\>_{L^{2}}
- \left\<\mathcal{P}_n X, \tt{\mathcal{R}}_{1,k} \mathcal{P}_n X \right\>_{L^{2}} \\
&+ \left\<\mathcal{P}_n X, \Q_k\tt\Pi\Q_k \mathcal{P}_n X \right\>_{L^{2}}
+ \left\<\tt{\mathcal{R}}_{2,k,n} X, \tt\Pi\Q_k \mathcal{P}_n X \right\>_{L^{2}}.
\end{align*}
Accordingly, adding $ \left\<\mathcal{P}_n\tt\K_k X, \tt\Pi\Q_k\mathcal{P}_n X \right\>_{L^{2}}$ to both sides of the above equation and then using $\tt{\mathcal{R}}_{2,k,n}$ yield
\begin{equation*}\begin{split}
&2 \left\<\mathcal{P}_n\tt\K_k X, \tt\Pi\Q_k\mathcal{P}_n X \right\>_{L^{2}} \\
=\, & 
- \left\<\mathcal{P}_n X, \tt{\mathcal{R}}_{1,k} \mathcal{P}_n X \right\>_{L^{2}} + \left\<\mathcal{P}_n X, \Q_k\tt\Pi\Q_k \mathcal{P}_n X \right\>_{L^{2}}
+2 \left\<\tt{\mathcal{R}}_{2,k,n} X, \tt\Pi\Q_k \mathcal{P}_n X \right\>_{L^{2}}
\end{split}\end{equation*} 
Combining the above estimates gives 
\begin{align}
\Big|\left\<J_n\tt\K_k^2 X, J_nX\right\>_{H^\si} + \| J_n \tt\K_k X\|_{H^\si}^2\Big|
\leq \sum_{i=1}^6|I_i|\label{LO1-2-K}
\end{align}
with
\begin{equation*}
I_{1}=\left\< \tt{\mathcal{R}}_{3,k,n}X, \mathcal{P}_n X \right\>_{L^2}, \ \ 
I_{2}=\left\< \tt{\mathcal{R}}_{2,k,n}X, \tt{\mathcal{R}}_{2,k,n} X \right\>_{L^2},
\end{equation*}
\begin{equation*}
I_{3}=\left\<\mathcal{P}_n X, \tt\Pi\Q_k\tt{\mathcal{R}}_{2,k,n}X \right\>_{L^2},
\ \ I_{4}=-\frac{1}{2}\left\< \mathcal{P}_n X, \tt{\mathcal{R}}_{1,k} \mathcal{P}_n X \right\>_{L^2}, 
\end{equation*}
\begin{equation*}
I_{5}= \frac{1}{2}\left\< \mathcal{P}_n X, \Q_k\tt\Pi\Q_k \mathcal{P}_n X \right\>_{L^2},\ \
I_{6} =\left\< \tt{\mathcal{R}}_{2,k,n}X, \tt\Pi\Q_k\mathcal{P}_n X \right\>_{L^2}
\end{equation*}

Now we argue in two cases. We first consider   $\tt\Pi={\rm diag}\{\tt\Pi_1,\cdots,\tt\Pi_m\}$ with $\tt\Pi_i\in \OP\S_0^0$ $(1\le i\le m)$ in \ref{Ak-assum-Pi}. In this case,   $\{\tt\K_k\}_{k\ge1}\subset \OP\SS_0^{r_1}$ and $\{\tt\Pi\Q_k\}_{k\ge1}\subset \OP\SS_0^{0}$ are bounded, so that  Lemmas \ref{LOP2} and \ref{LOP3} imply 
\begin{equation*}
\sup_{n,k\ge1}\left\{\big\|\tt{\mathcal{R}}_{1,k}\big\|_{\LL(L^{2};L^2)}
+\big\|\tt{\mathcal{R}}_{2,k,n}\big\|_{\LL(H^{s};L^2)}
+\big\|\tt{\mathcal{R}}_{3,k,n}\big\|_{\LL(H^{s};L^2)}\right\}<\infty.
\end{equation*}
Hence \eqref{LO2-Ak} holds true   for $\mathcal Y_k=q_k\tt\K_k$. 

Next, we  assume that  \eqref {tt Pi U} in \ref{Ak-assum-Pi} holds.  In this case we let 
\begin{equation}
\mathcal{R}_{1,k}:=[\K_k,\Q_k],\ \ \mathcal{R}_{2,k,n}:=[\mathcal{P}_n,\K_k],\ \
\mathcal{R}_{3,k,n}:=[\mathcal{R}_{2,k,n},\K_k].
\end{equation}
Then \eqref{tt Pi U} and the fact that $[\mathcal P_n,\tt\Pi]=0$ tell us
\begin{equation*}
\tt{\mathcal{R}}_{1,k}\big|_{\tt H^s}= \mathcal{R}_{1,k}\big|_{\tt H^s},\ \
\tt{\mathcal{R}}_{2,k,n}\big|_{\tt H^s}=\mathcal{R}_{2,k,n}\big|_{\tt H^s},\ \
\tt{\mathcal{R}}_{3,k,n}\big|_{\tt H^s}=\mathcal{R}_{3,k,n}\big|_{\tt H^s}.
\end{equation*}
On one hand, as $X\in\tt H^{s+r_0}$ (hence $\mathcal P_n X\in \tt L^2$), one can replace
$\tt{\mathcal{R}}_{1,k}$, $\tt{\mathcal{R}}_{2,k,n}$ and $\tt{\mathcal{R}}_{3,k,n}$ in above $I_{i}$ ($i=1,2,3,4,6$) by ${\mathcal{R}}_{1,k}$, ${\mathcal{R}}_{2,k,n}$ and ${\mathcal{R}}_{3,k,n}$, respectively.
On the other hand, 
due to the boundedness of $\{\K_k\}_{k\ge1}\subset \OP\SS^{r_1}$ and $\{\Q_k\}_{k\ge1}\subset \OP\SS^{0}$, 
Lemmas \ref{LOP2} and \ref{LOP3} imply
\begin{equation*}
\sup_{n,k\ge1}\Big\{\|\mathcal{R}_{1,k}\|_{\LL(L^{2};L^2)}+
\|\mathcal{R}_{2,k,n}\|_{\LL(H^{s};L^2)}+\| \mathcal{R}_{3,k,n}\|_{\LL(H^{s};L^2)}\Big\}<\infty.
\end{equation*}
Consequently, for $X\in\tt H^{s+r_0}$, we have 
\begin{align*}
&\sup_{n\ge 1} \sum_{k=1}^\infty \left|\left\<J_n\left(q_k\tt\K_k\right)^2 X, J_nX\right\>_{H^\si} + \left\| J_n (q_k\tt\K_k) X\right\|_{H^\si}^2\right| \\
\leq\,  &\sum_{k=1}^\infty q_k^2\left(\displaystyle\sum_{i=1}^{6} |I_{i}|\right)
\leq\, C\|X\|^2_{H^\si},
\end{align*}
which is the desired result.

\textbf{Verify \eqref{LO2-Ak} for $\Y_k=a_k\tt\J_k$.} In this case,  by the same argument leading to   \eqref{LO1-2-K}, we derive the same estimate    for $(\tt\J_k, \TT_k)$ replacing $(\tt\K_k,   \Q_k)$.   By  \ref{Ak-assum-JK}, in each situations of \ref{Ak-assum-Pi} we have 
$\mathcal{R}_{1,k}=\mathcal{R}_{2,k,n}=\mathcal{R}_{3,k,n}=0$ for $k,n\ge1$, 
so that 
\begin{align*}
&\sum_{k=1}^\infty\bigg|\left\<J_n\left(a_k\tt\J_k\right)^2 X, J_nX\right\>_{H^\si} + \| J_n (a_k\tt\J_k) X\|_{H^\si}^2\bigg|  \\
=\, &\sum_{k=1}^\infty a_{k}^2 \left|\frac{1}{2}\left\< \mathcal{P}_n X, \TT_k\tt\Pi\TT_k \mathcal{P}_n X \right\>_{L^2}\right|
\leq C\|{X}\|_{H^\si}^2. 
\end{align*}

\textbf{Verify \eqref{LO3-Ak}.} For the case that $\Y_k=q_k\tt\K_k$, going along the lines of the above proof of \eqref{LO2-Ak} with replacing $X$ by $J_n X$ gives the desired upper bound. The case of $\Y_k=a_k\tt\J_k$ is also similar. We omit the details for brevity.
\end{proof}

With Lemma \ref{LAK} at hand,
now we are in the position to prove Theorem \ref{T3.1}.

\begin{proof}[Proof of Theorem \ref{T3.1}] 
Recall that \eqref{EN} is equivalent to \eqref{EN-Ito}:
\begin{align*}
\d X(t)
=\,&\bigg\{(\tt{\EE} X)(t)+\tt b(X(t))+\tt g(X(t)) 
+ \frac 1 2 \sum_{k=1}^\infty \left[\left(a_k\tt\J_k\right)^2 X+\left(q_k\tt\K_k\right)^2 X\right](t)\bigg\} \d t\nonumber\\
&+ \sum_{k=1}^\infty \bigg\{\left(a_k\tt\J_k X\right)(t) \d \ol W_k(t)
+\left(q_k\tt\K_k X\right)(t) \d\hh W_k(t)\bigg\}\nonumber\\
&+\sum_{k=1}^\infty \tt h_k(t,X(t))\d \tt W_k(t),\ \ t\ge 0.
\end{align*}

\ref{T3.1-a} \textbf{(Existence and uniqueness)} The above equation is embedded into \eqref{E1} with 
\begin{equation}\label{GH}
\left\{\begin{aligned}
&b(t,X):=\tt b(X),\\
& g(t,X)= g(X):=\tt\EE X +\tt g(X)+ \sum_{k=1}^\infty \left[\left(a_k\tt\J_k\right)^2 X+\left(q_k\tt\K_k\right)^2 X\right],\\
& h(t,X)e_{3k-2} := a_k\tt\J_kX,\ \ k\ge 1,\\ 
&h(t,X)e_{3k-1} :=  q_k \tt\K_k  X,\ \ k\ge 1,\\ 
&h(t,X)e_{3k}:= \tt h_k(t,X),\ \ k\ge 1,\\
& \W(t):=\sum_{k=1}^\infty \left(\ol W_k(t) e_{3k-2} +\hh W_k(t) e_{3k-1}+\tt W_k(t) e_{3k}\right).
\end{aligned}\right.
\end{equation} 
Let $\{J_n\}_{n\ge 1}$ be defined in \eqref{JN}. 
For any $n\in\mathbb N$ and $X\in \tt H^{s_0}$, define
\begin{equation}\label{GH-n}
\left\{\begin{aligned}
&g_n(X):=\,  J_n\tt\EE J_n X+J_n \tt g(J_nX)\\
&\qquad\qquad+\frac{1}{2} \sum_{k=1}^{\infty}J_n^3(a_k\tt\J_k)^2J_nX+\frac{1}{2} \sum_{k=1}^{\infty}J_n^3(q_k\tt\K_k)^2J_nX,\ \ k\ge 1,\\
&h_n(t,X)e_{3k+1} :=\, J_n(a_k\tt\J_k) J_nX,\ \ k\ge 1,\\ 
&h_n(t,X)e_{3k+2} :=\,J_n(q_k \tt\K_k) J_nX,\ \ k\ge 1,\\ 
&h_n(t,X)e_{3k}:=\, \tt h_k(t,X),\ \ k\ge 1.
\end{aligned}\right.
\end{equation} 
In order to prove Theorem \ref{T3.1}, by Theorem \ref{T1} and Proposition \ref{P1}, we will show that $(g_n,h_n)$ is the desired \textit{proper regularization} satisfying \ref{R1}-\ref{R4} and assumptions \ref{A} with the choice (cf. \eqref{tt Hs}):
\begin{equation}\label{H M SPDE}
\H:=\tt H^{s_0},\ \ \M=\tt H^{\theta_0}, 
\end{equation}
where $\theta_0$ is given in \eqref{theta 0}. It is easy to see that 
Lemma \ref{LJN}, \ref{Ak-assum}, \ref{Assum-E} and \ref{Assum-hbg} imply  \ref{R1}, \ref{R2} and \ref{R3}.  It remains to verify \ref{R4} and \ref{A}.

\textbf{Verify \ref{R4}.}
By \ref{Assum-h} and \eqref{LO1-Ak} (with $\si={s_0}$), we find some function $\hh K: [0,\infty)\times [0,\infty)\to (0,\infty)$ increasing in both variables such that for all $X\in \H$ and $t\ge 0$,
\begin{align} 
\sum_{k=1}^\infty\<h_n(t,X) e_k, X\>_{\H}^2 
\le\, &\hh K(t,\|X\|_{\M}) \big(1+ \|X\|_{\H}^4\big).\label{Check R4-1}
\end{align}
On the other hand,   keeping in mind that $\tt\Pi X=X$ for $X\in\H$ and $[\D^{s_0},\EE]=0$,  we infer 
from \ref{Assum-E}, \eqref{tt Pi} and Lemma \ref{LJN}  that
\begin{equation*}    
-\<J_n\tt\EE J_n X,X\>_{H^{s_0}}= -\<\EE J_n\D^{s_0} X,J_n\D^{s_0} X\>_{L^2}= \|\G J_n\D^{s_0} X\|^2_{L^2},\ \ X\in\H.
\end{equation*}
Using this, \ref{Assum-h}, Lemma \ref{LJN}, \ref{Assum-g} and \eqref{LO3-Ak} (with $\si={s_0}$), we derive for all $X\in \tt H^{s_0}$ and $t\ge 0$ that
\begin{align}
&2\<g_n(X), X\>_{\H} +\|h_n(X)\|_{\LL_2(\U;\H)}^2 \nonumber\\
=\, & 2\big\<\tt \EE J_nX, J_n X\big\>_{H^{s_0}}+2\big\<\tt g(J_nX), J_nX\big\>_{H^{s_0}}\nonumber\\
&+\sum_{k=1}^{\infty}\left\<J_n^3(a_k\tt\J_k)^2J_nX+ J_n^3(q_k\tt\K_k)^2J_nX,X\right\>_{H^{s_0}} \nonumber\\
&+\sum_{k=1}^{\infty} \Big\{\|J_n(a_k\tt\J_k) J_nX\|_{H^{s_0}}^2+\|J_n(q_k\tt\K_k) J_nX\|_{H^{s_0}}^2+\|\tt h_k(t,X)\|_{H^{s_0}}^2 \Big\} \nonumber\\
\le\,& \hh K(t,\|X\|_{\M}) (1+\|X\|_{\H}^2). \label{Check R4-2}
\end{align} 
Thus, \ref{R4} holds.

\textbf{Verify \ref{A}.} Obviously, \ref{A1} is a consequence of \ref{Assum-b}. It suffices to prove \ref{A2}. Remember that (cf. \eqref{tt E A J K})
\begin{equation*}    
\tt{ \mathcal{U}}:=\tt \Pi \mathcal{U},\ \ \mathcal{U}\in\{\EE,\, \A_k,\, \J_k,\,  \K_k\},\ \ k\ge1.
\end{equation*}   
By the definition of $h_n$ in \eqref{GH-n}, Lemma \ref{LJN}, \ref{Assum-b} and \eqref{LO1-Ak} (with $\si=\theta_0$ and $X-Y$ replacing $X$), we can find a constant $C>0$ and a function $\hh K:[0,\infty)\times [0,\infty)\rightarrow (0,\infty)$ increasing in both variables,  such that
for any $\epsilon\in(0,s_0-\theta_0-r_0)$,
\begin{align}
& \sum_{k=1}^\infty\big\<\{h_n(t,X)-h_l(t,Y)\}e_k, X-Y\big\>_{\M}^2 \nonumber\\                               
\le\,& 2 \sum_{k=1}^\infty\Big\{\big\<J_n \tt\A_k J_nX-J_l \tt\A_k J_lY , X-Y\big\>_{H^{\theta_0}}^2\nonumber\\
&\qquad\quad+\|X-Y\|_{H^{\theta_0}}^2\|\tt h_k(t,X)-\tt h_k(t,Y)\|_{H^{\theta_0}}^2\Big\} \nonumber\\
\le\,& 6\sum_{k=1}^\infty\Big\{\big\<(J_n-J_l)\tt\A_kJ_nX, X-Y\big\>_{H^{\theta_0}}^2
+ \big\<J_l \tt\A_k (J_n-J_l)X, X-Y\big\>_{H^{\theta_0}}^2\Big\} \nonumber\\
&+6\sum_{k=1}^\infty\big\<\tt\A_k J_l(X-Y), J_l (X-Y)\big\>_{H^{\theta_0}}^2\nonumber\\
&+ 2K(t,\|X\|_{H^{s_0}}+\|Y\|_{H^{s_0}})^2\|X-Y\|_{H^{\theta_0}}^4  \nonumber\\
\le\,& C (l\land n)^{-(s_0-r_0-\theta_0-\epsilon)}\|X\|_{\H}^2 \|X-Y\|_{\M}^2\nonumber\\
&+ 
\hh K(t, \|X\|_{\H}+\|Y\|_{\H})\|X-Y\|_{\M}^4,\qquad n,\, l\ge 1,\ X,\, Y\in \H,\label{Check A2-1} 
\end{align} 
which implies the first condition in \ref{A2}. To verify the second condition in \ref{A2}, we observe from \eqref{GH-n} that 
$$ 2\big\<g_n(X)-g_l(Y), X-Y \big\>_{\M}+ 
\| h_n(t,X)-h_l(t,Y)\|^2_{\LL_2(\U;\M)} =\Theta_{1} +\Theta_{2}+\sum_{k=1}^\infty \sum_{i=3}^7\Theta_{i,k},$$ where 
\begin{align*} 
\Theta_{1} :=&\,  2 \left\<J_n \tt\EE J_n X-J_l \tt\EE J_l Y,X-Y \right\> _{H^{\theta_0}},\\
\Theta_{2}:=&\,2\left\<J_n\tt g(J_n X)-J_l\tt g(J_l Y),X-Y \right\>_{H^{\theta_0}}\\
\Theta_{3,k}:=&\, \left\<J_n^3(a_k\tt\J_k)^2J_n X-J_l^3(a_k\tt\J_k)^2J_l Y, X-Y \right\>_{H^{\theta_0}},\\ 
\Theta_{4,k}:=&\, \left\<J_n^3(q_k\tt\K_k)^2J_n X-J_l^3(q_k\tt\K_k)^2J_l Y, X-Y \right\>_{H^{\theta_0}}, \\
\Theta_{5,k}:=&\, \| h_n(t,X)e_{3k-2}-h_l(t,Y)e_{3k-2}\|^2_{H^{\theta_0}},\\  
\Theta_{6,k}:=&\, \| h_n(t,X)e_{3k-1}-h_l(t,Y)e_{3k-1}\|^2_{H^{\theta_0}}, \\
\Theta_{7,k}:=&\,\| \tt h_k(t,X)-\tt h_k(t,Y)\|^2_{H^{\theta_0}}.  
\end{align*} 
Firstly,   by    \ref{Assum-g} and \ref{Assum-b} we find an increasing function $\tt K: [0,\infty)\to (0,\infty)$ and a map $\lambda: \mathbb N\times\mathbb N\to (0,\infty)$ with 
$\lambda_{n.l}\to 0$ as $n,l\to\infty$, such that
\begin{equation}\label{QQ1}   \Theta_{2}+ \sum_{k=1}^\infty\Theta_{7,k} \le \tt K(\|X\|_\H+\|Y\|_\H) (\lambda_{n,l}+ \|X-Y\|_\M^2),\ \ n,l\ge 1, X,Y\in\H.\end{equation} 
Next, by  \ref{Assum-E} and Lemma \ref{LOP3}, we have $\EE\in\OP\SS_0^{2p_0}$, where $p_0$ is given in \ref{Assum-E}. Then it follows from \ref{Assum-E} and Lemmas \ref{LJN} and \ref{LOP2} that for any $\epsilon\in(0,s_0-2p_0-\theta_0)$,
\begin{align}            
\Theta_1
=\,& \left\<(J_n-J_l)\tt\EE J_n X,X-Y \right\>_{H^{\theta_0}}\notag\\
&+\left\<J_l\tt\EE (J_n-J_l) X,X-Y \right\>_{H^{\theta_0}}\notag\\
&+\left\<J_l\tt\EE J_l (X-Y),X-Y \right\>_{H^{\theta_0}} \notag\\
\lesssim\, & (l\land n)^{-(s_0-\theta_0-2p_0-\epsilon)}\|X\|_{H^{s_0}}\|X-Y\|_{H^{\theta_0}}\notag\\
\lesssim\,& \|X\|_{\H}\left((l\land n)^{-2(s_0-\theta_0-2p_0-\epsilon)} 
+\|X-Y\|^2_{\M}\right),\quad X,\, Y\in \H,\ \ n,\, l\ge 1.\label{QQ2}    
\end{align}  
Moreover,  to estimate  $\sum_{k=1}^\infty \left\{\Theta_{3,k}+\Theta_{5,k}\right\}$, we  find
$$
\Theta_{3,k}=\sum_{j=1}^{3}\Theta_{3,k,j},\ \ \ \Theta_{5,k}=\sum_{i,j=1}^{3}\left\<\Theta_{5,k,i},\Theta_{5,k,j} \right\>_{H^{\theta_0}},$$
where 
\begin{equation*}
\begin{cases}
\Theta_{3,k,1}:= \left\<(J^3_n-J^3_l)(a_k\tt\J_k)^2 J_n X, X-Y \right\>_{H^{\theta_0}},\ \ 
&\Theta_{5,k,1}:=\, (J_n-J_l)(a_k\tt\J_k)J_nX,\\
\Theta_{3,k,2}:= \left\< J^3_l(a_k\tt\J_k)^2 (J_n-J_l) X, X-Y \right\>_{H^{\theta_0}},\ \
&\Theta_{5,k,2}:=\, J_l(a_k\tt\J_k) (J_n-J_l) X,\\
\Theta_{3,k,3}:= \left\<J^3_l(a_k\tt\J_k)^2 J_l (X-Y), X-Y\right\>_{H^{\theta_0}},\ \
&\Theta_{5,k,3}:=\, J_l(a_k\tt\J_k) J_l (X-Y).
\end{cases}
\end{equation*}
Analogous to the analysis in \eqref{Check A2-1}, we have for any $\epsilon\in(0,s_0-2r_0-\theta_0)$ and $X,\, Y\in \H$,
\begin{equation*}         
\sum_{k=1}^\infty\Theta_{3,k,1},\ 
\sum_{k=1}^{\infty}\Theta_{3,k,2}
\lesssim\, (l\land n)^{-(s_0-2r_0-\theta_0-\epsilon)} 
\|{X}\|_{\H}\|{X-Y}\|_{\M},              
\end{equation*}
\begin{equation*}
\sum_{k=1}^{\infty}\sum_{i,j\in \{1,2\}}
\left\< \Theta_{5,k,i},\Theta_{5,k,j} \right\>_{H^{\theta_0}} 
\lesssim\, (l\land n)^{-2(s_0-2r_0-\theta_0-\epsilon)} 
\|{X}\|^2_{\H},                            
\end{equation*}
\begin{equation*}
\sum_{k=1}^{\infty}\sum_{i\in \{1,2\}}
\left\< \Theta_{5,k,i},\Theta_{5,k,3} \right\>_{H^{\theta_0}}
\lesssim
(l\land n)^{-(s_0-2r_0-\theta_0-\epsilon)} 
\|{X}\|_{\H}\|{X-Y}\|_{\M}
\end{equation*}
Then we apply \eqref{LO3-Ak} (with $\si=\theta_0$ and $X-Y$ replacing $X$) to find 
$$\sum_{k=1}^{\infty}\left\{\Theta_{3,k,3}+\left\< \Theta_{5,k,3},\Theta_{5,k,3} \right\>_{H^{\theta_0}}\right\}\lesssim \|X-Y\|^2_{\M}, \ \ X,\, Y\in \H.$$
In conclusion, we derive that for all $X,Y\in\H$ and $n$, $l\ge1$,
\begin{align*}
\sum_{k=1}^\infty \left\{\Theta_{3,k}+\Theta_{5,k}\right\} 
\lesssim\,& (1+\|X\|_\H^2+\|Y\|_\H^2)\big\{(l\land n)^{-(s_0-2r_0-\theta_0-\epsilon)}+\|X-Y\|_\M^2\big\}.
\end{align*}
Similarly,   the same estimate holds for $\sum_{k=1}^{\infty}\left\{\Theta_{4,k}+\Theta_{6,k}\right\}$. Combining these with \eqref{QQ1} and \eqref{QQ2},   we verify the second condition in \ref{A2}. Therefore,  Theorem \ref{T1} \ref{T1-existence} implies that \eqref{EN-Ito} has a unique maximal solution $(X,\tau^*)$ with 
\begin{equation*}
\limsup_{t\rightarrow \tau^*}\|X\|_{\M}=\limsup_{t\rightarrow \tau^*}\|X\|_{H^{\theta_0}}=\infty \  \ \text{on}\ \{\tau^*<\infty\}.
\end{equation*}

\ref{T3.1-a} \textbf{(Time continuity)}  As explained in the proof of Theorem \ref{T1}\ref{T1-conti}, it suffices to prove the continuity of 
$[0,\tau^*)\ni t\mapsto \|X(t)\|_\H$. 
To this end, we recall  \eqref{GH}, \eqref{H M SPDE}, \eqref{tt X Hs}, and reformulate  \eqref{EN-Ito} as 
\begin{align*}
\d X(t)-(\tt\EE X)(t)\d t=\,&\big\{ b(X)+\hh g(X(t)) \big\}\d t+ h(t,X(t))\d \W(t),\ \ t\ge0,\\ \hh g(X(t)):=\,& g(X(t))-(\tt\EE X)(t), \ \  t\ge 0.
\end{align*} 
By \ref{Assum-E}, \eqref{tt Pi}, \eqref{tt Pi pi0} and $[J_n,\tt\EE]=0$,   we arrive at
\begin{equation*}    
-\<J_n \tt\EE X,J_n X\>_{H^{s_0}}= \|\G J_n\D^{s_0} X\|^2_{L^2},\ \ X\in\H. 
\end{equation*}
This, Lemmas \ref{LJN} and \ref{LAK}, \ref{Assum-b} and \ref{Assum-g}  yield that \ref{B} holds true for $\hh g$ replacing $g$ and $T_n=J_n$, so that as in \eqref{C0T},   for all $n,\, N\ge 1$, there is a constant $K_N>0$ such that
\begin{equation}\label{C0T'}
\left\{\begin{aligned}
&\d \<M^{(n)}\>(t)\le K_N \d t,\ \ t\in [0,\tau_N],\\
-K_N \d t\le &\d \|J_n X(t)\|_\H^2 + 2\|\mathcal{G} J_nX(t)\|^2_{\H}\d t + \d M^{(n)}(t) \le K_N \d t,\ \ t\in [0,\tau_N],
\end{aligned}\right.
\end{equation}
for some martingales $M^{(n)}$.
This implies 
$$\sup_{n\ge 1} \E  \left[\sup_{t\in[0,\tau_N]}\|J_n X(t)\|^2_{\H}+2 \int_0^{\tau_N}\|\mathcal{G} J_nX(t)\|^2_{\H}\d t\bigg|\F_0\right]<\infty,\ \ N\ge 1.$$
By Fatou's lemma  
\begin{equation}\label{E X+GX F0} \E\left[\sup_{t\in[0,\tau_N]}\|X(t)\|^2_{\H}+2 \int_0^{\tau_N}\|\mathcal{G} X(t)\|^2_{\H}\d t\bigg|\F_0\right]<\infty,\ \ N\ge 1.
\end{equation}       
This implies that the stopping times 
\begin{equation*} 
\tt\tau_N:=N\land 
\inf\bigg\{t\ge 0: \|X(t)\|_\H+ 
\int_0^t\|\mathcal{G}X(s)\|^2_\H \d s\ge N\bigg\},\ \ N\ge 1
\end{equation*}    
satisfies $\tt\tau_N\le \tau_N$ and 
$
\p\Big(\lim_{N\to\infty} \tt\tau_N= \tau^*\Big)=1.
$
Let $\eta(t):=\|X(t)\|^2_\H+2 \int_0^t\|\mathcal{G}X(s)\|^2_\H \d s$. Then by  
the  argument leading to  \eqref{CTT}, \eqref{C0T'} implies  
$$\E\Big[\big|\eta(t\land\tt\tau_N)-\eta(s\land\tt \tau_N)\big|^4\Big]\le G(N)|t-s|^2,\ \ t,s\ge 0,\ N\ge 1 $$ for some 
map $G: \mathbb N\to (0,\infty)$. 
By  Kolmogorov's continuity theorem,  $\p\big(\eta(\cdot) \in C([0,\tau^*))\big)=1$. On the other hand,  \eqref{E X+GX F0} implies that  $\p$-{\rm a.s.},    $[0,\tau^*)\ni t\mapsto 2 \int_0^{t}\|\mathcal{G}X(s)\|^2_\H \d s$  is continuous. Therefore, as desired,
$\|X(\cdot)\|^2_\H\in C([0,\tau^*))$ $\p$-{\rm a.s.}

\ref{T3.1-b} By   \ref{Assum-b}, \eqref{Check R4-1} and \eqref{Check R4-2}, we have   
\begin{align*} 
&\|\tt b(X)\|_{\H}\leq\, \tt K\big(\B(X)\big)\|X\|_{\H},\\
\sum_{k=1}^\infty\big\<h_n(t,&X) e_k, X\big\>_{\H}^2 
\le \hh K (t,\B(X))) \big(1+ \|X\|_{\H}^4\big),\\
2\big\<g_n(X), X\big\>_{\H} +&\|h_n(t,X)\|_{\LL_2(\U;\H)}^2 
\le\,  \hh K (t,\B(X))  (1+\|X\|_{\H}^2).
\end{align*}
So,  Proposition \ref{P1} implies  Theorem \ref{T3.1} \ref{T3.1-b}.

\ref{T3.1-c}   
Recall the coefficients $(b,g,h)$  in \eqref{GH} and the definitions of $\H$ and $\M$ in \eqref{H M SPDE}. Let  
$$\mathfrak{J}(t,X):=2\left\<\tt b(X)+\tt g(X), X\right\>_{H^{\theta_0}} +\mathfrak{L}(t,X),$$
$$\mathfrak{L}(t,X):=\sum_{k=1}^\infty 
\bigg(\left\|\tt\Pi\tt h_k(t,X)\right\|_{H^{\theta_0}}^2 - \frac{2\big\<\tt\Pi\tt h_k(t,X),X\big\>_{H^{\theta_0}}^2}{{\rm e}+ \|X\|_{H^{\theta_0}}^2}\bigg).$$
Taking $\si=\theta_0$ and letting $n \to\infty$ in \eqref{LO2-Ak} with noticing that $\tt\Pi$ is a self-adjoint projection, we obtain 
\begin{align*}                                                       
\sum_{k=1}^\infty&\bigg|\left\<(a_k\tt\J_k)^2 X, X\right\>_{H^{\theta_0}} + \left\|a_k\tt\J_k X\right\|_{H^{\theta_0}}^2\bigg| \notag\\
&+ \sum_{k=1}^\infty\bigg|\left\<(q_k\tt\K_k)^2 X, X\right\>_{H^{\theta_0}} + \left\|q_k\tt\K_k X\right\|_{H^{\theta_0}}^2\bigg|
\lesssim\,\|X\|_{H^{\theta_0}}^2,\ X\in\tt  H^{s_0}.
\end{align*} 
Notice that, by \ref{Assum-E},
$\big\<\tt\EE X,X\big\>_{H^{\theta_0}}\leq 0$, $X\in \tt  H^{s_0}.$ 
From these estimates and \ref{Assum-hbg}, we have
\begin{align*}
\mathfrak{J}(t,X) 
\leq\, K(\|X\|_{H^{\theta_2}})\|X\|_{H^{\theta_0}}^2
\bigg\{1+\frac{1}{K(\|X\|_{H^{\theta_2}})\|X\|_{H^{\theta_0}}^2} 
\sup_{t\in [0,T]}\mathfrak{L}(t,X)\bigg\}.
\end{align*}
Therefore \eqref{NE SPDE} implies
\begin{equation*}
\limsup_{\|X\|_{H^{\theta_0}}\to\infty}\sup_{t\in [0,T]} \frac{\mathfrak{J}(t,X(t))}{({\rm e}+\|X\|_{H^{\theta_0}}^2)\log ({\rm e}+\|X\|_{H^{\theta_0}}^2)}<\infty,\ \ T\in(0,\infty).
\end{equation*} 
Hence \ref{C} holds for $V(x):=\log({\rm e}+x)$, so that the proof is finished by Theorem \ref{T1} \ref{T1-global}.
\end{proof}

\section{Application to specific models with pseudo-differential noise}
\label{Section:Applications}

In this part, we apply Theorem \ref{T3.1} to specific models mentioned in Section \ref{Section:Results on moels}.
We state our results for each model in separate sections. 
As before, for $d,m\ge1$ and $s\ge0$, we write $$H^s=H^s(\mathbb K^d,\R^m),\ \ \mathbb K=\R\ \text{or}\ \T.$$
As mentioned in Remark \ref{T3.1-Remark}, $\big(\tt b,\tt g\big)$ in following examples turns out to satisfy \ref{Assum-b} and \ref{Assum-g} with $s_0$ being replaced by $s$ such that $s$ is arbitrary in some range (see Lemmas \ref{g-MHD lemma}, \ref{g-ch lemma}, \ref{g-ad lemma} and \ref{g-sqg lemma}). 
Therefore we also state the following assumption  stronger than \ref{Assum-h}:
\begin{ManualHypo}{$ ({\bf F}_h^\prime)$}\label{Assum-h-s} 
There exists $l\ge1$ such for all $s>\frac{d}{2}+l$,
$\tt h_k:[0,\infty)\times H^s\to H^s$ satisfies
\begin{equation*}
\begin{split}
\sum_{k\ge1}\|\tt h_k(t,X)\|^2_{H^{s}}
\leq\, &K(t,\|X\|_{W^{k,\infty}})(1+\|X\|^2_{H^{s}}),\ \ t\ge0,\ X\in H^s,\\
\sum_{k=1}^\infty \| \tt h_k(t,X)- \tt h_k(t, Y)\|_{H^{s}}^2 
\le\, &K(t, \|X\|_{H^{s}}+\|Y\|_{ H^{s}})\|X-Y\|^2_{H^{s}},\ \ t\ge0,\ X,\,Y\in H^s.
\end{split}
\end{equation*}
\end{ManualHypo}
Obviously, since $H^{s}\hookrightarrow W^{l,\infty}$ with $s>\frac{d}{2}+l$,  \ref{Assum-h-s} implies \ref{Assum-h} with $\theta_1\in(\frac{d}{2}+l,\infty)$ and $s_0=s$.

\subsection{Stochastic \textbf{MHD} equations}
Consider the stochastic \textbf{MHD} equations
\begin{align}
\d X(t)=\,& \Big\{\big(\tt\Pi\EE^{{\rm mhd}} X\big)(t)+ g^{{\rm mhd}}(X(t)) \Big\}\d t\nonumber\\ &+\sum_{k=1}^\infty \Big\{ \big(\tt\Pi\A_kX\big)(t) \circ \d W_k(t)
+\tt\Pi\tt h_k (t,X(t))\d \tt W_k(t)\Big\},\ \ t\ge0,\label{SMHD}
\end{align}
which is a special case of \eqref{EN} with
$$d\ge1,\ m=2d,\ X=(V,M)^T,\ \tt b\equiv 0,\ \EE=\EE^{{\rm mhd}},\ \tt g=g^{{\rm mhd}},$$
$$\tt\Pi={\rm diag}(\Pi_d,\Pi_d) \ \ \text{if}\ \   \mathbb K=\R\ \ \text{and}\ \ \tt\Pi={\rm diag}(\Pi_d\Pi_0,\Pi_d\Pi_0)\ \ \text{if}\ \ \mathbb K=\T.$$ 
Here we recall that $\EE^{{\rm mhd}}$ and $g^{{\rm mhd}}$ are given in \eqref{MHD-E-g} as
\begin{equation*} 
\begin{cases}
\EE^{{\rm mhd}}:=\,-{\rm diag} \big(\mu_1 \Lambda^{2\alpha_1},\ \mu_2\Lambda^{2\alpha_2}\big),\\
g^{{\rm mhd}}(X):=\,\Big( \Pi (M\cdot\nn ) M-\Pi (V\cdot\nn )V,\ (M\cdot\nn)V-(V\cdot\nn )M\Big)^T.
\end{cases} 
\end{equation*}
and $\Pi_d$, $\Pi_0$ are defined in \eqref{Pi-d define} and \eqref{Pi-0 define}, respectively. Recalling $H^{s}_{{\rm div}}(\mathbb K^d; \mathbb R^{d})$ in \eqref{H-div}, and then using $m=2d$, \eqref{Pi-d bound}  and \eqref{Pi-0 bound}, we have
$$\tt H^s:= \tt\Pi H^{s}(\mathbb K^d; \mathbb R^{m})= H^s_{{\rm div}}(\mathbb K^d; \mathbb R^{d})\times H^s_{{\rm div}}(\mathbb K^d; \mathbb R^{d}).$$

\begin{Lemma}\label{g-MHD lemma} Let $g_n^{{\rm mhd}}(X):= J_n g^{{\rm mhd}}(J_nX)$, $n\ge 1.$ 
Then for all $\si>1+\frac{d}{2}$, $g^{{\rm mhd}}:\tt H^{\si}\to \tt H^{\si-1}$ and for all $X,Y\in \tt H^{\si}$, 
\begin{equation}
\sup_{n\ge 1}\|g_n^{{\rm mhd}}(X)\|_{H^{\si-1}}+
\|g^{{\rm mhd}}(X)\|_{H^{\si-1}}
\lesssim \|X\|_{W^{1,\infty}}\|X\|_{H^{\si}},\label{g-mhd Hs}
\end{equation}
\begin{equation}
\big|\big\<g^{{\rm mhd}}(X)- g^{{\rm mhd}}(Y),X-Y\big\>_{H^{\si-1}} \big|
\lesssim (\|X\|_{H^{\si}}+\|Y\|_{H^{\si}})
\|X-Y\|^2_{H^{\si-1}},\label{g-mhd x-y}
\end{equation}
\begin{align}
\sup_{n\ge 1}\bigg\{ \left|
\left\< g_{n}^{{\rm mhd}}(X),X\right\>_{H^{\si-1}}\right|+ 
\left|\left\< J_n g^{{\rm mhd}}(X),J_n X\right\>_{H^{\si-1}}\right|\bigg\}
\lesssim \,&\|X\|_{W^{1,\infty}}\|X\|^2_{H^{\si-1}}.\label{g-mhd y}
\end{align}
Moreover, for any $\zeta\in(\frac{d}{2},\si-1)$, there exists $\lambda:\mathbb N\times\mathbb N\to (0,\infty)$ with $\lambda_{n,l}\to 0$ as $n,l\to\infty$ such that for all $X, Y\in \tt H^{\si}$ and $n,l\ge 1$,
\begin{align}
&\big|\big\<g_{n}^{{\rm mhd}}(X)-g_l^{{\rm mhd}}(Y),X-Y\big\>_{H^{\zeta}} \big| \notag\\
\leq\, &\big(1+\|X\|^4_{H^{\si}}+\|Y\|^4_{H^{\si}}\big)
\left(\lambda_{n,l}+\|X-Y\|^2_{H^{\zeta}}\right). \label{g-mhd-m-n x-y}
\end{align}
\end{Lemma}
\begin{proof}
Due to the divergence free condition, $g^{{\rm mhd}}:\tt H^{\si}\to \tt H^{\si-1}$. Moreover, by Lemmas \ref{LJN} and \ref{LOPN},  
\eqref{g-mhd Hs} holds true. 

To prove \eqref{g-mhd x-y}, let $X=(X_1,X_2), Y=(Y_1,Y_2)$. For simplicity, we let $F=X_1-Y_1$ and $H=X_2-Y_2$. By \eqref{MHD-E-g}, we have 
\begin{align*}
&-\big\<g^{{\rm mhd}}(X)-g^{{\rm mhd}}(Y),X-Y\big\>_{H^{\si-1}}\\
=\,& \big\<(X_1\cdot\nn)X_1- (Y_1\cdot\nn)Y_1,F \big\>_{H^{\si-1}}
+\big\<(Y_2\cdot\nn)Y_2- (X_2\cdot\nn)X_2, F \big\>_{H^{\si-1}}\\
&+ \big\<(X_1\cdot\nn)X_2- (Y_1\cdot\nn)Y_2, H \big\>_{H^{\si-1}}
+\big\<(Y_2\cdot\nn)Y_1-(X_2\cdot\nn)X_1, H \big\>_{H^{\si-1}}\\
:=\,&\sum_{i=1}^4 I_i,
\end{align*}
Since ${\si-1}>\frac{d}{2}$, we have $H^{\si-1}\hookrightarrow L^{\infty}$. Combining this with Lemmas \ref{LOPN} and \ref{LOP5} (with $\Q=\D^{\si-1}$ and $q=0$) and the divergence free condition, we obtain
\begin{align*}
|I_1|
\lesssim \left(\|X_1\|_{H^{\si}}+\|Y_1\|_{H^{\si}}\right)\|F\|^2_{H^{\si-1}},
\end{align*}
\begin{align*}
|I_{3}|
\lesssim \|X_2\|_{H^{\si}}\|F\|_{H^{{\si-1}}}\|H\|_{H^{{\si-1}}}
+\|Y_1\|_{H^{\si}}\|H\|^2_{H^{\si-1}}.
\end{align*}
On the other hand, by using the divergence free conditions and integration by parts, we have 
\begin{align*}
-I_2-I_4=\,&
\Big\<\D^{\si-1}((H\cdot\nn) X_2),\D^{\si-1} F \Big\>_{L^2}
+ \Big\<[\D^{\si-1}, (Y_2\cdot\nn)]H,\D^{\si-1}F\Big\>_{L^2}\\
&\Big\<\D^{\si-1}((H\cdot\nn) X_1),\D^{\si-1}H \Big\>_{L^2}+ \Big\<[\D^{\si-1}, (Y_2\cdot\nn)]F,\D^{\si-1}H\Big\>_{L^2}.
\end{align*}
By $H^{\si-1}\hookrightarrow L^{\infty}$ for $s >\frac{d}{2}$, Lemmas \ref{LOPN} and Lemma \ref{LOP5}, we arrive at 
\begin{align*} 
&|I_{2}+I_{4}| 
\lesssim
\|X_2\|_{H^{\si}}\|H\|_{H^{\si-1}}\|F\|_{H^{\si-1}}
+\|Y_2\|_{H^{\si}}\|F\|_{H^{\si-1}}\|H\|_{H^{\si-1}}
+\|X_1\|_{H^{\si}}\|H\|^2_{H^{\si-1}}.\end{align*} 
Collecting the above estimates, we obtain \eqref{g-mhd x-y}.

Concerning \eqref{g-mhd y}, we only prove the estimate on $ \left\< J_n g^{{\rm mhd}}(X),J_n X\right\>_{\tt H^{\si}},$ since the other one can be derived similarly (even simpler because $J_n$ is self-adjoint). 
We write 
\begin{align*}
&\left\<J_n g^{{\rm mhd}}(X),J_nX\right\>_{\tt H^{\si-1}}\\
=\, & -\left\<[\D^{\si-1},(X_1\cdot\nn)]X_1,\D^{\si-1}J_n^2X_1 \right\>_{L^2}
-\left\<(X_1\cdot\nn)\D^{\si-1} X_1,\D^{\si-1}J_n^2X_1 \right\>_{L^2}\nonumber\\
&+ \left\<[\D^{\si-1},(X_2\cdot\nn)]X_2,\D^{\si-1}J_n^2X_1 \right\>_{L^2}
+ \left\<(X_2\cdot\nn)\D^{\si-1}X_2,\D^{\si-1}J_n^2X_1 \right\>_{L^2}\nonumber\\
&- \left\<[\D^{\si-1},(X_1\cdot\nn)]X_2,\D^{\si-1}J_n^2X_2 \right\>_{L^2}
- \left\<(X_1\cdot\nn)\D^{\si-1}X_2,\D^{\si-1}J_n^2X_2 \right\>_{L^2}\nonumber\\
&+ \left\<[\D^{\si-1},(X_2\cdot\nn)]X_1,\D^{\si-1}J_n^2X_2 \right\>_{L^2}
+ \left\<(X_2\cdot\nn)\D^{\si-1} X_1,\D^{\si-1}J_n^2X_2 \right\>_{L^2}\nonumber\\
:=\, & \sum_{i=1}^{8}N_i. 
\end{align*}
Using Lemma \ref{LOP4} (with $\Q=\D^{\si-1}$), we have 
\begin{align*}
|N_1|,\, |N_3|,\, |N_5|,\, |N_7| \lesssim\, (\|X_1\|_{W^{1,\infty}}+\|X_2\|_{W^{1,\infty}})(\|X_2\|^2_{H^{\si-1}}+\|X_1\|^2_{H^{\si-1}}).
\end{align*}
For $N_2$, we use the divergence free conditions and Lemma \ref{LJN} to obtain,
\begin{align*}
|N_2|
=\left| \left\<\left[J_n,(X_1\cdot\nn)\right]\D^{\si-1} X_1,\D^{\si-1}J_n X_1 \right\>_{L^2}\right|
\lesssim \|X_1\|_{W^{1,\infty}}\|X_1\|^2_{H^{\si-1}}.
\end{align*}
Similarly, 
\begin{align*}
|N_6|
=\left| \left\<\left[J_n,(X_1\cdot\nn)\right]\D^{\si-1}X_2,\D^{\si-1}J_n X_2\right\>_{L^2}\right|
\lesssim \|X_1\|_{W^{1,\infty}}\|X_2\|^2_{H^{\si-1}}.
\end{align*}
Again, by the divergence free conditions,
\begin{align*}
&N_4+N_8\\
=\,&\left\<\left[J_n,(X_2\cdot\nn)\right]\D^{\si-1}X_2,\D^{\si-1}J_n X_1 \right\>_{L^2}
+ \left\<(X_2\cdot\nn)\D^{\si-1}J_n X_2,\D^{\si-1}J_n X_1 \right\>_{L^2}\\
&+\left\<\left[J_n,(X_2\cdot\nn)\right]\D^{\si-1} X_1,\D^{\si-1}J_n X_2 \right\>_{L^2}
+ \left\<(X_2\cdot\nn)\D^{\si-1}J_n X_1,\D^{\si-1}J_n X_2 \right\>_{L^2}\\
=\,&\left\<\left[J_n,(X_2\cdot\nn)\right]\D^{\si-1}X_2,\D^{\si-1}J_n X_1 \right\>_{L^2}
+\left\<\left[J_n,(X_2\cdot\nn)\right]\D^{\si-1} X_1,\D^{\si-1}J_n X_2 \right\>_{L^2}.
\end{align*}
Therefore, we use Lemma \ref{LJN} to find
\begin{align*}
|N_4+N_8|
\lesssim & \|X_2\|_{W^{1,\infty}}\|X_1\|_{H^{\si-1}}\|X_2\|_{H^{\si-1}}.
\end{align*} These imply the desired upper bound for $\big|\left\< J_n g^{{\rm mhd}}(X),J_n X\right\>_{H^{\si-1}}\big| $.

Now we prove \eqref{g-mhd-m-n x-y}. 
We have 
\begin{align*}
&J_n g^{{\rm mhd}}(J_n X)-J_l g^{{\rm mhd}}(J_l Y) \\
=\,&(J_n-J_l)g^{{\rm mhd}}(J_n X)+J_l[g^{{\rm mhd}}(J_n X)-g^{{\rm mhd}}(J_l X)]
+J_l[g^{{\rm mhd}}(J_l X)-g^{{\rm mhd}}(J_l Y)]\\
:=\,&\sum_{i=1}^{3} Q_i.
\end{align*}
To estimate $Q_2$, 
we write \begin{align*}
g^{{\rm mhd}}(J_n X)-g^{{\rm mhd}}(J_l X)=\left(\Pi p_{2,1,1}-\Pi p_{2,1,2},\ p_{2,2,1}-p_{2,2,2}\right)^T,
\end{align*}
where 
\begin{align*}
p_{2,1,1}:=\,&((J_n-J_l)X_2\cdot\nn )J_n X_2+(J_l X_2\cdot\nn ) (J_n-J_l)X_2,\\
p_{2,1,2}:=\,&((J_l-J_n)X_1\cdot\nn )J_l X_1+(J_n X_1\cdot\nn ) (J_l-J_n)X_1,\\
p_{2,2,1}:=\,&(J_n-J_l)X_2\cdot\nn )J_n X_1+(J_l X_2\cdot\nn ) (J_n-J_l)X_1,\\
p_{2,2,2}:=\,&(J_l-J_n)X_1\cdot\nn )J_l X_2+(J_n X_1\cdot\nn ) (J_l-J_n)X_2.
\end{align*}
We infer from Lemmas \ref{LJN} and \ref{LOPN}, $H^{\zeta}\hookrightarrow L^{\infty}$ and divergence free conditions to find for $\epsilon\in(0,\si-1-\zeta)$,
\begin{align*}
\left| \big\<J_l\Pi p_{2,1,1},X_1-Y_1 \big\>_{H^{\zeta}}\right|+\left| \big\<J_l\Pi p_{2,1,2},X_1-Y_1 \big\>_{H^{{\zeta}}}\right|
\lesssim (l\land n)^{-2\epsilon}\|X\|^4_{H^{\si}}+\|X-Y\|^2_{H^{\zeta}},
\end{align*}
\begin{equation*}
\left| \big\<J_l p_{2,2,1},X_2-Y_2 \big\>_{H^{\zeta}}\right|+
\left| \big\<J_l p_{2,2,2},X_2-Y_2 \big\>_{H^{\zeta}}\right|
\lesssim (l\land n)^{-2\epsilon}\|X\|^4_{H^{\si}}+\|X-Y\|^2_{H^{\zeta}}.
\end{equation*}
These estimates yield
\begin{align*}
\left|\left\<Q_2,X-Y \right\>_{H^{\zeta}}\right|\lesssim (l\land n)^{-2\epsilon}\|X\|^4_{H^{\si}}+\|X-Y\|^2_{H^{\zeta}}.
\end{align*}
Similarly, for $Q_1$, we have
\begin{align*}
\left|\left\<Q_1,X-Y \right\>_{H^{\zeta}}\right|
\leq&
\|\left(J_n-J_l\right) g^{{\rm mhd}}(J_n X)\|_{H^{\zeta}}\|X-Y\|_{H^{\zeta}}\\
\lesssim\, & (l\land n)^{-\epsilon}\|X\|^2_{H^{\si}}\|X-Y\|_{H^{\zeta}}
\lesssim\, (l\land n)^{-2\epsilon}\|X\|^4_{H^{\si}}+\|X-Y\|^2_{H^{\zeta}}.
\end{align*}
Finally, as $\zeta>d/2$, we use \eqref{g-mhd x-y} to find
\begin{align*}
\left|\left\<Q_3,X-Y \right\>_{H^{\zeta}}\right|
\lesssim \left(\|X\|_{H^{\si}}+\|Y\|_{H^{\si}}\right)\|X-Y\|^2_{H^{\zeta}}.
\end{align*}
Collecting the above estimates, we obtain \eqref{g-mhd-m-n x-y} and finish the proof.
\end{proof}
Recall the $\alpha_1,\alpha_2\in [0,1]$ and $\mu_1,\mu_2\ge 0$ come from the dissipation term $\mu_1\Lambda^{2\alpha_1}$ and $\mu_2\Lambda^{2\alpha_2}$. Let 
$$\alpha_0:= \max\big\{ \alpha_1\textbf{1}_{\{\mu_1>0\}},\alpha_2\textbf{1}_{\{\mu_2>0\}}\big\}.$$
We also recall $r_0$ given in \eqref{tt Pi U} and $H^{s}_{{\rm div}}(\mathbb K^d; \mathbb R^{d})$ defined in \eqref{H-div}. Then we have the following results on stochastic \textbf{MHD} equations \eqref{SMHD}:
\begin{Theorem}\label{SMHD-T} Let $d\ge2$, $m=2d$ and $\mathbb K=\R$ or $\T$. 
Assume {\rm \ref{Ak-assum}} with \eqref{tt Pi U} holding for $\tt\Pi={\rm diag}(\Pi_d,\Pi_d)$ if $\mathbb K=\R$ and $\tt\Pi={\rm diag}(\Pi_d\Pi_0,\Pi_d\Pi_0)$ if $\mathbb K=\T$. Assume {\rm \ref{Assum-h-s}} for  some $l\ge1$. Then, for all  $s>l+\frac d 2+ \max\{1,2r_0, 2\alpha_0\}$ and any $\F_0$-measurable $H^s_{{\rm div}}(\mathbb K^d; \mathbb R^{d})\times H^s_{{\rm div}}(\mathbb K^d; \mathbb R^{d})$-valued random variable $X(0)$, the following assertions hold. 
\begin{enumerate}[label={ $(\arabic*)$}]
\item \eqref{SMHD}  has a unique maximal solution $(X,\tau^*)$ such that Definition \ref{solution definition} is fulfilled with
$$\H=\H^s:= H^s_{{\rm div}}(\mathbb K^d; \mathbb R^{d})\times H^s_{{\rm div}}(\mathbb K^d; \mathbb R^{d}),$$
$$ \M=\M^\theta:=H^{\theta}_{{\rm div}}(\mathbb K^d; \mathbb R^{d})\times H^{\theta}_{{\rm div}}(\mathbb K^d; \mathbb R^{d}),\ \ \theta\in \left(l+\frac{d}{2},s-\max\{1,2r_0, 2\alpha_0\}\right).$$    
Besides, $(X,\tau^*)$ defines a map $\H^s\ni X(0)\mapsto X(t)\in C([0,\tau^*);\H^s)$ $\p$-{\rm a.s.},
where $\tau^*$ does not depend on $s$, and 
\begin{equation*}
\limsup_{t\rightarrow \tau^*}\|X(t)\|_{W^{l,\infty}}=\infty \ \text{on}\ \{\tau^*<\infty\}.
\end{equation*}

\item The solution is non-explosive, if $\Psi(T,X,\theta)$  defined in \eqref{NE h-k} satisfies  
\begin{equation*} 
\limsup_{\|X\|_{H^{\theta}} \to \infty}
\frac{\Psi(T,X,\theta)}{\|X\|_{W^{1,\infty}}\|X\|_{H^{\theta}}^2} 
<-1,\ \ T\in(0,\infty). 
\end{equation*}
\end{enumerate}
\end{Theorem}

\begin{proof}
It is easy to check that \ref{Assum-E} holds for $\EE=\EE^{{\rm mhd}}$ with $m=2d$, $p_0=\alpha_0$  and  
\begin{equation*}
\G_j:=
\begin{cases}
\sqrt{\mu_1}\Lambda^{\alpha_1},\  \ &\text{if}\ \ \alpha_1<1,\  1\le j\le d,\\
\sqrt{\mu_1}\pp_j,\ &\text{if}\ \ \alpha_1=1,\  1\le j\le d,\\
\sqrt{\mu_2}\Lambda^{\alpha_2},\ \ &\text{if}\ \ \alpha_2<1,\ d+1\le j\le 2d\\
\sqrt{\mu_2}\pp_{j-d},\ &\text{if}\ \ \alpha_2=1,\ d+1\le j\le 2d.
\end{cases}
\end{equation*}

Obviously, \ref{Assum-b} holds since $\tt b\equiv0$. 
Since $H^{\theta}\hookrightarrow\Wlip$,
we can infer from Lemmas \ref{LJN} and \ref{g-MHD lemma} that 
\ref{Assum-g} holds with $\tt g= g^{{\rm mhd}}$, $q_0=1$, $m=2d$,  $\theta_2=\theta$, 
$\tt K(\|X\|_{H^{\theta_2}})$ being replaced by $C\|X\|_{\Wlip}$ for  some $C>0$, $\tt\Pi={\rm diag}(\Pi_d,\Pi_d)$ if $\mathbb K=\R$ and $\tt\Pi={\rm diag}(\Pi_d\Pi_0,\Pi_d\Pi_0)$ if $\mathbb K=\T$.

Therefore, the statement of Theorem \ref{SMHD-T} follows from Theorem \ref{T3.1} \ref{T3.1-a},  \ref{T3.1-b} with  $\B(X)=\|X\|_{W^{l,\infty}}$, and \ref{T3.1-c} respectively.   The fact that $\tau^*$ is independent of $s$ has already been pointed out in Remark \ref{T3.1-Remark}.
\end{proof}

\subsection{Stochastic \textbf{CH} type equations}

We consider the stochastic \textbf{CH} equation given by \eqref{EN} with $d=m=1$,  $\EE\equiv 0$, $\tt\Pi=\I$, $\tt b=b^{{\rm ch}}$ and $\tt g=g^{{\rm ch}}$, that is,
\begin{align}
\d X(t)=\,& \Big\{b^{{\rm ch}}(X(t)) +g^{{\rm ch}}(X(t)) \Big\}\d t\notag\\
&+\sum_{k=1}^\infty \Big\{ (\A_kX)(t) \circ \d W_k(t)+  \tt h_k (t,X(t))\d \tt W_k(t)\Big\},\ \ t\ge0,\label{SCH}
\end{align}
where $b^{{\rm ch}}$ and $g^{{\rm ch}}$ are given in \eqref{CH-b-g}
as
\begin{equation*} 
b^{{\rm ch}}(X):= -\pp (\I-\pp^2)^{-1}\Big(\sum_{i=1}^4 a_i X^i+a |\pp X|^2\Big),\ \ 
g^{{\rm ch}}(X):= -X\pp X.
\end{equation*} 
Obviously, the solution space in this case is  $\tt H^s(\mathbb K;\R)=H^s(\mathbb K;\R).$ 
According to \cite{Tang-Yang-2022-AIHP}, we have 
\begin{align}\label{b-CH estimates}
\left\{\begin{aligned}
\|b^{{\rm ch}}(X)\|_{H^s}
\lesssim\,&\psi\left(\|X\|_{W^{1,\infty}}\right)\|X\|_{H^s},\ \ s>\frac{3}{2},\ \ X\in H^s,\\
\|b^{{\rm ch}}(X)-b^{{\rm ch}}(Y)\|_{H^s}
\lesssim \,&\psi\left(\|X\|_{H^{s}}+\|Y\|_{H^{s}}\right)\|X-Y\|_{H^s},\ \ s>\frac{3}{2},\ \ X,\, Y\in H^s,
\end{aligned}
\right.
\end{align}
where 
\begin{align}\label{CH-g-growth}
\psi(x)=|a_1|+(1+|a_2|+|a|)x +|a_3|x^2+ |a_4|x^3.
\end{align}

\begin{Lemma}\label{g-ch lemma}
Let   $g_{n}^{{\rm ch}}(X):=\, J_n g^{{\rm ch}}(J_n X)=-J_n\left[J_n XJ_n (\pp X)\right], n\ge 1.$ 
Then for any  $\sigma >1+\frac{1}{2}$ and $X\in H^{\si}$, we have
\begin{equation*}
\sup_{n\ge 1}\left\|g^{{\rm ch}}_n(X)\right\|_{H^{\si-1}}+ \left\|g^{{\rm ch}}(X)\right\|_{H^{\si-1}}
\lesssim \|X\|_{W^{1,\infty}}\|X\|_{H^{\si}},
\end{equation*}
\begin{equation*}
\sup_{n\ge 1}\bigg\{ \left|\Big\< g^{{\rm ch}}_n(X),X\Big\>_{H^{\si-1}}\right|+ 
\left|\Big\< J_n g^{{\rm ch}}(X),J_n X\Big\>_{H^{\si-1}}\right|\bigg\}
\lesssim \|X\|_{W^{1,\infty}}\|X\|^2_{H^{\si-1}}.
\end{equation*}
Besides, for any $\zeta\in(\frac{1}{2},\si-1)$, 
there is $\lambda:\mathbb N\times\mathbb N\to (0,\infty)$ with $\lambda_{n,l}\to 0$ as $n,l\to\infty$ such that for any $n,l\ge 1$ and $X,Y\in H^{\si}$, 
\begin{align*}
&\left|\Big\<g_{n}^{{\rm ch}}(X)-g_l^{{\rm ch}}(v),X-Y\Big\>_{H^{\zeta}} \right|
\leq\, \big(1+\|X\|^4_{H^{\si}}+\|Y\|^4_{H^{\si}}\big)
\left(\lambda_{n,l}+\|X-Y\|^2_{H^{\zeta}}\right).
\end{align*}
\end{Lemma}
\begin{proof}
One can verify the desired estimates
as in the proof of Lemma \ref{g-MHD lemma}, so we skip the 
details.
\end{proof}

Since for the present model we have $\tt\Pi=\I$ so that \ref{Ak-assum-Pi} is trivial, in the following theorem we only assume {\rm \ref{Ak-assum}}  without this condition. 

\begin{Theorem}\label{SCH-T} Let $\mathbb K=\R$ or $\T$. Let $d=m=1$.
Assume {\rm \ref{Ak-assum}} without \ref{Ak-assum-Pi},  and  assume  {\rm \ref{Assum-h-s}} for some $l\ge1$.
For any  $s>\frac{1}{2}+l+\max\{1,2r_0\}$, where $r_0$ is given in \eqref{tt Pi U}, and any $\F_0$-measurable $H^s(\mathbb K;\R)$-valued random variable $X(0)$,  the following assertions hold. 
\begin{enumerate}[label={ $(\arabic*)$}]
\item \eqref{SCH}  has  a unique maximal solution $(X,\tau^*)$  in the sense of Definition \ref{solution definition} for 
$\H=\H^s:= H^s(\mathbb K;\R)$ and $\M=\M^\theta:=  H^{\theta}(\mathbb K;\R)$ with $\theta\in \left(\frac 1 2+l, s- \max\{1,2r_0\}\right).$  
Moreover, $(X,\tau^*)$ defines a map $\H^s\ni X(0)\mapsto X(t)\in C([0,\tau^*);\H^s)$ $\p$-{\rm a.s.},
where $\tau^*$ is independent of $s$, and 
\begin{equation*}
\limsup_{t\rightarrow \tau^*}\|X(t)\|_{W^{l,\infty}}=\infty \ \text{on}\ \{\tau^*<\infty\}.
\end{equation*}

\item $\p(\tau^*=\infty)=1$ provided that 
\begin{equation*} 
\limsup_{\|X\|_{H^{\theta}} \to\infty}\frac{\Psi(T,X,\theta)}{\psi(\|X\|_{W^{1,\infty}})\|X\|^2_{H^{\theta}}} <-1,
\end{equation*}
where $\Psi(T,X,\theta)$ and $\psi(\cdot)$ are given in \eqref{NE h-k}  and \eqref{CH-g-growth}, respectively.
\end{enumerate}
\end{Theorem}
\begin{proof}
In this case $\EE\equiv0$ and $\tt b=b^{{\rm ch}}$. With \eqref{b-CH estimates} at hand, one can check \ref{Assum-hbg} holds with $s_0=s$ and $\theta_0=\theta$.
The details are similar to the proof of Theorem \ref{SMHD-T} and we skip them for brevity.
\end{proof}

\subsection{Stochastic \textbf{AD} equation}

Consider the stochastic \textbf{AD} equation embedded in \eqref{EN} with $d\ge2$, $m=1$, $\tt \Pi=\I$, $\tt b\equiv0$, $\EE=\EE^{{\rm ad}}$ and $\tt g=g^{{\rm ad}}$, where $\EE^{{\rm ad}}$ and $g^{{\rm ad}}$ are given in \eqref{AD-E-g}. That is,
\begin{align}\label{SAD}
\d X(t)=\, &\Big\{(\EE^{{\rm ad}} X)(t)+ g^{{\rm ad}}(X(t)) \Big\}\d t \notag\\
&+\sum_{k=1}^\infty \Big\{ (\A_kX)(t) \circ \d W_k(t)
+ \tt h_k (t,X(t))\d \tt W_k(t)\Big\},\ \  t\ge0,
\end{align} 
where
\begin{equation*}
\EE^{{\rm ad}}:=-\nu(-\DD)^{\beta},\ \ g^{{\rm ad}}(X):= - \gamma\nn\cdot\left(X\BB X\right), \ \ \BB= (\BB_i)_{1\le i\le d}:=\big(\pp_i[\Phi\star]\big)_{1\le i\le d}.
\end{equation*} 

Then the working space for \eqref{SAD} is  $\tt H^s=\tt\Pi H(\mathbb K^d;\R)=H^s(\mathbb K^d;\R)$ and we have the following
\begin{Lemma}\label{g-ad lemma}
Let $\Phi$ in \eqref{AD} satisfy that $\Phi\in H^\infty$ such that $(\mathscr F\Phi)(\xi)\in \S_0^{-2}$. Let $g_{n}^{{\rm ad}}(X):= J_ng^{{\rm ad}}(J_n X)$.  Then for any   $\sigma\ge\eta>\frac{d}{2}+1$ and $X\in H^{\si}$,
\begin{equation}
\|g^{{\rm ad}}(X)\|_{H^{\si-1}}+\sup_{n\ge 1} \|g_n^{{\rm ad}}(X)\|_{H^{\si-1}}
\lesssim \left(\|X\|_{W^{1,\infty}}+\|\BB X\|_{W^{1,\infty}}\right)\|X\|_{H^{\si}},\label{g-ad Hs}
\end{equation}
\begin{align} \label{g-ad y} 
&\sup_{n\ge 1}\bigg\{ \Big|\big\< g_{n}^{{\rm ad}}(X),X\big\>_{H^{\si}}\Big|+  
\Big|\big\< J_n g^{{\rm ad}}(X),J_n X\big\>_{H^{\si}}\Big|\bigg\} \notag\\
\lesssim \, &\left(\|X\|_{W^{1,\infty}}+\|\BB X\|_{W^{1,\infty}}\right)\|X\|^2_{H^{\si}},
\end{align}
and
for any $\zeta\in(\frac{1}{2},\si-1)$, 
there is $\lambda:\mathbb N\times\mathbb N\to (0,\infty)$ with $\lambda_{n,l}\to 0$ as $n,l\to\infty$ such that for any $n,l\ge 1$, 
\begin{align}
&\left|\Big\<g_{n}^{{\rm ad}}(X)-g_l^{{\rm ad}}(v),X-Y\Big\>_{H^{\zeta}} \right| \notag\\
\leq \,&\big(1+\|X\|^4_{H^{\si}}+\|Y\|^4_{H^{\si}}\big)
\left(\lambda_{n,l}+\|X-Y\|^2_{H^{\zeta}}\right),\ \  X,\, Y\in H^{\si}. \label{g-ad-m-n x-y}
\end{align}
\end{Lemma}

\begin{proof}   By Lemma \ref{LOP2}, $\BB_i=\pp_i[\Phi\star]\in\OP\S^{-1}_0$, $1\le i\le d$ provided $(\mathscr F\Phi)(\xi)\in \S_0^{-2}$.  Then, \eqref{g-ad Hs} comes from Lemma \ref{LOPN} and the fact $H^{\eta-1}\hookrightarrow L^{\infty}$.

Concerning \eqref{g-ad y}, we only prove the upper bound estimate for $\left\< J_n g^{{\rm ad}}(X),J_n X\right\>_{H^{\si}}$, since the estimate for $\big\< g_{n}^{{\rm ad}}(X),X\big\>_{H^{\si}}$ can be derived similarly. Recalling that
$H^{s} \hookrightarrow W^{1,\infty}$ for $s>1+\frac d 2$, 
by Lemmas \ref{LOPN}, \ref{LOP2}, \ref{LOP4}, \ref{LJN}, and integrating by parts, we obtain 
\begin{align*}
&\left|\Big\<\D^{\si} J_n g^{{\rm ad}}(X),\D^{\si} J_n X\Big\>_{L^2}\right|\nonumber\\
\lesssim 
&\left|\big\<\left[\D^{\si},(\BB X\cdot\nn)\right]X,\D^{\si}J^2_n X\big\>_{L^2}\right|
+
\left|\big\<[J_n,(\BB X\cdot\nabla)]\D^{\si}X, \D^{\si}J_n X\big\>_{L^2}\right|\\
&+\left|\big\<\BB X\cdot\nabla\D^{\si}J_n X, \D^{\si}J_n X\big\>_{L^2}\right|
+ \left|\big\<\D^{\si}J_n \left[X\nn\cdot \BB X\right], \D^{\si}J_n X\big\>_{L^2} \right|\\
\lesssim\, & \|\BB X\|_{H^{\si}}\|\nabla X\|_{L^\infty}\| X\|_{H^{\si}}
+\|\BB X\|_{W^{1,\infty}}\|X\|^2_{H^{\si}}
+\|X\nn\cdot \BB X\|_{H^{\si}}\|X\|_{H^{\si}}\nonumber\\
\lesssim\, & \left(\|X\|_{W^{1,\infty}}+\|\BB X\|_{W^{1,\infty}}\right)
\|X\|^2_{H^{\si}}\notag\\
&+\left(\|X\|_{L^\infty}\|\nn\cdot \BB X\|_{H^{\si}}
+\|X\|_{H^{\si}}\|\nn\cdot \BB X\|_{L^\infty}\right)\|X\|_{H^{\si}}\nonumber\\
\lesssim\, & \left(\|X\|_{W^{1,\infty}}+\|\BB X\|_{W^{1,\infty}}\right)
\|X\|^2_{H^{\si}}\lesssim \|X\|_{H^{\eta}}\|X\|^2_{H^{\si}},
\end{align*} 
so that the desired estimate holds. 

Let $Z=X-Y$. By $H^{\si}\hookrightarrow W^{1,\infty}$, integration by parts, Lemmas \ref{LOP2} and \ref{LOP5}, we derive
\begin{align*}
&\big\<g^{{\rm ad}}(X)-g^{{\rm ad}}(Y),X-Y\big\>_{H^{\si-1}} \nonumber\\
\lesssim \,&\left|\big\<\nn\cdot (X\BB Z),Z\big\>_{H^{\si-1}} \right|
+\left|\big\<\nn\cdot (Z \BB Y),Z\big\>_{H^{\si-1}} \right|\nonumber\\
\lesssim\,&\left|\big\<\nn\cdot (X\BB Z), Z\big\>_{H^{\si-1}} \right|
+\left|\big\<\BB Y\cdot \nn)Z,Z\big\>_{H^{\si-1}} \right|
+\left|\big\<Z \nn\cdot\BB Y,Z\big\>_{H^{\si-1}} \right|\nonumber\\
\lesssim\,&\left|\big\<\nn\cdot(X\BB Z), Z\big\>_{H^{\si-1}}\right|
+\left| \big\<[\D^{\si-1},(\BB Y\cdot \nn)]Z, \D^{\si-1} Z\big\>_{L^2}\right|\nonumber\\
&+\left| \big\<\BB Y\cdot \nn \D^{\si-1} Z,\D^{\si-1} Z\big\>_{L^2}\right|
+\left| \big\<Z\nn\cdot \BB Y, Z\big\>_{H^{\si-1}}\right|\nonumber\\
\lesssim\,&\|X\|_{H^{\si}}\|\BB Z\|_{H^{\si}}\|Z\|_{H^{\si-1}}
+\|\BB Y\|_{H^{\si}}\|Z\|^2_{H^{\si-1}}
+\|\nn\cdot \BB Y\|_{L^{\infty}}\|Z\|^2_{H^{\si-1}}\nonumber\\
\lesssim\,& \left(\|X\|_{H^{\si}}+\|Y\|_{H^{\si}}\right)\|X-Y\|^2_{H^{\si-1}}.
\end{align*} 
With this at hand, as in the proof of \eqref{g-mhd-m-n x-y}, we can verify \eqref{g-ad-m-n x-y} with using Lemmas \ref{LJN} and \ref{LOP2}.
The details are omitted for brevity.
\end{proof}

Recall the dissipation term $\nu\Lambda^{2\beta}$ in \eqref{AD}, and let
$\be_0:= \be\textbf{1}_{\{\nu>0\}}.$ Recall that $r_0$ is given in \eqref{tt Pi U}. As in Theorem \ref{SCH-T}, in the present case condition \ref{Ak-assum-Pi} is trivial.

\begin{Theorem}  \label{SAD-T} Let $d\ge2$, $m=1$ and $\mathbb K=\R$ or $\T$. 
Suppose that $\Phi$ in \eqref{AD} satisfies $\Phi\in H^\infty(\mathbb K^d;\R)$ and $(\mathscr F\Phi)(\xi)\in \S_0^{-2}$. 
Assume  {\rm \ref{Ak-assum}} without \ref{Ak-assum-Pi}, and  assume {\rm \ref{Assum-h-s}} for some $l\ge1$. 
Then the following assertions hold for any $s>  \frac{d}{2}+l +\max\{1,2r_0,2\be_0\}$ and
$\F_0$-measurable $H^s$-valued random variable $X(0)$.
\begin{enumerate}[label={ $(\arabic*)$}]
\item \eqref{SAD} admits a unique maximal solution $(X,\tau^*)$ in the sense of Definition \ref{solution definition} for $\H=\H^s:= H^s(\mathbb K^d;\R)$ and $\M=\M^\theta:=  H^{\theta}(\mathbb K^d;\R)$, where $\theta\in \left(\frac d 2+l, s- \max\{1,2r_0,2\be_0\}\right).$  Furthermore,  $(X,\tau^*)$ induces a map $\H^s\ni X(0)\mapsto X(t)\in C([0,\tau^*);\H^s)$ $\p$-{\rm a.s.},
where $\tau^*$ is independent of $s$, and 
\begin{equation*}
\limsup_{t\rightarrow \tau^*}\|X(t)\|_{W^{l,\infty}}+\|\BB X\|_{W^{1,\infty}}=\infty \ \text{on}\ \{\tau^*<\infty\}.
\end{equation*}

\item The solution is non-explosive, if $\Psi(T,X,\theta)$ given by \eqref{NE h-k} enjoys
\begin{equation*}
\limsup_{\|X\|_{H^{\theta}} \to \infty}
\frac{\Psi(T,X,\theta)}{\left(\|X\|_{W^{1,\infty}}+\|\BB X\|_{W^{1,\infty}}\right)\|X\|_{H^{\theta}}^2} 
<-1,\ \ T\in(0,\infty). 
\end{equation*}
\end{enumerate}
\end{Theorem}
\begin{proof}
In the same way as we prove Theorem \ref{SMHD-T}, with  Lemma \ref{g-ad lemma} at hand, we can verify \ref{Assum-E} and \ref{Assum-hbg}. We only remark that in this case, we take $\B(X)=\|X\|_{W^{l,\infty}}+\|\BB X\|_{W^{1,\infty}}$ and $\EE=\EE^{{\rm ad}}=\nu\Lambda^{2\beta}$.
\end{proof}

\subsection{Stochastic \textbf{SQG} equation}

Let $d=2,m=1$ and $s>2$. Recall the operator $\Pi_0$ defined by \eqref{Pi-0 define}. 
As before, 
\begin{equation*}
\tt H^s:=\tt \Pi H^s(\mathbb K^2;\R),\ \ \tt\Pi=\I \ \ \text{if}\ \   \mathbb K=\R\ \ \text{and}\ \ \tt\Pi=\Pi_0\ \ \text{if}\ \ \mathbb K=\T.
\end{equation*}
When $\mathbb K=\T$, we recall \eqref{Pi-0 Hs} and notice that
\begin{equation}\label{Ds=Ls Pi0}
\left\<\D^s\Pi_0X,\D^s\Pi_0Y\right\>_{L^s}\simeq
\left\<\Lambda^{s}\Pi_0X,\Lambda^{s}\Pi_0Y\right\>_{L^s},\ \  X,\, Y\in \tt H^s(\T^2,\R)= H_0^s(\T^2,\R),
\end{equation}
where $\D^s$ and $\Lambda^{s}$ are defined in \eqref{Ds Lambda s define}. 
By
\eqref{tt X Hs} and \eqref{Ds=Ls Pi0}, we have 
$$\left\<X, Y\right\>_{H_0^s}=\left\<X,Y\right\>_{H^s}\simeq
\left\<\Lambda^{s} X,\Lambda^{s} Y\right\>_{L^s},\ \ X,\, Y\in H_0^s(\T^2,\R),$$
and recall \eqref{SQG-g}, i.e.,
$g^{{\rm sqg}}(X): = -(\mathcal R^\perp X)\cdot \nn X.$
Then we
consider the following stochastic \textbf{SQG} equation with $t\ge0$:
\begin{align}
\d X(t)=\,&  g^{{\rm sqg}}(X(t))\d t +\sum_{k=1}^\infty \Big\{ (\tt\Pi\A_kX)(t) \circ \d W_k(t)
+\tt\Pi\tt h_k (t,X(t))\d \tt W_k(t)\Big\}.\label{SSQG}
\end{align}

\begin{Lemma}\label{g-sqg lemma}
Let $ g_n^{{\rm sqg}}(X):= J_n g^{{\rm sqg}}(J_nX)$.
Then for any $\si\ge\eta>2$ and $X\in \tt H^{\si}$, $g^{{\rm sqg}}(X)\in \tt H^{\si-1}$ and
\begin{equation}
\|g^{{\rm sqg}}(X)\|_{H^{\si-1}}+\sup_{n\ge 1} \|g_n^{{\rm sqg}}(X)\|_{H^{\si-1}}
\lesssim \left(\|X\|_{W^{1,\infty}}+\|{\mathcal R} X\|_{W^{1,\infty}}\right)\|X\|_{H^{\si}},\label{g-sqg Hs}
\end{equation}
\begin{align}
\sup_{n\ge 1}\bigg\{ \Big|\big\< g_{n}^{{\rm sqg}}(X),X\big\>_{H^{\si}}\Big|+  &
\Big|\big\< J_n g^{{\rm sqg}}(X),J_n X\big\>_{H^{\si}}\Big|\bigg\}\notag\\
\lesssim\, &\left(\|X\|_{W^{1,\infty}}+\|{\mathcal R} X\|_{W^{1,\infty}}\right) \|X\|^2_{H^{\si}}. \label{g-sqg y}
\end{align}
Furthermore, for any $\zeta\in(2,\si-1)$, 
there is $\lambda:\mathbb N\times\mathbb N\to (0,\infty)$ with $\lambda_{n,l}\to 0$ as $n,l\to\infty$ such that for any $n,l\ge 1$, 
\begin{align}
&\left|\Big\<g_{n}^{{\rm sqg}}(X)-g_l^{{\rm sqg}}(v),X-Y\Big\>_{H^{\zeta}} \right|\nonumber\\
\leq\, &\big(1+\|X\|^4_{H^{\si}}+\|Y\|^4_{H^{\si}}\big)
\left(\lambda_{n,l}+\|X-Y\|^2_{H^{\zeta}}\right),\ \  X,\, Y\in \tt H^{\si}. \label{g-sqg x-y}
\end{align}
\end{Lemma}

\begin{proof}  By
Lemma \ref{LOPN}, the boundedness of $\mathcal{R}$ in $L^2$, $[J_n,\mathcal R^\perp]=0$ and the fact $g^{{\rm sqg}}(X),g_n^{{\rm sqg}}(X)$ have zero average on $\mathbb K=\T$, we obtain
\eqref{g-sqg Hs}.

By Lemma \ref{LOP4} with $s>\si>2$ and the boundedness of $\mathcal{R}$ in $L^2$, we see that for any $\si>2$,
\begin{align*}\big|\big\< g_n^{{\rm sqg}}(X), X\big\>_{H^{\si}}\big|=\,&\Big|\Big\<J_n[J_n (\mathcal{R}^\perp X) \cdot\nn J_n X)],X\Big\>_{\tt H^{\si}}\Big| \\
\lesssim\, &\left(\|X\|_{W^{1,\infty}}+\|\mathcal{R}  X\|_{W^{1,\infty}}\right)\|X\|^2_{H^\si},\ \ X\in\tt H^\si,\, n\ge 1.
\end{align*}
Similarly, by Lemmas \ref{LOP4} and \ref{LJN}, integration by parts and $\nn\cdot (\mathcal{R}^\perp X)=0$, we have that for $\mathbb K=\T$,
\begin{align*}
& \big|\big\<J_ng^{{\rm sqg}}(X),J_nX\big\>_{H^\si}\big|\\
\lesssim\, &  \Big|\big\<[\Lambda^{\si},(\mathcal{R} X\cdot\nn)] X), J^2_n\Lambda^{\si}X\big\>_{L^2}\Big|+ \Big|\big\<[J_n,(\mathcal{R}^\perp X\cdot\nn)]\Lambda^{\si}X,J_n\Lambda^{\si}X\big\>_{L^2}\Big|\\
&+\Big|\big\<(\mathcal{R}^\perp X\cdot\nn)J_n\Lambda^{\si}X,J_n\Lambda^{\si}X\big\>_{L^2}\Big|\\
\lesssim \, &
\left(\|X\|_{W^{1,\infty}}+\|\mathcal{R} X\|_{W^{1,\infty}}\right)
\|X\|^2_{H^\si},\ \ \ X\in \tt H^\si,\, n\ge 1.
\end{align*} 
Hence \eqref{g-sqg y} holds true for $\big\<J_ng^{{\rm sqg}}(X),J_nX\big\>_{H^\si}$. The estimate for $\big\< g_{n}^{{\rm sqg}}(X),X\big\>_{H^{\si}}$, and the estimated for the case $\mathbb K=\R$ can be obtained in the same way.
Finally, proceeding as in the proof of \eqref{g-mhd-m-n x-y}, we can verify \eqref{g-sqg x-y}. We omit the details.
\end{proof} 

Since $m=1$ and  $\tt\Pi\in \OP \mathcal S_0^0$, where   $\tt\Pi=\I$ for $\mathbb K=\R$ and $\tt \Pi=\Pi_0$ for $\mathbb K=\T$,
the condition \ref{Ak-assum-Pi} is trivial.  

\begin{Theorem}\label{SSQG-T} Let $d=1$, $m=1$ and $\mathbb K=\R$ or $\T$. Assume {\rm \ref{Ak-assum}} without  \ref{Ak-assum-Pi},  and assume {\rm \ref{Assum-h-s}} for  some $l\ge1$. 
Then for any $s> 1+l +\max\{1,2r_0\}$, where $r_0$ is given in \eqref{tt Pi U},  and any $\F_0$-measurable $\H$-valued random variable $X(0)$, we have the following assertions.
\begin{enumerate}[label={ $(\arabic*)$}]
\item \eqref{SSQG} has a unique maximal solution $(X,\tau^*)$  such that Definition \ref{solution definition} is fulfilled with
\begin{equation*}
\H=\H^s:=
\begin{cases}
H^s(\mathbb R^2; \mathbb R)&\text{if}\ \ \mathbb K=\R,\\
H_0^s(\mathbb T^2; \mathbb R)&\text{if}\ \ \mathbb K=\T,
\end{cases}\ \
\M=\M^\theta:=
\begin{cases}
H^{\theta}(\mathbb R^2; \mathbb R)&\text{if}\ \  \mathbb K=\R,\\
H_0^{\theta}(\mathbb T^2; \mathbb R)&\text{if}\ \ \mathbb K=\T,
\end{cases}
\end{equation*}
where $H_0^{s}(\T^2;\R)$ is defined in \eqref{Pi-0 Hs} and $\theta\in \left(l+\frac{d}{2},s-\max\{1,2r_0\}\right).$    
Besides, $(X,\tau^*)$ defines a map $\H^s\ni X(0)\mapsto X(t)\in C([0,\tau^*);\H^s)$ $\p$-{\rm a.s.},
where $\tau^*$ does not depend on $s$, and 
\begin{equation*}
\limsup_{t\rightarrow \tau^*}\|X(t)\|_{W^{l,\infty}}+\|\mathcal R X\|_{W^{1,\infty}}=\infty \ \text{on}\ \{\tau^*<\infty\}.
\end{equation*}

\item The solution is non-explosive, if $\Psi(T,X,\theta)$ defined in \eqref{NE h-k} satisfies    
\begin{equation*}
\limsup_{\|X\|_{H^{\theta}} \to \infty}
\frac{\Psi(T,X,\theta)}{\left(\|X\|_{W^{1,\infty}}+\|\mathcal R X\|_{W^{1,\infty}}\right)\|X\|_{H^{\theta}}^2} 
<-1,\ \ T\in(0,\infty).
\end{equation*}
\end{enumerate}
\end{Theorem}

\begin{proof}
With Lemma \ref{g-sqg lemma}, 
one can prove this theorem in a way analogous to Theorem \ref{SMHD-T} with noticing  that $\B(X)=\|X\|_{W^{l,\infty}}+\|\mathcal R X\|_{W^{1,\infty}}$ in the current situation.
\end{proof}

\subsection{Further examples}\label{Section:2 example}
In the above cases,  we take  $q_0= 1$ and $\theta>\frac{d}{2}+l.$ However, our general  result Theorem \ref{T3.1} (with stronger assumption \ref{Assum-h-s} replacing \ref{Assum-h}) applies to many other models.   Below  we present two  more examples where   $\theta$ and $q_0$ have to be larger.

\textit{Modified Camassa-Holm} \textbf{(MCH)} {\it equation}. Consider the modified CH equation
\begin{equation}\label{GCH k}
\left\{\begin{aligned}
&\frac{\d}{\d t}Y+2(\pp X)Y+X\pp Y:=\,0,\ \ 
X=X(t,x):[0,\infty)\times\mathbb K\to\R,\\
&Y:=\,(\I-\pp^2)^{p}X,\ p\in \N,
\end{aligned}\right.
\end{equation} 
which can be viewed as the Euler-Poincar\'{e} differential system on the Bott-Virasoro group with respect to the $H^k$-metric. For simplicity, we consider $p=2$ and we refer to \cite{Tang-Liu-2015-ZAMP} and the references therein for the background of this equation. Then \eqref{GCH k} can be reformulated as
\begin{equation*}
\frac{\d}{\d t}X=g^{{\rm mch}}(X(t))+b^{{\rm mch}}(X(t)),
\end{equation*}
where $g^{{\rm mch}}(X)=-X\pp X$ and
\begin{equation*} 
b^{{\rm mch}}(X):= -\pp \D^{-4}(X^2)-2\pp \D^{-4}\big[(\pp X)^2\big]+\frac 7 2 \pp\D^{-4} \big[(\pp^2 X)^2\big]+3 \pp \D^{-4}\big(\pp[X\pp^3 X]\big).  
\end{equation*}
Now we consider the stochastic \textbf{MCH} equation:
\begin{align*}
\d X(t)=\,&  \Big(g^{{\rm mch}}(X(t))+b^{{\rm mch}}(X)\Big)\d t\\
&+\sum_{k=1}^\infty \Big\{ (\A_kX)(t) \circ \d W_k(t)
+\tt h_k (t,X(t))\d \tt W_k(t)\Big\},\ \ t\ge0.
\end{align*} 
When $s>7/2$, we have for $X,Y\in H^s$ that, cf. \cite{Tang-Liu-2015-ZAMP},
\begin{align*}
\|b^{{\rm mch}}(X)\|_{H^{s}}
\lesssim \,& \|X\|_{W^{3,\infty}} \|X\|_{H^{s}},\ \ s>7/2,\\
\|b^{{\rm mch}}(X)-b^{{\rm mch}}(Y)\|_{H^{s}}
\lesssim\, &\left(\|X\|_{H^{s}}+\|Y\|_{H^{s}}\right)\|X-Y\|_{H^{s}},\ \ s>7/2.
\end{align*}
The above estimates mean that $b^{{\rm mch}}$ is locally Lipschitz in $H^s$ for  $s>7/2$. Notice that $g^{{\rm mch}}(X(t))=g^{{\rm ch}}(X(t))$. Then
the corresponding results for \eqref{SMHD} can be stated as Theorem \ref{SCH-T} for all
$s> \frac{1}{2}+\max\{3,l\} +\max\{1,2r_0\}$ ($\theta\in (\frac{1}{2}+\max\{3,l\}, s- \max\{1,2r_0\})$ in this case) with blow-up criterion
\begin{equation*}
\limsup_{t\rightarrow \tau^*}\|X(t)\|_{W^{\max\{3,l\},\infty}}=\infty \ \text{on}\ \{\tau^*<\infty\}.
\end{equation*}
Global regularity criterion now reduces to 
\begin{equation*} 
\limsup_{\|X\|_{H^{\theta}} \to\infty}\frac{\Psi(T,X,\theta)}{\|X\|_{W^{3,\infty}}\|X\|^2_{H^{\theta}}} <-1,
\end{equation*}
where $\Psi(T,X,\theta)$ is in \eqref{NE h-k}.
\medskip  

\textit{Korteweg-De Vries} {\bf(KdV)} \textit{equation}: We consider the
following Korteweg-De Vries equation for $X=X(t,x):[0,\infty)\times\R\to\R$:
\begin{equation*}
\frac{\d}{\d t} X(t)= g^{{\rm kdv}}(X(t)),\ \ g^{{\rm kdv}}(X):=-X\pp X-\pp^3 X=0,\ \ 
t\ge 0.
\end{equation*}
This equation was 
introduced by Korteweg-de and Vries \cite{Korteweg-de-Vries-1895-PM} to model the motion of long, unidirectional, weakly nonlinear water waves on a
channel.  Then we consider the stochastic {\bf KdV} equation with $t\ge0$:
\begin{align}
\d X(t)=\,& g^{{\rm kdv}}(X(t)) \d t +\sum_{k=1}^\infty \Big\{ (\A_kX)(t) \circ \d W_k(t)
+\tt h_k (t,X(t))\d \tt W_k(t)\Big\}.\label{SKdV}
\end{align} 
Notice that $g^{{\rm kdv}}(X)=g^{{\rm ch}}(X)-\pp^3_{x}X$ and 
$$\left\<\pp^3f,f\right\>_{H^s}=-\left\<\D^s \pp^2f,\D^s \pp f\right\>_{L^2}=-\frac{1}{2}\int_{\R}\partial (\D^s \pp f(x))^2\d x=0,\ \ f\in H^{s+3}.$$
Therefore, one can also apply Theorem \ref{T3.1} to \eqref{SKdV} and in this case $q_0=3$. The main result on $\mathbb K=\R$ can be stated as in Theorem \ref{SCH-T} for some $s>\frac{1}{2}+l+\max\{3,2r_0\}$ (in this case $\theta\in (\frac 1 2+l, s- \max\{3,2r_0\})$) with the global regularity criterion
\begin{equation*} 
\limsup_{\|X\|_{H^{\theta}} \to\infty}\frac{\Psi(T,X,\theta)}{\|X\|_{W^{1,\infty}}\|X\|^2_{H^{\theta}}} <-1,\ \  \psi(T,X,\theta)\ \text{is given in} \ \eqref{NE h-k}.
\end{equation*}

 \setlength{\bibsep}{0.6ex}


\end{CJK}
\end{document}